\documentclass[10pt]{article}

\usepackage{amsmath}
\usepackage{amsfonts}
\usepackage{amssymb}
\usepackage{framed}
\usepackage{mathtools}
\usepackage{enumerate}
\usepackage{mathrsfs}
\usepackage[hidelinks]{hyperref}
\usepackage{enumitem}
\usepackage{mathrsfs}
\usepackage{scrextend}
\usepackage{amsthm}
\usepackage{bbold}
\usepackage{bbm}
\usepackage{thmtools}
\usepackage{thm-restate}
\usepackage{mathabx}
\usepackage{accents}
\usepackage{cancel}
\usepackage[margin=1.5in]{geometry}
\usepackage{rotating}
\usepackage{tkz-graph}
\usetikzlibrary{calc}

\def\DashedEdge[#1](#2)(#3){%
\tikzstyle{EdgeStyle}=[dashed,#1]
\Edge(#2)(#3)%
}

\usepackage{xcolor}
\hypersetup{
    colorlinks,
    linkcolor={red!50!black},
    citecolor={blue!50!black},
    urlcolor={blue!80!black}
}

\usepackage{tikz-cd}
\usepackage{comment}
\usetikzlibrary{decorations.pathmorphing}

\theoremstyle{plain}
\newtheorem{thm}{Theorem}[section]
\newtheorem{lem}[thm]{Lemma}
\newtheorem{prop}[thm]{Proposition}
\newtheorem{cor}[thm]{Corollary}

\theoremstyle{definition}
\newtheorem{defn}[thm]{Definition}

\newtheorem{notn}[thm]{Notation}

\theoremstyle{remark}
\newtheorem*{rem}{Remark}

\tikzset{
  symbol/.style={
    draw=none,
    every to/.append style={
      edge node={node [sloped, allow upside down, auto=false]{$#1$}}}
  }
}

\newcommand{\Spec}{\operatorname{Spec} }
\newcommand{\Sp}{\operatorname{Sp} }

\newcommand\restr[2]{{
	\left.\kern-\nulldelimiterspace
	#1
	\vphantom{\big|}
	\right|_{#2}
	}}

\newcommand{\ep}{\varepsilon}
\newcommand{\Aut}{\operatorname{Aut}}

\newcommand{\NL}{\textrm{NL}}

\newcommand{\codim}{\textrm{codim}}

\newcommand{\GL}{\textrm{GL}}
\newcommand{\Spa}{\operatorname{Spa}}
\newcommand{\Spf}{\operatorname{Spf}}

\newcommand{\rk}{\operatorname{rk}}
\newcommand{\Pic}{\operatorname{Pic}}
\newcommand{\im}{\operatorname{im}}

\newcommand{\dR}{\operatorname{dR}}

\tikzset{
  trim node/.default=1cm,
  trim node/.style={
    overlay,
    append after command={
      ([xshift={+#1}]\tikzlastnode.north west)
      ([xshift={+-#1}]\tikzlastnode.south east)}},
  down and trim/.default=1cm,
  down and trim/.style={
    yshift=-(\pgfmatrixcurrentcolumn-1)*1.5\baselineskip,
    trim node={#1}},
  downup and trim/.default=1cm,
  downup and trim/.style={
    yshift=iseven(\pgfmatrixcurrentcolumn) ? -1.5\baselineskip : 0pt,
    trim node={#1}},
  -|/.style={to path={-|(\tikztotarget)\tikztonodes}},
  |-/.style={to path={|-(\tikztotarget)\tikztonodes}},
  -| sl/.style={-|, xslant=-1},
  |- sl/.style={|-, xslant= 1},
  center picture/.style={
    trim left=(current bounding box.center),
    trim right=(current bounding box.center)}}
    
\tikzset{
    labl/.style={anchor=south, rotate=90, inner sep=.5mm}
}

\usepackage{amsmath,calligra,mathrsfs}

\usepackage[cal=boondoxo]{mathalfa}

\DeclareMathOperator{\sheafhom}{\mathcal{H \kern -1pt o \kern -2pt m}}
\DeclareMathOperator{\sheafend}{\mathcal{E \kern -1pt n \kern -2pt d}}
\DeclareMathOperator{\sheafaut}{\mathcal{A \kern -1pt u \kern -2pt t}}

\title{Galois Orbit Bounds for Surface Degenerations}
\author{David Urbanik}

\begin{document}

\maketitle

\begin{abstract}
Given a smooth proper family $g : X \to S$ of surfaces over a number field $K \subset \mathbb{C}$, with $S$ an irreducible curve and $\eta \in S$ its generic point, we consider the general problem of constraining the locus $\NL(S) \subset S(\overline{K})$ of points $s$ where the Picard rank $\rk \Pic(X_{s})$ is larger than the generic Picard rank $\rk \Pic(X_{\overline{\eta}})$. Assuming that the local system $\mathbb{V} = R^{2} g_{*} \mathbb{Z}$ admits a non-trivial monodromy logarithm $N$ at infinity, we give a general condition under which certain points of $\NL(S)$ of unexpectedly large Picard rank satisfy a ``Galois-orbit'' height bound. This leads to the following result of Zilber-Pink type:
\begin{quote}
Let $g : X \to S$ be a one-parameter family of polarized K3 surfaces admitting a non-trivial limit mixed Hodge structure and such that $S(\mathbb{C})$ contains a Hodge-generic point. Then the locus in $S(\mathbb{C})$ where the Picard rank jumps by $3$ or more is finite.
\end{quote}
Our arguments include a new technique for ``spreading out'' formal geometry, a study of the rigid geometry of equicharacteristic zero semistable surface degenerations, and use the model-free Hyodo-Kato theory of Colmez-Nizio\l. 
\end{abstract}

\tableofcontents

\subsection{Motivation}

\subsubsection{Heights and Galois Orbits}
\label{heightsandgalsec}

Suppose that $g : X \to S$ is a smooth proper family of algebraic varieties over a number field $K \subset \mathbb{C}$, and with $S$ a curve. One often wants to study the points $s \in S(\overline{\mathbb{Q}})$ where the fibre $X_{s}$ admits more algebraic structure than the geometric generic fibre $X_{\overline{\eta}}$; depending on the setting this could mean that $X_{s}$ has in some sense more symmetries, endomorphisms, embedded subvarieties, line bundles, or cohomologically non-trivial self-correspondences. Suppose that $\mathcal{S} \subset S(\overline{\mathbb{Q}})$ is such a set of ``special'' points.

In most cases one expects, following the Zilber-Pink conjecture, the cardinality of $\mathcal{S}$ to be finite (cf. \S\ref{relsec}). Even when one doesn't expect finiteness, one expects the arithmetic complexity of the points $s \in \mathcal{S}$ to grow. A common way of making this expectation precise is through ``Galois-orbit lower bounds''; i.e., an inequality of the shape
\begin{equation}
\label{gengaloisorbitlowbound}
\theta(s) \leq a [K(s) : K]^{b}; \hspace{2em} \forall s \in \mathcal{S}
\end{equation}
where $a, b > 0$ are two real constants, and $\theta$ is a (often logarithmic) Weil height.

Such inequalities (\ref{gengaloisorbitlowbound}) are important for multiple reasons. First, one deduces immediately from the Northcott property that any subset of points $s \in \mathcal{S}$ for which $[K(s) : K]$ is absolutely bounded is finite. This in particular implies that, in situations where $\mathcal{S}$ is infinite, the field extension degrees $[K(s) : K]$ must become arbitrarily large. But perhaps most importantly, inequalities of the shape (\ref{gengaloisorbitlowbound}) are a crucial part of the so-called Pila-Zannier strategy for the Zilber-Pink conjecture, which predicts the precise circumstances under which $\mathcal{S}$ should be finite; the recent papers \cite{zbMATH08109702} \cite{papas2023unlikelyintersectionstorellilocus} \cite{zbMATH07481643} \cite{zbMATH07608391} \cite{zbMATH07931113} \cite{daw2025newcaseszilberpinky13} \cite{zbMATH08109694} all deduce the \emph{finiteness} of appropriate sets $\mathcal{S}$ from inequalities of this shape.

\subsubsection{$G$-functions}

To understand why extra algebraic structure appearing in the fibre $X_{s}$ should constrain the height $\theta(s)$, a natural approach is to analyze periods. In particular, if one fixes a place $v$ of the number field $K$, one can consider $v$-analytic families $\mathbf{P}_{v}(s)$ of period matrices constructed by comparing the $v$-analytic de Rham cohomology with either the Betti cohomology (when $v$ is an infinite place) or a $p$-adic cohomology theory (when $v$ is a finite place above $p$). When $X_{s}$ has extra algebraic structure, the transcendence degree of $\mathbf{P}_{v}(s)$ often drops, and this can constrain the $v$-adic size of $s$. If one can combine such estimates for all places $v$, one can hope to constrain the height $\theta(s)$.

A tool of much recent interest for carrying out such an argument is a $G$-function-based height theorem of Bombieri, coined the ``Hasse principle for $G$-functions'' by Andr\'e in his book \cite{zbMATH00041964}. A system of $G$-functions $\mathbf{G}$ is a tuple of one-variable power series in a common parameter with coefficients defined over a number field $K$, and satisfying some differential and growth conditions. Given a system of $G$-functions $\mathbf{G}$, Bombieri's theorem produces real numbers $a, b > 0$ and inequalities of the form $\theta(s_{1}) \leq a \delta(s_{1})^{b}$, where $s_{1}$ ranges over a subset of $\overline{\mathbb{Q}}$ and $\delta(s_{1})$ is the degree of a certain ``global non-trivial relation'' on the values $\mathbf{G}(s_{1})$ which holds whenever the evaluation of $\mathbf{G}$ at $s_{1}$ makes sense.\footnote{To be more precise, one considers $s_{1} \in \overline{\mathbb{Q}}$ and asks whether there exists a $K$-algebraic polynomial which vanishes on the tuple $\mathbf{G}(s_{1})$ whenever one embeds $K(s_{1})$ into an algebraically closed completion $\mathbb{C}_{v}$ of $K$ corresponding to the place $v$, and for which $s_{1}$ lies strictly inside the $v$-adic radius of convergence of $\mathbf{G}$. All $s_{1}$ for which such a relation of degree $\delta(s_{1})$ exists then satisfy $\theta(s_{1}) \leq a \delta(s_{1})^{b}$.} In his book \cite{zbMATH00041964}, Andr\'e showed for certain degenerating abelian families that one could use Bombieri's theorem to produce inequalities of the form (\ref{gengaloisorbitlowbound}) by constructing global non-trivial relations with $\delta(s_{1})$ proportional to $[K(s_{1}) : K]$. In this case the tuple $\mathbf{G}$ is obtained by expanding a portion of the period matrices $\mathbf{P}_{v}(s)$ around a degeneration point $s_{0}$ in some compactification $\overline{S}$ of $S$; one can show such an expansion is defined by power series with coefficients in $K$ which are independent of $v$.

When carrying out Andr\'e's strategy one needs to be able to produce relations on the $v$-adic evaluations of period functions at each place $v$ of $K$. This process is the least understood at finite places, where one has to study $p$-adic cohomology comparisons in degenerating families, including at small primes where one may not have good integral models. A central objective of this paper is to explain how to produce such relations in the case where $g$ is a degenerating family of surfaces.

\subsubsection{The Role of Rigid Geometry}
\label{roleofrigidgeosec}

It turns out that producing relations on $p$-adic evaluations of $G$-functions invariably involves rigid geometry. We illustrate with a simple example. Suppose that $\overline{g} : \overline{X} \to \overline{S}$ is a semistable family of surfaces over a compactification $\overline{S}$ of $S = \overline{S} \setminus \{ s_{0} \}$, and that $Y = \overline{X}_{s_{0}}$ is a divisor with simple normal crossings and components $Y_{1}, \hdots, Y_{\ell}$. Consider a point $p$ in the intersection of three such components, say $p \in Y_{1} \cap Y_{2} \cap Y_{3}$. Then in formal coordinates, the map of germs $(\overline{X}, p) \to (\overline{S}, s_{0})$ is given by $s \mapsto f_{1} f_{2} f_{3}$, where $f_{i}$ defines $Y_{i}$ near $p$. Let $\mathcal{D}^{\wedge}$ be the formal neighbourhood of $p$ in $\overline{X}$, and $\mathcal{S}^{\wedge}$ the formal neighbourhood of $s_{0}$ in $\overline{S}$. Then from the formal coordinates we obtain an isomorphism $\chi : \Spf K[[x_{1}, x_{2}, x_{3}]] \xrightarrow{\sim} \mathcal{D}^{\wedge}$ which sends $f_{i}$ to $x_{i}$. For each place $v$ of $K$ (finite or infinite), the map $\chi$ admits an extension $\chi^{v} : \Delta^{3,v} \to \overline{X}^{v}$, where $\Delta^{3,v}$ is a $3$-dimensional open $v$-adic disk in $\mathbb{A}^{3,v}$ centred at $0$, and $\chi^{v}$ is an open embedding with image $\mathcal{D}^{v}$. For points $s_{1}$ of $S(K_{v})$ which are $v$-adically close to $s_{0}$, the intersections $X^{v}_{s_{1}} \cap \mathcal{D}^{v}$ contain a subspace $A_{s_{1}}$ which is isomorphic over a finite extension of $K_{v}$ to a product of $v$-adic annuli, and one can consider the maps\footnote{Here we are considering the overconvergent de Rham cohomology when $v$ is finite.}
\begin{equation}
\label{naiveresmaps}
H^{2}_{\textrm{dR}}(X_{s_{1}}) \to H^{2}_{\textrm{dR}}(A_{s_{1}}) \cong K_{v}; \hspace{2em} \omega \mapsto \restr{\omega}{A_{s_{1}}} .
\end{equation}
After fixing de Rham bases for $R^{2} \overline{g}_{*} \Omega^{\bullet}_{\overline{X}/\overline{S}}(\log Y)$ and $R^{2} (\restr{\overline{g}}{\mathcal{D}^{\wedge}})_{*} \Omega^{\bullet}_{\mathcal{D}^{\wedge}/\mathcal{S}^{\wedge}}(\log \mathcal{D}^{\wedge} \cap Y)$, the coordinates of the maps (\ref{naiveresmaps}) in these bases (which make sense in some $v$-adic disk around $s_{0}$) give the values at $s_{1}$ of a system $\mathbf{G}$ of G-functions. Now if $\mathcal{L}$ is a line bundle on $X_{s_{1}}$, and $v$ is a finite place, the restricted line bundle $\restr{\mathcal{L}}{A_{s_{1}}}$ is trivial over a finite extension of $K_{v}$ (in contrast to what happens in complex geometry, the Picard group of a product of closed rigid annuli is trivial). This results in a non-trivial algebraic relation on the $v$-adic evaluation $\mathbf{G}(s_{1})$ defined over a finite extension of $K$.

In general, when one considers a semistable degeneration $\overline{g}$ with simple normal crossing special fibre $Y = \overline{X}_{s_{0}}$, one would like to carry out similar arguments by replacing the point $p$ with some intersection $D$ of components of $Y$, where $D$ has positive dimension. In this case one can still try to ``thicken'' the formal neighbourhood $\mathcal{D}^{\wedge}$ of $D$ in $\overline{X}$ in the $v$-adic directions, but understanding the geometry of the intersections $X^{v}_{s_{1}} \cap \mathcal{D}^{v}$ becomes substantially more complicated. Most of the machinery we develop in this paper is ultimately motivated by controlling this geometry and its cohomology in a way that is in some sense uniform over all finite places $v$ of $K$. Although we will focus here on the case where $\overline{g}$ is a family of surfaces, we expect these techniques to apply when $\overline{g}$ has higher relative dimension as well.

\subsection{Results}
\label{mainappsec}

In what follows we write $g : X \to S$ for a smooth proper family of surfaces over a number field $K$, with $S$ a smooth geometrically connected curve, and write $\overline{g} : \overline{X} \to \overline{S}$ for a semistable extension of $g$ with $\overline{S} \setminus S = \{ s_{0} \}$. We write $\eta$ for the generic point of $S$. We fix a logarithmic Weil height $\theta : \overline{S}(\overline{K}) \to \mathbb{R}_{\geq 0}$. Given a point $s \in S$ we write $\rk \Pic(X_{s})$ for the rank of the associated Neron-Severi group. 

We consider the monodromy representation $\pi_{1}(S, s) \to H^{2}(X_{s}, \mathbb{Q})$ on middle cohomology, write $V := \Pic(X_{\overline{\eta}})^{\perp} \subset H^{2}(X_{s}, \mathbb{Q})$ for the subrepresentation which is the complement of the global Picard lattice, write $Q : V \otimes V \to \mathbb{Q}$ for the restriction of the cup product pairing, and $\Aut(V,Q)$ for the group of polarization-preserving automorphisms of $V$. We let $N$ be the logarithm of the action of monodromy around $s_{0}$; we recall that if $N$ is non-zero, the order $k$ of nilpotency of $N$ is either $2$ or $3$ (cf. \S\ref{specialfibreconstrsec}). 

\begin{thm}
\label{mainsurfaceapp1}
Suppose that $k = 3$ and representation of $\pi_{1}(S,s)$ on $V$ is $\mathbb{Q}$-simple. Then there exists $a, b > 0$ such that
\[ \theta(s) \leq a [K(s) : K]^{b} \]
for all $s \in S(\overline{K})$ such that 
\begin{equation}
\label{rankNineq1}
\rk \Pic(X_{s}) \geq \rk \Pic(X_{\overline{\eta}}) + 2.
\end{equation}
\end{thm}

\begin{thm}
\label{mainsurfaceapp2}
Suppose that $k = 2$ and the image of $\pi_{1}(S,s)$ in $\Aut(V,Q)$ is Zariski dense. Then there exists $a, b > 0$ such that
\[ \theta(s) \leq a [K(s) : K]^{b} \]
for all $s \in S(\overline{K})$ such that 
\begin{equation}
\label{rankNineq2}
\rk \Pic(X_{s}) \geq \rk \Pic(X_{\overline{\eta}}) + \dim \im(N) + 1 .
\end{equation}
\end{thm}

\begin{rem}
The proof of \autoref{mainsurfaceapp2} will show that if the pure Hodge structure on $\im(N)$ determined by the limiting mixed Hodge structure at $s_{0}$ admits a non-trivial summand $W$, one can replace $\dim \im(N)$ with $\dim W$ in (\ref{rankNineq2}). 
\end{rem}

As an immediate corollary of the Northcott property, one obtains:
\begin{cor}
In the situation of \autoref{mainsurfaceapp1} (resp. \autoref{mainsurfaceapp2}), for each $d$, there are only finitely many $s \in S(\overline{K})$ with $[K(s) : K] \leq d$ and satisfying (\ref{rankNineq1}) (resp. (\ref{rankNineq2})). \qed
\end{cor}

Once one fixes a subclass of surfaces of interest, one can obtain results without reference to the degeneration at $s_{0}$ (because such degenerations are classified for varieties in the subclass). For instance:
\begin{cor}
\label{K3cor}
Suppose that $g : X \to S$ is a $K$-algebraic family of polarized K3 surfaces over a geometrically connected curve $S = \overline{S} \setminus \{ s_{0} \}$ admitting:
\begin{itemize}
\item[-] quasi-unipotent monodromy at $s_{0}$ with a monodromy logarithm $N$ of nilpotency order $k > 1$; and
\item[-] a Hodge generic point $s \in S(\mathbb{C})$.
\end{itemize}
Then there exists $a, b > 0$ such that
\[ \theta(s) \leq a [K(s) : K]^{b} \]
for all $s \in S(\overline{K})$ such that 
\begin{equation}
\label{K3rankineq}
\rk \Pic(X_{s}) \geq \rk \Pic(X_{\overline{\eta}}) + 5-k .
\end{equation}
\end{cor}

\vspace{0.5em}

\noindent We emphasize again that $k$ is either $2$ or $3$, hence correspondingly $5-k$ is either $3$ or $2$. The term ``Hodge generic point'' in \autoref{K3cor} should be interpreted as follows: there is a period map $\varphi : S \to M$, where $M$ is the moduli space of K3 surfaces of appropriate polarization degree, and the image of $\varphi$ contains a point outside of the tensorial Hodge locus of $M$.

\begin{rem}
We note that the moduli space of quasi-polarized K3 surfaces with polarization degree $2\ell$ admits toroidal compactifications with boundary of codimension one, and any curve in such a moduli space passing through a Hodge generic point and intersecting the boundary will satisfy the hypotheses of \autoref{K3cor}. 
\end{rem}

\begin{proof}
The assumptions are unchanged after replacing $S$ with a finite base-change $S' \to S$ and $X$ with $X' = X \times_{S} S'$; in particular we may apply a reduction theorem of Mumford \cite[Ch II]{zbMATH03425769} to reduce to the situation where $g$ admits a semistable compactification $\overline{g} : \overline{X} \to \overline{S}$. The monodromy assumptions of \autoref{mainsurfaceapp1} and \autoref{mainsurfaceapp2} follow from the assumption that there exists a Hodge generic point and the Andr\'e-Deligne monodromy theorem \cite[Thm. 1]{Andre1992}. Thus to apply \autoref{mainsurfaceapp1} and \autoref{mainsurfaceapp2} it suffices to prove that $\dim \im(N) \leq 2$ in the case $k = 2$, which follows from \cite[Prop. 4.1]{zbMATH07286305}.
\end{proof}

In this case one can even run the full Pila-Zannier strategy to obtain finiteness results without bounding the degree of the number field.

\begin{thm}
\label{truefinitenessforK3s}
Suppose that $g : X \to S = \overline{S} \setminus \{ s_{0} \}$ is a one-parameter family of polarized K3 surfaces over $\mathbb{C}$ for which:
\begin{itemize}
\item[-] the monodromy around $s_{0}$ is quasi-unipotent and has a monodromy logarithm $N$ with nilpotency order $k > 1$; and
\item[-] a Hodge generic point $s \in S(\mathbb{C})$.
\end{itemize} 
Then the locus
\begin{equation}
\label{truefinitenessset}
 \{ s \in S(\mathbb{C}) : \rk \Pic(X_{s}) \geq \rk \Pic(X_{\overline{\eta}}) + 5-k \} 
\end{equation}
is finite.
\end{thm}

\subsection{Techniques}

As discussed in \S\ref{roleofrigidgeosec}, our approach crucially relies on rigid geometry in order to study $v$-adic relations on a system $\mathbf{G}$ of G-functions at each finite place $v$ of $K$. In particular given a degenerating family $g : \overline{X} \to \overline{S}$ with special fibre $Y = \overline{X}_{s_{0}}$ at $s_{0}$, one would like to produce certain rigid-analytic varieties $\mathcal{D}^{v}$ which embed into $\overline{X}^{v}$, and which are a ``spreading out'' of some piece of the geometry of $Y$ into the nearby fibres. If $s$ is the uniformizing parameter at $s_{0}$ which is associated to $\mathbf{G}$, it is crucial when doing this that the image of $\mathcal{D}^{v}$ in $\overline{X}^{v}$ makes sense over the component containing $s_{0}$ of the $v$-adic region $|s| < 1$ away from finitely many primes, and over a disk around $s_{0}$ of controlled radius at the remaining finite primes. In particular, one should have some formal object $\mathcal{D}^{\wedge}$ defined over the number field $K$ which interpolates the different $\mathcal{D}^{v}$'s for different $v$, and such that the power series that describe $\mathcal{D}^{\wedge}$ converge in the prescribed regions.

To make all of this precise we will develop a kind of ``adelic geometry'' in \S\ref{adelicgeosec}. This basically means that, given a finite-type $K$-algebra $A$, we study families $\{ \| \cdot \|_{v} \}$ of norms on $A$ that come from a common presentation of $A$. When one replaces $A$ with some $v$-adic completion, the norm $\| \cdot \|_{v}$ extends to a norm on the resulting Tate algebra. One then carries out calculations involving power series with coefficients in $A$ and uses these norms and their properties to characterize how the geometry of the associated rigid varieties behave. We develop our adelic geometry in greater generality than required for studying surface degenerations because we plan future applications of this theory in forthcoming work. The most important consequence of this theory is the following (for definitions and a more detailed statement, see \S\ref{adelictubesec} and \autoref{tubeconstrprop}).

\begin{thm}
\label{mainadelictubethm}
Let $K$ be a number field, $U$ an affine algebraic $K$-variety with \'etale coordinates $\overline{f} = (f_{1}, \hdots, f_{q})$, and let $E = V(f_{1}, \hdots, f_{p})$ for some $p \leq q$. Then $\overline{f}$ induces an isomorphism $\eta : E[p]^{\wedge} \xrightarrow{\sim} \mathcal{E}$ between the formal neighbourhoods of $E$ in $E \times \mathbb{A}^{p}$ and $U$, respectively. Moreover there exists a standard adelic tube $E\langle p \rangle$ and an adelic tube $\mathcal{T}$ associated to $(E, U)$, together with open embeddings $\eta^{v} : E\langle p \rangle^{v} \hookrightarrow \mathcal{T}^{v}$ of adic spaces for each $v \in \Sigma_{K,f}$, whose formal completions are identified with the corresponding rigid-analytification of $\eta$. The images of the maps $\eta^{v}$ contain the $v$-adic analytifications of an adelic tube associated to $(E, U)$ which refines $\mathcal{T}$. 
\end{thm}

\vspace{0.5em}

A second input is the model-free Hyodo-Kato cohomology of Colmez-Nizio\l. The point is that after constructing the spaces $\mathcal{D}^{v}$ one wants to understand what effect the subspace $\mathcal{D}^{v}_{s_{1}} := \mathcal{D}^{v} \cap X_{s_{1}}$ has on line bundles $\mathcal{L}$ on $X_{s_{1}}$; for instance, one would like to know if the restriction $\restr{\mathcal{L}}{\mathcal{D}^{v}_{s_{1}}}$ is trivial in de Rham cohomology. To avoid studying the Picard group of $\mathcal{D}^{v}_{s_{1}}$ directly, one can instead use a $p$-adic cohomology theory which compares with the de Rham cohomology and constrain the cohomological incarnation of $\mathcal{L}$ using weights. Since we don't necessarily have good integral models for the family $\overline{g}$, its fibres, and the spaces $\mathcal{D}^{v}$, one would like to use a suitable $p$-adic cohomology theory which is ``model-free''; for this we use Colmez-Nizio\l{}'s Hyodo-Kato cohomology.

\subsection{Relation to Other Work}
\label{relsec}

The results of this paper belong broadly to the area of mathematics known as unlikely intersection theory. In this area, one tries to understand special subvarieties of a fixed algebraic variety which are ``unlikely'' from the perspective of an expected dimension computation. This subject includes a broad family of conjectures --- some now resolved and many under active study --- including the Manin-Mumford conjecture \cite{zbMATH03882517} \cite{zbMATH03899065} \cite{zbMATH03929174}, the Andr\'e-Oort conjecture \cite{zbMATH07272616} \cite{pila2024canonicalheightsshimuravarieties}, and the Zilber-Pink conjectures for algebraic groups and Shimura varieties \cite{zbMATH01818752} \cite{zbMATH05013101}. Recently, the subject has been recast in the language of Hodge theory, and such problems are now understood as special cases of Hodge-theoretic Zilber-Pink conjectures; see \cite{atyp} \cite{zbMATH07790853} for the pure case and \cite{arXiv:2406.16628} for the mixed one. In our setting from \S\ref{heightsandgalsec}, it is the Hodge-theoretic Zilber-Pink conjecture of \cite{zbMATH07790853} which governs when one expects the set $\mathcal{S} \subset S(\overline{Q})$ to be finite; in particular \autoref{truefinitenessforK3s} is a prediction of the Zilber-Pink conjecture. 

It is tempting to try to prove our results on K3 surfaces by applying the Kuga-Satake correspondence to reduce to prior work \cite{zbMATH07931113} \cite{daw2025largegaloisorbitsconjecture} on Zilber-Pink in moduli spaces of abelian varieties. However such a reduction is not sufficient for our purposes for several reasons. Firstly, both \cite{zbMATH07931113} and \cite{daw2025largegaloisorbitsconjecture} work with Hodge-generic curves in $\mathcal{A}_{g}$, and the Kuga-Satake correspondence maps moduli spaces of K3 surfaces into Shimura subvarieties of $\mathcal{A}_{g}$ of very large codimension. In the case of \cite{zbMATH07931113}, no $p$-adic input is used, which limits the endomorphism algebras that are amenable to analysis to those satisfying a strong incompatiblity condition \cite[\S4.2, Def. 4.6]{zbMATH07931113} with the degeneration. The results of \cite{daw2025largegaloisorbitsconjecture} on the other hand apply only in the case of completely multiplicative reduction, which is needed to apply a certain relative version of Raynaud uniformization that interpolates all places of the number field $K$. Our adelic geometry is intended as a broad replacement for this kind of uniformization which works in contexts beyond abelian varieties, and it is why we are able to handle more general degeneration types. 

We expect our methods will apply to other variations of Hodge structure of K3 type of geometric origin, and should give essentially the same results in those cases. We have not carried out such a general analysis in this paper in the interest of keeping the article shorter and more accessible.


\subsection{Conventions}

Given a number field $K$ we write $\Sigma_{K,f}$ (resp. $\Sigma_{K,\infty}$) for its set of finite (resp. infinite) places, and set $\Sigma_{K} = \Sigma_{K,f} \sqcup \Sigma_{K,\infty}$. For a place $v$ of $K$ we write $K_{v}$ for the completion of $K$ at $v$. Given a $K$-algebraic variety $X$ and a finite place $v \in \Sigma_{K,f}$, we write $X^{v}$ for the associated adic space over $\Spa(K_{v}, \mathcal{O}_{K_{v}})$, which we will identify with the associated rigid space. Given a place $v \in \Sigma_{K,\infty}$ we write $X^{v}$ for the complex analytic space obtained as the analytification of $X_{K_{v}}$. 

When we deal with a fixed number field $K$, we will always assume it comes with a fixed embedding $K \subset \overline{\mathbb{Q}} \subset \mathbb{C}$. 

\section{Adelic Geometry}
\label{adelicgeosec}

In this section we develop the rudiments of a theory of ``adelic geometry'' which will be important for analyzing the rigid-geometry of degenerating families of algebraic varieties. We develop this theory in greater generality than will be necessary with a view towards future applications.

\subsection{Polynomial Algebra Norms}
\label{polyalgnormsec}

In what follows we let $K$ be a number field, and $v$ a finite place of $K$. Let $\| \sum_{J} a_{J} \overline{x}^{J} \|_{v} := \textrm{max}_{J} |a_{J}|_{v}$ be the usual induced norm on $A := K[x_{1}, \hdots, x_{n}]$, where we use the symbol $J$ to denote a multi-index. Let $I \subset A$ be an ideal such that $\Spec A/I$ is a geometrically irreducible and geometrically reduced $K$-variety passing through the point $0 \in \mathbb{A}^{n}$. We set $X = \Spec A/I$. 

\begin{lem}
\label{compinj}
Let $\hat{A} =  K_{v}\langle x_{1}, \hdots, x_{n} \rangle$ be the standard Tate algebra, and set $\hat{I} = \hat{A} I$. Then the map $A/I \to \hat{A}/\hat{I}$ is injective.
\end{lem}

\begin{proof}
We first claim we may replace $K$ with $K_{v}$: indeed, the base-change $X_{K_{v}} \to X$ is faithfully flat, and hence the map $A/I \to A_{K_{v}} / A_{K_{v}} I$ is injective. 

The ring $\hat{A}/\hat{I}$ can be seen to be ring of global functions on $X^{v} \cap \mathbb{B}$, where $\mathbb{B} \subset \mathbb{A}^{n,v}$ is the closed affinoid ball of radius $1$. Let $\mathfrak{m}_{0} = (x_{1}, \hdots, x_{n})$ be the ideal corresponding to the $K_{v}$-point $0$. Then we obtain a commuting diagram
\begin{equation}
\label{rigidstalkdiag}
\begin{tikzcd}
A/I \arrow[d] \arrow[r] & \mathcal{O}_{X,\mathfrak{m}_{0}} \arrow[d] \arrow[r] & \left[\mathcal{O}_{X,\mathfrak{m}_{0}}\right]^{\wedge}_{\mathfrak{m}_{0}} \arrow[d, "\sim" labl] \\
\widehat{A}/\widehat{I} \arrow[r] & \mathcal{O}_{X^{v},\mathfrak{m}_{0}} \arrow[r] & \left[\mathcal{O}_{X^{v},\mathfrak{m}_{0}}\right]^{\wedge}_{\mathfrak{m}_{0}} ,
\end{tikzcd}
\end{equation}
where the arrows are the natural ones. It is immediate from the construction of the analytification that the map on the right is an isomorphism, so the injectivity of the left arrow reduces to the injectivity of the top row. This follows from the fact that $X$ is (geometrically) reduced and irreducible (for the first arrow) and \cite[\href{https://stacks.math.columbia.edu/tag/00MC}{Lemma 00MC}]{stacks-project} for the second.
\end{proof}

\begin{defn}
\label{normdef}
Given a field $L$ with a non-archimedean absolute value $| \cdot |$ and a $L$-algebra $V$, a non-archimedian norm extending $| \cdot |$ is a map $\| \cdot \| : V \to \mathbb{R}_{\geq 0}$ with the following properties for all $u, w \in V$ and $a \in L$:
\begin{itemize}
\item[(1)] $\| u \| = 0 \iff u = 0$;
\item[(2)] $\| u + w \| \leq \max \{ \| u \|, \| w \|\}$;
\item[(3)] $\| a u \| = |a| \| u \|$; and
\item[(4)] $\| u \cdot w \| \leq \| u \| \| w \|$.
\end{itemize}
\end{defn}

For the Gauss norm $\| \cdot \|_{v}$, one has equality in \autoref{normdef}(4), but this need not hold in general.

\begin{defn}
\label{equivnormdef}
We say two norms $\| \cdot \|_{1}, \| \cdot \|_{2} : V \to \mathbb{R}_{\geq 0}$ are equivalent if there are positive real constants $c, c' > 0$ such that
\[ c \| \cdot \|_{2} \leq  \| \cdot \|_{1} \leq c' \| \cdot \|_{2} \]
identically on $V$. 
\end{defn}

\begin{lem}
\label{normwelldef}
The map $\| \cdot \|_{(X,0),v} : A/I \to \mathbb{R}_{\geq 0}$ given by 
\begin{equation}
\label{nonarchnorm}
f + I \mapsto \min_{f' \in f + I} \| f' \|_{v} ,
\end{equation}
is a well-defined non-archimedean norm on $A/I$. The same holds for the analogously-defined norm $\| \cdot \|_{(X_{K_{v}},0)} : A_{K_{v}} / I_{K_{v}} \to \mathbb{R}_{\geq 0}$, where $I_{K_{v}} = A_{K_{v}} I$.
\end{lem}

\begin{proof}
It will suffice to prove the statement over $K$, the case over $K_{v}$ being analogous. We first check that the minimum exists. It suffices to consider the case where $f \not\in I$. Since $K$ is a number field the valuation $v$ is discrete, so the only way the minimum could fail to exist is if there is a sequence $\{ f_{i} \}^{\infty}_{i=1}$ with $\| f_{i} \|_{v} \xrightarrow{i \to \infty} 0$ and $f + I = f_{i} + I$ for all $i$; we will show this is not the case. 


 Now $\{ f - f_{i} \}_{i \in I}$ is a sequence of elements in $I$, converging to $f$. Consider the embedding $\iota : A \hookrightarrow \hat{A}$ into the Tate algebra at $0$. Then the Gauss norm on $\hat{A}$ extends $\| \cdot \|_{v}$. Setting $\hat{I} = \hat{A} I$ we then obtain a diagram of exact sequences
\begin{equation}
\label{riganalyticexactseqdiag}
\begin{tikzcd}
 0 \arrow[r] & I \arrow[d] \arrow[r] & A \arrow[r] \arrow[d, "\iota"] & A/I \arrow[d] \arrow[r] & 0 \\
 0 \arrow[r] & \hat{I} \arrow[r] & \hat{A} \arrow[r] & \hat{A}/\hat{I} \arrow[r] & 0 ,
\end{tikzcd}
\end{equation}
and applying \autoref{compinj} to deduce that the right arrow is injective we get that $\iota^{-1}(\hat{I}) = I$. Then applying \cite[I.2.3, Cor. 8]{zbMATH06255263} the ideal $\hat{I}$ is closed, so it follows that $\iota(f - f_{i}) \xrightarrow{i \to \infty} \iota(f) \in \hat{I}$. But then $f = \iota^{-1}(\iota(f)) \in \iota^{-1}(\hat{I}) = I$. 

\vspace{0.5em}

Having shown the minimum exists, (1) and (3) listed in \autoref{normdef} are easily checked. For (2) we choose $x_{u}, x_{w} \in I$ such that $\| u + x_{u} \|_{v}$ and $\| w + x_{w} \|_{v}$ are minimized. Then
\[ \| u + w \|_{(X,0),v} \leq \| u + w + (x_{u} + x_{w}) \|_{v} \leq \textrm{max}\{ \|u + x_{u}\|_{v}, \|w + x_{w}\|_{v} \} = \textrm{max} \{ \|u\|_{(X,0),v}, \|w\|_{(X,0),v} \} . \]
For (4) we choose the same $x_{u}$ and $x_{w}$ and compute that
\begin{align*}
\| u w \|_{(X,0),v} \leq \| (u + x_{u})(w + x_{w}) \|_{v} = \| u \|_{(X,0),v} \| w \|_{(X,0),v} . 
\end{align*}
\end{proof}

\begin{defn}
\label{adelicnormdef}
By an adelic norm we mean a collection $\{ \| \cdot \|_{(X,0),v} \}_{v \in \Sigma_{K,f}}$ of norms as in \autoref{normwelldef} associated to a fixed presentation $X \cong \Spec A/I$. 
\end{defn}

The notation $\| \cdot \|_{(X,0),v}$ suggests a kind of dependence of the norm only on the germ $(X,0)$, but the definition suggests a dependence on the embedding $X \hookrightarrow \mathbb{A}^{n}$. In fact there is a kind of ``adelic'' sense in which the norms $\| \cdot \|_{(X,0),v}$ depend only on the pair $(X,0)$, which will be justified by \autoref{normundercoordchange} below. 

\begin{lem}
\label{normextagrees}
Let $\iota : A/I \hookrightarrow A_{K_{v}}/I_{K_{v}}$ be the obvious embedding. Then $\| \cdot \|_{(X,0),v} = \| \iota(\cdot) \|_{(X_{K_{v}},0)}$ as functions on $A/I$.
\end{lem}

\begin{proof}
We may assume that $f \not\in I$ and let $\hat{x} \in I_{K_{v}}$ be such that $\| f \|_{(X_{K_{v}},0)} = \| f + \hat{x} \|_{v}$. Write $\hat{x} = \hat{h} x$ for $x \in I$ and $\hat{h} \in A_{K_{v}}$. Since $K$ is dense in $K_{v}$, it follows that $A$ is dense in $A_{K_{v}}$. Thus there is a sequence $\{ h_{i} \}_{i=1}^{\infty}$ of elements of $A$ converging to $\hat{h}$. Then letting $x_{i} = h_{i} x$, we have that
\[ \| x_{i} - \hat{x} \|_{v} = \| (h_{i} - \hat{h}) x \|_{v} = \| h_{i} - \hat{h} \|_{v} \| x \|_{v} \to 0 , \]
and so $f + x_{i} \to f + \hat{x}$ and thus $\| f + x_{i} \|_{v} \to \| f + \hat{x} \|_{v}$. Since the valuation is discrete and $f \not\in I$ (so $f + \hat{x} \neq 0$) the limit is obtained for some $i$. This shows that $\| \iota(\cdot) \|_{(X_{K_{v}},0)} \leq \| \cdot \|_{(X, 0),v}$; the other inequality is immediate from the definitions.
\end{proof}

\begin{lem}
\label{normextagrees2}
Let $\iota : A_{K_{v}}/I_{K_{v}} \hookrightarrow \hat{A}/\hat{I}$ be the obvious embedding, where $\hat{A}$ is the Tate algebra at zero and $\hat{I} = \hat{A} I$. Then $\| \cdot \|_{(X,0),v} = \| \iota(\cdot) \|_{(X_{K_{v}},0)}$ as functions on $A_{K_{v}}/I_{K_{v}}$.
\end{lem}

\begin{proof}
The proof is completely analogous to \autoref{normextagrees} using the density of $A_{K_{v}}$ in $\hat{A}$. 
\end{proof}

\begin{lem}
\label{oneforallbut}
For each $f + I \in A/I$ with $f \not\in I$ we have that $\| f \|_{(X,0),v} = 1$ for all but finitely many $v$.
\end{lem}

\begin{proof}
By clearing denominators, we may reduce to considering $f \in \mathcal{O}_{K}[x_{1}, \hdots, x_{n}] =: \mathcal{A}$. For each prime $\mathfrak{p}$ of $\mathcal{O}_{K}$ write $\mathcal{O}_{K,(\mathfrak{p})}$ for the stalk (not completion) at $\mathfrak{p}$, $_{(\mathfrak{p})}A$ for the ring $\mathcal{O}_{K,(\mathfrak{p})}[x_{1}, \hdots, x_{n}]$, $\mathcal{I}_{(\mathfrak{p})}$ for the intersection $I \cap {_{(\mathfrak{p})}A}$, and $\mathcal{I}$ for the intersection $I \cap \mathcal{A}$. Then if $v$ is a valuation corresponding to $\mathfrak{p}$, we may observe that a minimal representative $f + x_{\mathfrak{p}}$ for $f$ can always be chosen such that $x_{\mathfrak{p}} \in \mathcal{I}_{(\mathfrak{p})}$. 

Now
\[ \| f \|_{(X,0),v} < 1 \iff \| f + x_{\mathfrak{p}} \|_{v} < 1 \iff f = -x_{\mathfrak{p}} \, \textrm{(mod } {_{(\mathfrak{p})}A} \, \mathfrak{p} \textrm{)} , \]
so it suffices to show that $f$ reduces modulo $_{(\mathfrak{p})}A \, \mathfrak{p}$ to the reduction of $\mathcal{I}_{(\mathfrak{p})}$ for at most finitely may $\mathfrak{p}$. 

The reduction of $\mathcal{I}_{(p)}$ modulo $_{(\mathfrak{p})}A \, \mathfrak{p}$ is the same as the reduction of $\mathcal{I}$ modulo $\mathfrak{p}$, as one can see by scaling coefficients. Therefore, if $\mathfrak{p}$ is a prime with this property, then the fibre of $\Spec \mathcal{A}/(f) \to \mathcal{O}_{K}$ above $\mathfrak{p}$ contains the fibre of $\Spec \mathcal{A}/\mathcal{I} \to \mathcal{O}_{K}$ above $\mathfrak{p}$. By assumption this does not happen for the generic fibre (we assumed $f \not\in I$), so it can only happen for finitely many fibres as a consequence of \cite[\href{https://stacks.math.columbia.edu/tag/05F7}{Lemma 05F7}]{stacks-project} (consider the fibres of $\Spec \mathcal{A}/(f) \cap \Spec \mathcal{A}/\mathcal{I}$, and use geometric irreducibility of $\Spec A/I$ together with \cite[\href{https://stacks.math.columbia.edu/tag/0559}{Lemma 0559}]{stacks-project}).
\end{proof}

\begin{lem}
\label{bcinvariance}
Let $K'/K$ be a finite extension and $I' \subset K'[x_{1}, \hdots, x_{n}] =: A'$ the ideal generated by $I$. Set $X' = \Spec A'/I'$. Then for all pairs $(v', v)$, with $v$ (resp. $v'$) a place of $K$ (resp. $K'$) and $v' | v$, we have
\[ \| f \|_{(X,0),v} \leq c_{v} \| f \|_{(X',0),v'} \]
for all $f \in A/I$, where $c_{v}$ is a positive real constant depending only on $v$. For all but finitely many $v$ we may take $c_{v} = 1$, so in particular the norms $\| \cdot \|_{(X,0),v}$ and $\| \cdot \|_{(X',0),v'}$ agree on $A/I$ for all but finitely many $(v', v)$. 
\end{lem}


\begin{proof}
We may reduce to the case where $K'/K$ is Galois with $d = [K' : K]$. Suppose that $x_{1} \in I'$ is chosen such that $\| f + x_{1} \|_{v'}$ is minimized, and let $x_{1}, \hdots, x_{d}$ be a list of its Galois conjugates, where each Galois conjugate appears the same number of times. Then since the Galois action preserves the norm, we have $\| f + x_{i} \|_{v'} = \| f + x_{1} \|_{v'}$ for all $1 \leq i \leq d$. Then
\begin{align*}
\| f + (x_{1} + \cdots + x_{d})/d \|_{v} &= |1/d|_{v'} \left\| \sum_{i=1}^{d} (f + x_{i}) \right\|_{v'} \\
&\leq |1/d|_{v} \textrm{max}_{1 \leq i \leq d} \{ \| f + x_{i}\|_{v'} \} =  |1/d|_{v'} \| f + x_{1} \|_{v'} .
\end{align*}
Using that $(x_{1} + \cdots + x_{d})/d$ is defined over $K$ we conclude that $\| f \|_{(X,0),v} \leq |1/d|_{v} \| f \|_{(X',0),v'}$.
\end{proof}


\begin{prop}
\label{norminvunderrestric}
Fix $f \in A \setminus I$ with $f(0) \neq 0$. Consider the ideal $I' \subset A' := K[x_{1}, \hdots, x_{n+1}]$ given by $I' = A' I + (x'_{n+1} f - 1)$ where $x'_{n+1} = x_{n+1} + 1/f(0)$. Set $X' = \Spec A'/I'$. Then for all but finitely many places $v$ one has
\begin{equation}
\label{restnormagreeswithoriginal}
\| g \|_{(X,0),v} = \| g \|_{(\Spec A'/A'I, 0), v} \leq \| g \|_{(X',0),v}
\end{equation}
for all $g \in A/I$. In particular the norms $\| \cdot  \|_{(X,0),v}$ and $\| \iota( \cdot ) \|_{(X',0),v}$ agree on $A/I$ for all but finitely many $v$, where $\iota : A/I \hookrightarrow A'/I'$ is the natural embedding.
\end{prop}

The purpose of using the function $x'_{n+1}$ instead of $x_{n+1}$ in the statement of \autoref{norminvunderrestric} to construct an inverse for $f$ is to ensure that $0 \in \mathbb{A}^{n}$ is a point of $\Spec A'/I'$. We start by showing the equality in the statement of \autoref{norminvunderrestric}.

\begin{lem}
\label{addvariableunchanged}
For all places $v$ of $K$ we have $\| g \|_{(X,0),v} = \| g \|_{(\Spec A'/A'I, 0),v}$ for all $g \in A/I$.
\end{lem}

\begin{proof}
The inequality $\| g \|_{(X,0),v} \geq \| g \|_{(\Spec A'/A'I, 0),v}$ may be seen directly from the definitions, so we prove the reverse inequality. Suppose we have $x' \in A' I$ which we write as $x' = (a + a') y$, where $a \in A$, $a' \in x_{n+1} A'$, and $y \in I$. Then we have
\[ \| g + x' \|_{v} = \| (g + a y) + a' y \|_{v} = \textrm{max} \{ \| g + a y \|_{v}, \| a' \|_{v} \| y \|_{v} \} , \]
where the equality comes from the fact $x_{n+1} | a'$, so $g + a y$ and $a' y$ share no monomials in common. But $\| g + ay \|_{v} \geq \| g \|_{(X,0),v}$ by definition, so it follows that $\| g + x' \|_{v} \geq \| g \|_{(X,0),v}$. Now take $x'$ to be a minimal representative for $g + A' I'$.
\end{proof}

\begin{proof}[Proof of (\ref{norminvunderrestric}):]
We start by introducing some notation. We let $S$ be a finite set of finite primes of $K$. We set $\mathcal{A} = \mathcal{A}_{S} = \mathcal{O}_{K,S}[x_{1}, \hdots, x_{n}]$, $\mathcal{A}' = \mathcal{A}'_{S} = \mathcal{O}_{K,S}[x_{1}, \hdots, x_{n}, x_{n+1}]$ and $\mathcal{I} = I \cap \mathcal{A}$. We also set $\mathcal{A}'_{(v)} = \mathcal{O}_{K,(\mathfrak{p}_{v})}[x_{1}, \hdots, x_{n}, x_{n+1}]$, where $\mathfrak{p}_{v}$ is the prime corresponding to $v$ and $\mathcal{O}_{K,(\mathfrak{p}_{v})}$ is the localization.

After possibly increasing $S$ we may assume $\| f \|_{v} = \| x'_{n+1} \| = \| x'_{n+1} f - 1 \|_{v} = 1$, which we do throughout. 

\begin{lem}
After possibly increasing $S$, for every $r \in A'$ there exists a scalar multiple $\lambda r$ with $\lambda \in K$ such that $\lambda r \in \mathcal{A}'$ and $\| \lambda r \|_{v} = 1$ for all $v \not\in S$.
\end{lem}

\begin{proof}
The class group of $\mathcal{O}_{K}$ is finitely generated, so we may choose $S$ large enough so that $\mathcal{O}_{K,S}$ is a principal ideal domain, hence a unique factorization domain. Then for each prime $\mathfrak{p} \not\in S$ there is a coefficient $r_{\mathfrak{p}}$ of $r$ for which the exponent $e_{\mathfrak{p}}$ of $\mathfrak{p}$ in the factorization of $(r_{\mathfrak{p}})$ is as small as possible. We then choose $\lambda$ so that the factorization of $(\lambda)$ has this prime occuring with exponent $- e_{\mathfrak{p}}$, which is possible by the triviality of the class group.
\end{proof}

\begin{lem}
Let $t = x'_{n+1} f - 1$. Then after possibly increasing $S$, we have 
\[ (A' \mathcal{I} + A' t) \cap \mathcal{A}' = \mathcal{A}' \mathcal{I} + \mathcal{A}' t . \]
And for each $v \not\in S$ that
\[ (A' \mathcal{I} + A' t) \cap \mathcal{A}'_{(v)} = \mathcal{A}'_{(v)} \mathcal{I} + \mathcal{A}'_{(v)} t . \]
\end{lem}

\begin{proof}
We begin with the first statement. The basic claim is that if $H = (h_{1}, \hdots, h_{\ell}) \subset A'$ is any ideal for which $\Spec A'/H$ is an integral scheme, then for $S$ large enough one has both $h_{i} \in \mathcal{A}'$ and that $H \cap \mathcal{A}' = \sum_{i=1}^{\ell} \mathcal{A}' h_{i} =: \mathcal{H}$. It is clear we can increase $S$ so that the $h_{i}$ lie in $\mathcal{A}'$. Then we have a diagram of exact sequences
\begin{equation}
\label{spreadidealseq}
\begin{tikzcd}
 0 \arrow[r] & \mathcal{H} \arrow[d] \arrow[r] & \mathcal{A}' \arrow[r] \arrow[d, "\iota"] & \mathcal{A}'/\mathcal{H} \arrow[d] \arrow[r] & 0 \\
 0 \arrow[r] & H \arrow[r] & A' \arrow[r] & A'/H \arrow[r] & 0 ,
\end{tikzcd}
\end{equation}
so the claim reduces to the injectivity of the arrow $\mathcal{A}'/\mathcal{H} \to A'/H$. Increasing $S$ so that $\Spec \mathcal{A}'/\mathcal{H}$ is integral (i.e., irreducible and reduced), the map $\Spec A'/H \to \Spec \mathcal{A}'/\mathcal{H}$ is dominant. The result then follows from \cite[\href{https://stacks.math.columbia.edu/tag/0CC1}{Lemma 0CC1}]{stacks-project}.

The second statement is now argued analogously, by taking (\ref{spreadidealseq}) and localizing at $v$.
\end{proof}

Using the first lemma and the scale-invariance of the norms $\| \cdot \|_{(\Spec A'/A'I, 0),v}$ and $\| \cdot \|_{(X',0),v}$, it now suffices to prove the statement for those $g \in \mathcal{A}$ with $\| g \|_{v} = 1$. Let $\widetilde{g} = \widetilde{g}_{v}$ be a minimal representative for $g + I'$ with respect to $\| \cdot \|_{(X',0), v}$. Necessarily $\| \widetilde{g} \|_{v} \leq 1$ so $\widetilde{g} \in \mathcal{A}'_{(v)}$. Using the second lemma may write $\widetilde{g}$ as 
\[ \widetilde{g} = g + \underbrace{\left( \sum_{i \geq 0} a_{i} x^{i}_{n+1} \right)}_{a} y + \underbrace{\left( \sum_{i \geq 0} b_{i} x^{i}_{n+1} \right)}_{b} ((x_{n+1} + 1/f(0)) f - 1) \]
where $a_{i}, b_{i} \in \mathcal{A}_{(v)}$ for each $i$, and $y \in \mathcal{I}$.

By grouping coefficients of $x^{i}_{n+1}$, we observe that
\begin{align}
\| \widetilde{g} \|_{v} &= \textrm{max} \begin{cases} \| g + a_{0} y + b_{0} (f/f(0) - 1) \|_{v} \\ \textrm{max}_{i \geq 1} \{ \| a_{i} x^{i}_{n+1} y + b_{i-1} x^{i-1}_{n+1} x_{n+1} f + b_{i} x^{i}_{n+1} (f /f(0) - 1) \|_{v} \} \end{cases}  \\
\label{gtildexp}
 &= \textrm{max} \{ \| g + a_{0} y + b_{0} (f/f(0) - 1) \|_{v}, \textrm{max}_{i \geq 1} \{ \| a_{i} y + b_{i-1} f + b_{i} (f/f(0) - 1) \|_{v} \} \} 
\end{align}
\noindent Since $f$ is not a zero divisor in $A/I$ (because $A/I$ is integral), we may increase $S$ to assume it is also not a zero divisor in $\mathcal{A}_{(v)}/(\mathcal{I} + \mathfrak{p}_{(v)})$. Now \emph{suppose for contradiction} that $\| \widetilde{g} \|_{v} < \| g \|_{(X,0),v} = 1$.

\begin{lem}
$\| b_{i} \|_{v} = 1$ for all $i \geq 0$.
\end{lem}

\begin{proof}
We observe that our assumption $\| \widetilde{g} \|_{v} < 1$ implies that each of the quantities $\| ? \|_{v}$ appearing (\ref{gtildexp}) also have magnitude less than $1$. 

Now for the base case $i = 0$ we observe that $\| g + a_{0} y \| \geq \| g \|_{(X,0),v} = 1$, so in order to have $\| g + a_{0} y + b_{0} (f/f(0) - 1) \|_{v} < 1$ we must also have $\| b_{0} (f/f(0) - 1) \|_{v} \geq 1$. Our choice of $S$ ensures $\| (f/f(0) - 1) \|_{v} \leq 1$ so this implies $\| b_{0} \|_{v} \geq 1$. But $b_{0}$ is integral, so $\| b_{0} \|_{v} = 1$.

For the inductive case we reduce $a_{i} y + b_{i-1} f + b_{i} (f/f(0) - 1)$ modulo $\mathfrak{p}_{v}$. If $\| b_{i} \|_{v} < 1$ then the fact that $\| b_{i-1} \|_{v} = 1$ and $y \in \mathcal{I}$ implies that $f$ is a zero-divisor modulo $\mathcal{I} + \mathfrak{p}_{(v)}$. We chose $S$ so that this is not true, hence $\| b_{i} (f/f(0) - 1) \|_{v} = 1$, which implies $\| b_{i} \|_{v} = 1$ as before.
\end{proof}

But $b$ is a polynomial, so $b_{i} = 0$ for $i \gg 0$. This is our contradiction.
\end{proof}

\begin{lem}
\label{getfracreplem}
Let $A = K[x_{1}, \hdots, x_{n}]$, $A' = K[x_{1}, \hdots, x_{n+1}]$, and consider $f \in A$ with $f(0) \neq 0$. Set $x'_{n+1} = x_{n+1} + 1/f(0)$. Choose $a' \in A'$. Then there exists $a \in A$ and an integer $m$ such that $a' = a/f^{m}$ in $A_{f} \cong A'/(x'_{n+1} f - 1)$ and 
\[ \deg a \leq m \deg f + \deg a', \hspace{2em} m = \deg_{x_{n+1}} a' \hspace{2em} \textrm{and} \hspace{2em} \| a \|_{v} \leq c^{m}_{v} \| a' \|_{v} , \]
where $c_{v}$ is a constant independent of $a$ and equal to $1$ for all but finitely many $v$.
\end{lem}

\begin{proof}
We may consider the unique decomposition $a' = \sum_{i=0}^{m} a'_{i} (x'_{n+1})^{i}$ with $a'_{i} \in A$, so that setting $a = \sum_{i} f^{m-i} a'_{i}$ we have $a' = a / f^{m}$ in $A'/(x'_{n+1} f - 1)$. Then $\| a \|_{v} \leq \max_{i} \| a'_{i} f^{m-i} \|_{v}$. 

Now write $\| \cdot \|'_{v}$ for the Gauss norm on $A'$ but with respect to the presentation $A' = K[x_{1}, \hdots, x_{n}, x'_{n+1}]$. Then there exists a constant $c_{v}$ such that for all $b \in A'$ we have $\| b \|'_{v} \leq c^{m(b)}_{v} \| b \|_{v}$ where $m(b)$ is the $x_{n+1}$ (or equivalently $x'_{n+1}$) degree of $b$. It follows then that 
\[ \| a'_{i} \|_{v} = \| a'_{i} \|'_{v} \leq \| a' \|'_{v} \leq c^{m}_{v} \| a' \|_{v} \]
hence we obtain the result after replacing $c_{v}$ with $c_{v} \max\{ 1, \| f \|_{v} \}$. 
\end{proof}

\begin{lem}
\label{normundercoordchange}
Let $B = K[y_{1}, \hdots, y_{m}]$, and suppose $\varphi : B/J \xrightarrow{\sim} A/I$ is an isomorphism preserving the point $0$. Write $Y = \Spec B/J$. Then for each $v$ there exists a positive constant $c_{v}$ such that $\| \varphi(f) \|_{(X,0),v} \leq c_{v} \| f \|_{(Y,0),v}$ for all $f \in B/J$. For all but finitely many $v$ one can take $c_{v} = 1$, so in particular the norms $\| \cdot \|_{(Y,0),v}$ and $\| \varphi(\cdot) \|_{(X,0),v}$ agree on $B/J$ for all but finitely many $v$.
\end{lem}

\begin{proof}
We start with the argument for all but finitely many places. Lift the map $\varphi$ to a map $\widetilde{\varphi} : B \to A$ represented by polynomials $g_{1}, \hdots, g_{m}$. Then if $f$ is represented by $\widetilde{f}$ such that $\| f \|_{(Y,0),v} = \| \widetilde{f} \|_{v}$, then we may observe that $\varphi(f)$ is represented by $\widetilde{f} \circ (g_{1}, \hdots, g_{m})$ and 
\[ \| \varphi(f) \|_{(X,0),v} \leq \| \widetilde{f} \circ (g_{1}, \hdots, g_{m}) \|_{v} \leq \| \widetilde{f} \|_{v} \underbrace{\| g_{1} \|_{v}^{b_{1}} \cdots \| g_{m} \|_{v}^{b_{m}}}_{c_{v}} \]
for some positive integer exponents $b_{1}, \hdots, b_{m}$ depending on $\widetilde{f}$. However, away from finitely many places we may assume $\| g_{i} \|_{v} = 1$ for each $i$, so $c_{v} = 1$, \emph{independently of }$\widetilde{f}$ (and hence $f$) at these places (note that the $g_{1}, \hdots, g_{m}$ are independent of $f$ and $\widetilde{f}$). Applying the same argument in reverse gives the result for $v$ away from a finite set.

Now we consider the case for fixed $v$. In this case we let $\hat{A} = K_{v} \langle x_{1}, \hdots, x_{n} \rangle$ and $\hat{B} = K_{v} \langle y_{1}, \hdots, y_{m} \rangle$ be the associated Tate algebras and set $\hat{I} = \hat{A} I$ and $\hat{J} = \hat{B} J$. Then using the embeddings $A/I \hookrightarrow \hat{A} / \hat{I}$ and $B/J \hookrightarrow \hat{B}/\hat{J}$ provided by \autoref{compinj}, we may use \autoref{normextagrees} and \autoref{normextagrees2} to reduce to the same problem for $\hat{\varphi} : \hat{A}/\hat{I} \to \hat{B}/\hat{J}$. This then follows by combining \cite[2.1.8. Cor. 4]{zbMATH03857279} with \cite[\S6.1.3]{zbMATH03857279}. Note that there is no dependence of $c_{v}$ on $f$ in this case. 
\end{proof}

\autoref{normundercoordchange} explains why, in a certain sense, one can think of the norms $\| \cdot \|_{(X,0),v}$ as depending only on the germ $(X,0)$ and not the embedding $X \hookrightarrow \mathbb{A}^{n}$. In particular, any two choices of embeddings will give the same norms for large enough places, and result in topologically equivalent norms at the other places.

\begin{lem}
\label{obviousboundlem}
Fix a finite place $v$ of $K$, and let $\| \cdot \|_{(X, 0), v} : A/I \to \mathbb{R}_{\geq 0}$ be the non-archimedean norm given by  (\ref{nonarchnorm}). Let $f \in A/I$ be an element not vanishing at $0$, and let $\| \cdot \|_{f,v} : (A/I)_{f} \to \mathbb{R}_{\geq 0}$ be any non-archimedean norm associated to a presentation $B/J \cong (A/I)_{f}$ of $(A/I)_{f}$ and given by 
\begin{equation}
\label{nonarchnorm2}
g + J \mapsto \min_{g' \in g + J} \| g' \|_{v} ,
\end{equation}
where the right-hand norm is the usual Gauss norm on $B = K[y_{1}, \hdots, y_{m}]$. Then there exists a constant $c_{v}$ such that $\| \iota(-) \|_{f,v} \leq c_{v} \| \cdot \|_{(X,0), v}$ as functions on $A/I$, where $\iota : A/I \to (A/I)_{f}$ is the natural map.
\end{lem}

\begin{proof}
Using the presentation-independence property of non-archimedean norms on $(A/I)_{f} \cong B/J$ from \autoref{normundercoordchange} (also from \cite[2.1.8. Cor. 4]{zbMATH03857279} and \cite[\S6.1.3]{zbMATH03857279} after reducing to considering $\hat{A}' / \hat{I}'$ and $\hat{B} / \hat{J}$, with $A'/I'$ an alternative presentation of $(A/I)_{f}$), we may assume the presentation $B/J$ is of the form $B = A'$ and $J = I'$, where we use the notation of \autoref{norminvunderrestric}. Then by \autoref{addvariableunchanged}, we have $\| \cdot \|_{(X,0),v} = \| \cdot \|_{(\Spec A'/A'I, 0),v}$ as functions on $A/I$. The inequality $\| \cdot \|_{f,v} \leq \| \cdot \|_{(\Spec A'/A'I, 0),v}$ as functions on $A'/I' \cong (A/I)_{f}$ then follows from the definitions, since $A' I \subset I'$. 
\end{proof}

\begin{lem}
\label{spunitparlem}
Suppose $\Spec A/I = \bigcup_{i=1}^{n} D(h_{i})$ is a finite open cover such that $0 \in \bigcap_{i} D(h_{i})$. Then there exist functions $g_{i} \in A/I$, not vanishing at $0$, such that $\sum_{i} g_{i} h_{i} = 1$. 
\end{lem}

\begin{proof}
Start by choosing non-zero constants $r_{i} \in K^{\times}$ such that $\sum_{i} r_{i} h_{i}(0) = 1$; note that $h_{i}(0) \neq 0$ for all $i$ by assumption, so this is always possible. Then because $\{ D(h_{i}) \}_{i=1}^{n}$ is a cover, we have that $(h_{1}, \hdots, h_{n}) = A/I$. Letting $\mathfrak{m} \subset A/I$ be the maximal ideal at $0$, it follows that $\mathfrak{m} (h_{1}, \hdots, h_{n}) = \mathfrak{m}$, and so there exists $g'_{i} \in \mathfrak{m}$ such that
\[ 1 - \sum_{i} r_{i} h_{i} = \sum_{i} g'_{i} h_{i} . \]
Rearranging and setting $g_{i} = r_{i} + g'_{i}$ the result follows. 
\end{proof}

\begin{lem}
\label{normsmallerthanreponsomeopen}
Let $\Spec A/I = \bigcup_{i=1}^{n} D(h_{i})$ be a finite open cover such that $0 \in \bigcap_{i} D(h_{i})$, and write $\{ \| \cdot \|_{i,v} \}$ for a collection of adelic norms associated to the $D(h_{i})$. Then there exists constants $c_{v}$ such that 
\[ \| \cdot \|_{(X,0),v} \leq c_{v} \textrm{max}_{i} \{ \| \cdot \|_{i,v} \} \]
with $c_{v} = 1$ for all but finitely many $v$ and the inequality is as functions on $A$.
\end{lem}

\begin{proof}
The claim for all but finitely many places is handled by \autoref{norminvunderrestric}, so it suffices to handle the claim at a fixed $v$. By \autoref{spunitparlem} there exists functions $g_{i} \in A/I$, not vanishing at $0$, such that $\sum_{i} g_{i} h_{i} = 1$. Since $D(g_{i} h_{i}) \subset D(h_{i})$, one has $\| \cdot \|_{g_{i} h_{i}, v} \leq c'_{i} \| \cdot \|_{i}$ as functions on $(A/I)_{h_{i}}$ by \autoref{obviousboundlem} for some constants $c'_{i}$, so it suffices to replace $h_{i}$ with $g_{i} h_{i}$ and hence consider the case where $g_{i} = 1$. 

 Using \autoref{normextagrees} and \autoref{normextagrees2} as usual to compare the norms, we now argue we may reduce to proving the same statement for the corresponding Tate algebra $\hat{A}/\hat{I}$ and the Tate algebras $B_{i} := (\hat{A}/\hat{I})\langle X \rangle/(X h_{i} - 1)$. By \cite[\S6.2.4]{zbMATH03857279} and the fact that $\Spec A$ is assumed geometrically irreducible and geometrically reduced, we may replace all the norms involved with the associated supremum norms. It therefore suffices to show that $\Sp B_{i}$ is a cover of $\Sp (\hat{A}/\hat{I})$. Indeed, the equality $\sum_{i} h_{i} = 1$ implies that, modulo each maximal ideal of $\hat{A}/\hat{I}$, at least one of the $h_{i}$ has norm $\geq 1$, hence the corresponding maximal ideal lies in $\Sp B_{i}$. 
\end{proof}

\begin{lem}
\label{norminvanyrest}
Suppose that $X' := \Spec B \subset \Spec A/I$ is any affine open subvariety containing $0$, and fix an adelic norm $\| \cdot \|_{(X', 0),v}$ associated to some presentation of $B$. Then if $\iota : A/I \to B$ is the natural map, then the norms $\| \cdot \|_{(X,0),v}$ and $\| \iota(\cdot) \|_{(X',0),v}$ on $A/I$ agree away from finitely many places of $K$.
\end{lem}

\begin{proof}
Choose $f \in A \setminus I$ such that $D(f) \subset X'$ and $0 \in D(f)$, and write $j$ for the map $A/I \to \mathcal{O}(D(f))$. We have from \autoref{norminvunderrestric} that $\| \cdot \|_{(X,0),v} = \| j( \cdot ) \|_{(D(f),0),v}$ away from finitely many places of $K$. On the other hand we may choose presentations $X' = \Spec A'/I'$ and $D(f) = \Spec A''/I''$ where $A \subset A' \subset A''$ is a sequence of polynomial rings over $K$ and see that
\[ \| \cdot \|_{(X,0),v} \geq \| \iota(\cdot) \|_{(X',0),v} \geq \| j(\cdot) \|_{(D(f),0),v} \]
if the adelic norms are constructed with respect to these presentations. Using the independence of the presentation from \autoref{normundercoordchange} of adelic norms away from finitely many places the result follows.
\end{proof}

\begin{lem}
\label{mulinverse}
Let $f \in (A/I)[y_{1}, \hdots, y_{m}]$ be a polynomial whose constant term is a unit. Then it has a multiplicative inverse $f^{-1}$ in $(A/I)[[[y_{1}, \hdots, y_{m}]]$ whose $y^{o_{1}}_{1} \cdots y^{o_{m}}_{m}$ coefficient is bounded by $\kappa_{v} c^{o}_{v}$ in the norm $\| \cdot \|_{(X,0),v}$, where $\kappa_{v}, c_{v} > 0$ are constants depending on $v$ and $o = o_{1} + \cdots + o_{m}$. For all but finitely many $v$ we may take $c_{v} = 1$ and $\kappa_{v} = 1$.
\end{lem}

\begin{rem}
Note that the situation where $A/I = K$ and $\| \cdot \|_{(X,0),v} = | \cdot |_{v}$ is a special case of \autoref{mulinverse}.
\end{rem}

\begin{proof}
Write $f = \sum_{J} a_{J} y^{J}$ using multi-index notation. By writing $f^{-1} = \sum_{L} b_{L} y^{L}$ and using that $f f^{-1} = 1$, we obtain recurrence formulas 
\[ b_{M} = \frac{-1}{a_{0}} \sum_{\substack{J+L=M \\ J \neq 0}} a_{J} b_{L} \]
which determine the coefficients of the inverse. Arguing inductively and using the non-Archimedean triangle inequality one can see that $\| b_{M} \|_{(X,0),v} \leq \kappa_{v} (\| f \|_{(X,0),v} \| 1/a_{0} \|_{(X,0),v})^{|M|}$ for some constant $\kappa_{v}$, which gives the result with $c_{v} = \| f \|_{(X,0),v} \| 1/a_{0} \|_{(X,0),v}$.
\end{proof}

Finally, we introduce some terminology for working with power series with coefficients in a polynomial ring which will be useful later.

\begin{defn}
\label{adandpowbound}
Let $P \in (A/I)[[y_{1}, \hdots, y_{m}]]$ be a power series, set $X = \Spec A/I$, which we assume is geometrically irreducible and geometrically reduced. We say that the coefficients of $P$ are
\begin{itemize}
\item[-] adelic (w.r.t. $\overline{y} = (y_{1}, \hdots, y_{m})$), if there exists positive real constants $(\kappa_{v}, c_{v})$, both equal to $1$ for all but finitely many $v$, such that the $y_{1}^{o_{1}} \cdots y_{m}^{o_{m}}$-coefficient $P_{o}$ is bounded in the norm $\| \cdot \|_{(X,0),v}$ by $\kappa_{v} c^{o}_{v}$, where $o = \sum_{i} o_{i}$;
\item[-] $c_{v}$-adelic, if it is adelic for particular constants $\{ c_{v} \}_{v \in \Sigma_{K,f}}$ that we specify;
\item[-] $\alpha$-adelic (w.r.t. $\overline{y} = (y_{1}, \hdots, y_{m})$), if in addition there exists integers $(\alpha, \beta)$ such that, for each $v$, each $P_{o}$ admits a representative $\widetilde{P}_{o} \in A$ with $\| \widetilde{P}_{o} \|_{v} \leq \kappa_{v} c^{o}_{v}$ and $\deg \widetilde{P}_{o} \leq o \alpha + \beta$; and
\item[-] $(c_{v}, \alpha)$-adelic, analogously.
\end{itemize}
We also say that $P$ is power-adelic if it is $\alpha$-adelic for some $\alpha$ which we do not specify.
\end{defn}
\noindent We note that the property of being adelic does not depend on the presentation of the ring $A/I$ as a consequence of \autoref{normundercoordchange}. 

\begin{lem}
\label{adelicsufftocheck}
Suppose that $P \in (A/I)[[y_{1}, \hdots, y_{m}]]$ is a power series whose coefficients have size $\leq 1$ in the norm $\| \cdot \|_{(X,0),v}$ for all but finitely many $v$, and which satisfies the following property:
\begin{quote}
There exist positive real constants $\{ (\kappa'_{v}, c'_{v}) \}_{v \in \Sigma_{K,f}}$ and integers $(\alpha, \beta)$ such that, for each $v$ and each $\mathbf{o} = (o_{1}, \hdots, o_{m})$ with $o = \sum_{i} o_{i}$, the coefficient $P_{\mathbf{o}}$ of $y_{1}^{o_{1}} \cdots y_{m}^{o_{m}}$ admits a representative $\widetilde{P}_{\mathbf{o}}$ with both $\| \widetilde{P}_{\mathbf{o}} \|_{v} \leq \kappa'_{v} c'^{o}_{v}$ and $\deg \widetilde{P}_{\mathbf{o}} \leq o \alpha + \beta$.
\end{quote}
Then $P$ is $\alpha$-adelic.
\end{lem}

Note that in the statement of the lemma, we do not require that all but finitely many $c'_{v}$ or $\kappa'_{v}$ are equal to $1$. The substance of the statement is that for verifying $P$ is $\alpha$-adelic, one only has to check that the representatives $\widetilde{P}_{o}$ can be chosen to simultaneously satisfy the norm and degree condition for an arbitrarily large finite set of places; for all sufficiently large places one can check the norm and degree conditions separately, and one will always be able to find an appropriate representative satisfying both. \autoref{adelicsufftocheck} thus follows from:

\begin{lem}
For all but finitely many finite places $v$ of $K$, each coset $g + I \in A/I$ admits a representative $\widetilde{g} \in A$ such that both $\| g \|_{(X,0),v} = \| \widetilde{g} \|_{v}$ and $\widetilde{g}$ has degree as small as possible for an element of $g + I$.
\end{lem}

\begin{proof}
Let $b_{1}, \hdots, b_{\ell}$ be a reduced Gr\"obner basis (see \cite{zbMATH00217454} for background) for $I$, obtained with respect to a monomial ordering which first orders by the (total) degree of the monomial and then lexicographically with respect to the variables (e.g., the degrevlex order). Then every coset $g + I$ has a unique reduced representative with respect to this basis (called its normal form) which is obtained by \emph{lead reduction}: one constructs a sequence
\[ g =: g_{0}, g_{1}, g_{2}, \hdots, g_{k}, g_{k+1}, \hdots, g_{\infty} \]
where $g_{i+1}$ is obtained from $g_{i}$ as $g_{i+1} = g_{i} - m_{i} b_{j(i)}$ for some $b_{j(i)}$, where $m_{i}$ is a monimal chosen so that $\textrm{LT}(g_{i}) = \textrm{LT}(m_{i} b_{j(i)})$, and where $\textrm{LT}$ denotes ``leading term'' with respect to the monomial order. This process eventually terminates when no applicable $m_{i} b_{j(i)}$ can be constructed and gives a unique result $g_{\infty}$. By construction, one has
\begin{itemize}
\item[-] $\deg g_{\infty} \leq \deg g$; and
\item[-] $\| g_{i+1} \|_{v} \leq \textrm{max} \{ \| g_{i} \|_{v}, \| m_{i} \|_{v} \| b_{j(i)} \|_{v} \}$ for all $i$.
\end{itemize}
Away from finitely many places $v$ of $K$ one can assume all coefficients of all the $b_{?}$ have valuation $1$, and in this case $\| m_{i} \|_{v} \leq \| g_{i} \|_{v}$ by construction (the only coefficient of $m_{i}$ agrees in valuation with a coefficient of $g_{i}$). Thus up to ignoring finitely many places of $K$, Gr\"obner basis reduction produces a representative for $g + I$ which is simultaneously minimal for both $\| \cdot \|_{v}$ and its degree.
\end{proof}

\subsection{Adelic Tubes}
\label{adelictubesec}

A useful piece of language for discussing where collections $K$-algebraic power series converge will be the notion of an ``adelic tube''.

\begin{defn}
A smooth pair $(Y, M)$ is a smooth pure-dimensional affine $K$-algebraic variety $M$ with a smooth pure-dimensional closed subvariety $Y \subset M$ defined which is a complete intersection in $M$. A morphism between smooth pairs $(Y', M')$ and $(Y, M)$ is a map $M' \to M$ which induces a map $Y' \to Y$. 
\end{defn}

The theory of rigid spaces $X$ over a non-archimedean local field $L$ admits many different formulations, include those due to Tate, Raynaud, Berkovich and Huber. Although we use Huber's adic spaces in this paper, all the rigid spaces we consider can in principle be understood from any of the other three perspectives. In particular, for us it will often be more convenient to specify an adic space by specifying an associated analytic space in the sense of Berkovich. For instance if we consider the affine line $(\Spec L[x])^{\textrm{an}}$, the open neighbourhoods defined by ``$|x| < 1$'' in Berkovich's and Huber's theories differ. To emphasize we mean to use the neighbourhood determined by Berkovich's theory, we will use the phrase ``Berkovich neighbourhood'', as in ``the Berkovich neighbourhood defined by $|x| < 1$''. 

\begin{defn}
\label{adelictubedef}
An adelic tube associated to a smooth pair $(Y, M)$ is a collection $\mathcal{T}(\iota, \rho, \alpha, \overline{y}) = \{ \mathcal{T}(\iota, \rho, \alpha, \overline{y})^{v} \}_{v \in \Sigma_{K,f}}$ of Berkovich neighbourhoods of $M^{v}$ defined by the inequalities 
\begin{equation}
\label{tubedefeqs}
|y^{\alpha_{1}}_{i} z_{r}^{\alpha_{2}}| < |\rho| \hspace{3em}    1 \leq r \leq \ell, \hspace{1em} 1 \leq i \leq p , \hspace{1em} \alpha_{1} + \alpha_{2} \leq \alpha , \hspace{1em} \alpha_{1} \geq 1, \hspace{1em} \alpha_{2} \geq 0 .
\end{equation}
where
\begin{itemize}
\item[-] $Y$ is the vanishing locus of $\overline{y} = (y_{1}, \hdots, y_{p})$ and $\codim_{M} Y = p$;
\item[-] $z_{1}, \hdots, z_{\ell}$ are the standard coordinates associated to a closed embedding $\iota : M \hookrightarrow \mathbb{A}^{\ell}$;
\item[-] $\alpha \geq 2$ is a positive integer, and the $\alpha_{1}$ and $\alpha_{2}$ range over all integers satisfying the stated conditions; and
\item[-] $\rho \in K^{\times}$ is a non-zero constant such that $|\rho|_{v} \leq 1$ for all finite places $v$ of $K$.
\end{itemize}
\end{defn}
\noindent For example, when $\alpha = 2$ and $p = 1$, the definition reduces to the collection of Berkovich neighbourhoods defined by the inequalities $|y| < |\rho|$ and $|y z_{r}| < |\rho|$ for $1 \leq r \leq \ell$. We note that all functions involved in the inequalities are $K$-algebraic; i.e., are elements of the coordinate ring of $M$.

\begin{defn}
\label{tuberefdef}
A refinement of $\mathcal{T}(\iota, \rho, \alpha, \overline{y})$ is an adelic tube $\mathcal{T}(\iota, \rho', \alpha', \overline{y})$ such that $\alpha \leq \alpha'$ and $|\rho'|_{v} \leq |\rho|_{v}$ for all places $v$. Note that this implies that
\[ \mathcal{T}(\iota, \rho', \alpha', \overline{y})^{v} \subset \mathcal{T}(\iota, \rho, \alpha, \overline{y})^{v} \]
for all $v$. 
\end{defn}

\begin{lem}
\label{adelictubereflem}
Given any two adelic tubes $\mathcal{T} = \mathcal{T}(\iota, \rho, \alpha, \overline{y})$ and $\mathcal{T}' = \mathcal{T}(\iota', \rho', \alpha', \overline{y}')$ associated to the same $(Y,M)$, there exists a refinement $\mathcal{T}'' = \mathcal{T}(\iota', \rho'', \alpha'', \overline{y}')$ of $\mathcal{T}'$ such that $\mathcal{T}''^{v} \subset \mathcal{T}^{v}$ for all $v$.
\end{lem}

\begin{proof}[Proof of (\ref{adelictubereflem}):]
It suffices to produce inequalities of the required type defining some $\mathcal{T}''$ which formally imply the inequalities defining $\mathcal{T}$. We write $\iota : M \hookrightarrow \mathbb{A}^{\ell}$ and $\iota' : M \hookrightarrow \mathbb{A}^{\ell'}$ for the two closed embeddings, where the associated affine spaces have coordinates $\overline{z} = (z_{1}, \hdots, z_{\ell})$ and $\overline{z}' = (z'_{1}, \hdots, z'_{\ell'})$, respectively. We can lift the identity map $M \to M$ to a map $m : \mathbb{A}^{\ell} \to \mathbb{A}^{\ell'}$. Then there exists:
\begin{itemize}
\item[-] polynomials $h_{iw}$ in the coordinates $\overline{z}'$ such that $y_{i} = \sum_{w} h_{iw} y'_{w}$ on $M$; and
\item[-] polynomials $g_{i}$ in the coordinates $\overline{z}'$ such that $z_{i} = g_{i}(z'_{1}, \hdots, z'_{\ell'})$. 
\end{itemize}
Applying the non-archimedean triangle inequality we find that, for each pair of non-negative integers $(\alpha_{1}, \alpha_{2})$ and $\alpha_{1} \geq 1$:
\begin{align*}
|y_{i}^{\alpha_{1}} z_{r}^{\alpha_{2}}| \leq \textrm{max}_{w, r', \tau \leq \tau_{0}}  \{ c_{v} |y'^{\alpha_{1}}_{w} z'^{\tau (\alpha_{1} + \alpha_{2})}_{r'}| \} 
\end{align*}
where $c_{v}$ and $\tau_{0}$ are some fixed constants, where $c_{v}$ depending on $v$ and equal to $1$ for all but finitely many $v$. Now we choose $\rho''$ such that $|\rho''|_{v} \ll |\rho|_{v}$ at all places $v$ of $K$ where $c_{v} \neq 1$, and $|\rho''|_{v} \leq |\rho|_{v}$ otherwise. We take $\alpha'' = 2 \tau_{0} \alpha$. Then inside $\mathcal{T}''$ and for each $(w, r')$, one sees from the inequalities defining $\mathcal{T}''$ that each term $c_{v} |y'^{\alpha_{1}}_{w} z'^{\tau (\alpha_{1} + \alpha_{2})}_{r'}|$ is $< |\rho|$ after choosing $|\rho''|$ sufficiently small. Ranging over all applicable $(\alpha_{1}, \alpha_{2})$ and shrinking $\rho''$ as necessary gives the result.
\end{proof}
We will use the term ``adelic neighbourhood'' to informally refer to a collection of Berkovich neighbourhoods indexed by $v \in \Sigma_{K,f}$ and where the functions involved in the inequalities all come from the coordinate ring of some $K$-algebraic variety. An adelic tube is an example of an adelic neighbourhood.

\vspace{0.5em}

\begin{defn}
\label{balltypenbhd}
Let $U$ be an affine algebraic $K$-variety. An adelic ball of $U$ is an adelic neighbourhood defined by the inequalities $\{ |z_{i}| \leq 1 \}_{1 \leq i \leq \ell}$, where $z_{1}, \hdots, z_{\ell}$ are the standard coordinates associated to some closed embedding $\iota : U \hookrightarrow \mathbb{A}^{\ell}$. 
\end{defn}

\begin{lem}
\label{chooseadelicballcontainingcohom}
Let $E$ be a smooth geometrically connected affine algebraic variety over $K$. Then there exists an adelic ball $\mathfrak{E}$ of $E$ such that the natural maps $H^{i}_{\textrm{dR}}(E^{v}) \to H^{i}_{\textrm{dR}}(\mathfrak{E}^{v})$ of overconvergent de Rham cohomology groups are, for each $i \geq 0$:
\begin{itemize}
\item[-] injective for all $v \in \Sigma_{K,f}$; and
\item[-] isomorphisms for all but finitely many $v \in \Sigma_{K,f}$.
\end{itemize}
\end{lem}

\begin{proof}
Fix a closed embedding $\iota : E \hookrightarrow \mathbb{A}^{\ell}$. For each $\rho \in K^{\times}$, let $\varphi_{\rho} : \mathbb{A}^{\ell} \to \mathbb{A}^{\ell}$ be the ``multiplication-by-$\rho$'' map. For each $\rho$ we set $\iota_{\rho} = \varphi_{\rho} \circ \iota$. We then choose a sequence $\rho_{j} \in K^{\times}$ such that $|\rho_{j}|_{v} \to \infty$ as $j \to \infty$ for each $v$. Set $\iota_{j} = \iota_{\rho_{j}}$, and let $\mathfrak{E}_{j}$ denote the associated adelic ball. It follows from the proof of \cite[Cor. 3.2]{zbMATH02133479} that $H^{i}_{\textrm{dR}}(E^{v}) = \varprojlim_{j} H^{i}_{\textrm{dR}}(\mathfrak{E}^{v}_{j})$. This implies the injectivity of $H^{i}_{\textrm{dR}}(E^{v}) \to H^{i}_{\textrm{dR}}(\mathfrak{E}^{v}_{j})$ for $j$ large enough (depending on $v$).

It then suffices to prove that, for each $j$, we have $H^{i}_{\textrm{dR}}(E^{v}) = H^{i}_{\textrm{dR}}(\mathfrak{E}^{v}_{j})$ for all but finitely many $v$; after this is shown one can adjust the sequence $\rho_{j}$ so that $|\rho_{j}| = 1$ for $j$ sufficiently large, and then take the adelic ball defined by $\iota_{j}$ for $j$ large enough. Since $\mathfrak{E}^{v}_{j} = \mathfrak{E}^{v}_{1}$ for all but finitely many $v$, we may take $j = 1$. Now by a result of Hironaka, we can find a smooth compactification $\overline{E}$ of $E$ such that $Y := \overline{E} \setminus E$ is a simple normal crossing divisor. We may spread everything out over some number ring $\mathcal{O}_{K}[1/N]$ such that the components of $Y$ and all schemes obtained by intersecting these components are smooth and proper over $\mathcal{O}_{K}[1/N]$. Let $\mathcal{E}$ be the associated integral model of $E$ over $\mathcal{O}_{K}[1/N]$. Then the overconvergent de Rham cohomology of $\mathfrak{E}^{v}$ is identified with the Monski-Washnitzer cohomology (cf. \cite{kedlayabook}) of the reduction fo $\mathcal{E}$ at the prime corresponding to $v$, and that this agrees with the de Rham cohomology of $E^{v}$ is then \cite[Thm. 13.1.14]{kedlayabook}.
\end{proof}

\begin{defn}
\label{adelicdiskdef}
Let $\rho \in K^{\times}$. The open (resp. closed) adelic disk of radius $|\rho|$ is the collection $\{ \Delta^{v} \}_{v \in \Sigma_{K,f}}$ of Berkovich neighbourhoods in $\mathbb{A}^{1,v}_{K}$ defined by $|x| < |\rho|$ (resp. $|x| \leq |\rho|$), where $x$ is the standard coordiante of $\mathbb{A}^1_{K}$. 
\end{defn}

\begin{defn}
\label{stdadelictubedef}
A standard adelic tube $Y\langle p \rangle$ is an adelic tube $\mathcal{T}(\iota, \rho, \alpha, \overline{x})$ associated to $(Y, \mathbb{A}^{p} \times Y)$ with $\overline{x} = (x_{1}, \hdots, x_{p})$ the coordinate functions of $\mathbb{A}^{p}$. 
\end{defn}

\begin{notn}
For an algebraic $K$-variety $Y$, we write $Y[p]^{\wedge}$ for the formal scheme obtained as the formal neighbourhood of $Y$ in $\mathbb{A}^{p} \times Y$.
\end{notn}

\begin{lem}
\label{adelictubecontainsprodlem}
Let $Y \langle p \rangle$ be a standard adelic tube and let $\mathfrak{Y}$ be an adelic ball of $Y$. Then there exists some $\rho \in K^{\times}$ such that we have $\Delta^{v,p}_{\rho} \times \mathfrak{Y}^{v} \subset Y \langle p \rangle^{v}$ for all $v \in \Sigma_{K,f}$, where $\Delta_{\rho}$ denotes the open adelic disk of radius $|\rho|$, and the product decomposition is compatible with the one on $\mathbb{A}^{p} \times Y$.  
\end{lem}

\begin{proof}
Let $\iota : Y \hookrightarrow \mathbb{A}^{\ell}$ be the closed embedding associated to $\mathfrak{Y}$, and let $\iota' : \mathbb{A}^{p} \times Y \hookrightarrow \mathbb{A}^{m}$ be the closed embedding associated to $Y \langle p \rangle$. Using the refinement lemma \autoref{adelictubereflem}, we may assume by possibly shrinking the adelic tube $Y \langle p \rangle$ that the embedding $\iota'$ is of the form $\mathbb{A}^{p} \times Y \to \mathbb{A}^{p} \times \mathbb{A}^{\ell}$, where the second component is just $\iota$. If we now inspect the inequalities (\ref{tubedefeqs}) defining the adelic tube, it is clear that the inequalities $|z_{r}| \leq 1$ and $|y_{i}| < |\rho'|$ imply the inequalities defining $Y \langle p \rangle$ for $|\rho'|$ small enough, hence the result.
\end{proof}

\begin{defn}
\label{adelictubemorphism}
Let $(Y, M)$ and $(Y', M')$ be smooth pairs with associated adelic tubes $\mathcal{T}$ and $\mathcal{T}'$. A morphism of adelic tubes is a map $\eta : \mathcal{Y} \to \mathcal{Y}'$ of formal schemes, with $\mathcal{Y}$ (resp. $\mathcal{Y}'$) the formal neighbourhood of $Y$ (resp. $Y'$) in $M$ (resp. $M'$), and a rigid-analytic map $\eta^{v} : \mathcal{T}^{v} \to \mathcal{T}^{v}$ for each $v \in \Sigma_{K,f}$ such that $\eta$ agrees with $\eta^{v}$ under the appropriate rigid-analytification and formal completion functors.
\end{defn}

\begin{lem}
\label{adelictubezarcover}
Suppose that $h_{1}, \hdots, h_{t}$ are functions on $M$ such that $M = \bigcup_{p} (M \setminus V(h_{p}))$. Let $f_{1}, \hdots, f_{j}$ be functions defining $Y$ inside $M$ as a complete intersection. Let $\iota_{p} : M \setminus V(h_{p}) \hookrightarrow \mathbb{A}^{\ell+1}$ be the embedding defined by $z_{\ell+1} h_{p} = 1$. Then up to replacing $\mathcal{T} = \mathcal{T}(\iota, \alpha, \rho, \overline{f})^{v}$ with a refinement, there exists $\rho'_{i}$ such that 
\[ \mathcal{T}(\iota, \alpha, \rho, \overline{f})^{v} \subset \bigcup_{p} \mathcal{T}(\iota_{p}, \alpha', \rho'_{p}, \overline{f})^{v} \]
for all $v$.
\end{lem}

\begin{rem}
We note that the statement for the other inclusion 
\[ \bigcup_{p} \mathcal{T}(\iota_{p}, \alpha, \rho, \overline{f})^{v} \subset \mathcal{T}(\iota, \alpha, \rho, \overline{f})^{v} \]
is obvious.
\end{rem}

\begin{proof}
The equality $Y = \bigcup_{p} (Y \setminus V(h_{p}))$ implies that there exist functions $r_{p}$ on $M$ such that $\sum_{p} r_{p} h_{p} = 1$. 

To show that $\mathcal{T}^{v}$ lies inside the union it suffices to show that whenever the inequalities defining $\mathcal{T}^{v}$ hold we also have, for some choice of $p$, for all $i$ and applicable $(\alpha_{1}, \alpha_{2})$, that $|f^{\alpha_{1}}_{i}| < |\rho'_{p} h_{p}^{\alpha_{2}}|$ holds. Because $\sum_{p} r_{p} h_{p} = 1$ there is some $p$ such that $|r_{p} h_{p}| \geq 1$, which also implies that $|r^{\alpha_{2}}_{p} h^{\alpha_{2}}_{p}| \geq 1$. Multiplying through by $|f^{\alpha_{1}}_{i}|$ we obtain $|f^{\alpha_{1}}_{i} r^{\alpha_{2}}_{p} h^{\alpha_{2}}_{p}| \geq |f^{\alpha_{1}}_{i}|$. Refining $\mathcal{T}$ if necessary and using that the $r_{p}$ are functions of the $z_{1}, \hdots, z_{\ell}$ one has $|f^{\alpha_{1}}_{i} r^{\alpha_{2}}_{p}| < |\rho'_{p}|$ for some appropriate $\rho'_{p}$. The result follows.
\end{proof}

\subsection{Parameterizing Analytic Tubes}
\label{paramadelictubesec}

We suppose that $E \subset U \subset \mathbb{A}^{m}$ are inclusions of closed irreducible smooth varieties defined over the number field $K$, with $E$ a complete intersection inside $U$. We will write $\mathbb{A}^{m} = \Spec K[z_{1}, \hdots, z_{m}]$, and suppose that $p = \dim U - \dim E$. The embeddings $U \hookrightarrow \mathbb{A}^{m}$ and $E \hookrightarrow \mathbb{A}^{m}$ will also sometimes be denoted by $\iota$. We denote the origin of $\mathbb{A}^{m}$ by $o$, and suppose that this point is a point of $E(K)$. For each place $v$ of $K$, We let $u_{1}, \hdots, u_{p}$ be the standard coordinates on a second affine space $\mathbb{A}^{p}$.

In addition, we suppose we have functions $f_{1}, \hdots, f_{p}$ on $U$ which define $E$ and suppose they extend to \'etale coordinates $f_{1}, \hdots, f_{p}, f_{p+1}, \hdots, f_{q}$ for $U$. Let $\mathcal{E}$ be the formal neighbourhood of $E$ in $U$. 

\subsubsection{Flat Deformations from $E$} 

The \'etale coordinates $f_{1}, \hdots, f_{q}$ determine a collection of partial derivative operators $\partial_{1}, \hdots, \partial_{q}$ on $U$. For a multi-index $\mathbf{J} \in \{ 1, \hdots, q \}^{|\mathbf{J}|}$ we write $\partial_{\mathbf{J}}$ for the induced differential operator. Let $\mathcal{B} \subset \mathcal{O}_{\mathcal{E}}$ be the subalgebra defined by the property that $\partial_{\ell} \mathcal{B} = 0$ for all $\ell \in \{ 1, \hdots, p \}$.

\begin{defn}
Given a function $a$ on $U$ or its analytifications $U^{v}$ for $v$ a place of $K$, we say that $a$ is $p$-flat if $\partial_{\ell} a = 0$ for $1 \leq \ell \leq p$.
\end{defn}

\begin{lem}
\label{formalretractexists}
The natural map $\mathcal{B} \to \mathcal{O}_{E}$ is an isomorphism of sheaves of $K$-algebras. Its inverse induces a map $\mathcal{E} \to E$ of formal schemes which is the identity on the level of topological spaces.
\end{lem}

\begin{proof}
The ideal $\mathcal{I}$ defining $\mathcal{O}_{E}$ inside $\mathcal{O}_{\mathcal{E}}$ is generated by $f_{1}, \hdots, f_{p}$: we show that this implies that $\mathcal{I} \cap \mathcal{B} = 0$. Given a fixed $a \in \mathcal{I}$, to check that it does not lie in $\mathcal{B}$ it suffices to check this is true after restricting to the formal stalk at some point $o \in E$. Because $f_{1}, \hdots, f_{q}$ define \'etale coordinates, one obtains a description of the formal stalk $\widehat{\mathcal{O}}_{\mathcal{E}, o}$ as formal power series ring in $f_{1}, \hdots, f_{p}, f_{p+1} - f_{p+1}(o), \hdots, f_{q} - f_{q}(o)$. The condition that $a$ lies in $\mathcal{B}$ means that it does not have any terms that involve the variables $f_{1}, \hdots, f_{p}$. But then obviously it cannot lie in the ideal $\mathcal{I}$ generated by $f_{1}, \hdots, f_{p}$ unless it is zero. Hence $\mathcal{I} \cap \mathcal{B} = 0$, and the map $\mathcal{B} \to \mathcal{O}_{E}$ is injective.

Let us now also check that it is surjective. For this we consider an operator $\delta : \mathcal{O}_{\mathcal{E}} \to \mathcal{B}$ defined as follows:
\begin{align}
\label{flatteningofa}
\delta(a) = \sum_{\mathbf{J} \in \mathbb{Z}^{j}_{\geq 0}} \left(\prod_{i=1}^{p} (-f_{i})^{J_{i}} \right) \left( \prod_{i=1}^{p} \frac{(\partial_{i})^{J_{i}}}{J_{i}!} \right) a ,
\end{align}
We note the infinite sum on the right is well-defined since $\mathcal{O}_{\mathcal{E}}$ is $I$-adically complete, with $I = (f_{1}, \hdots, f_{p})$. One checks from the formula that $\delta(\mathcal{O}_{\mathcal{E}}) \subset \mathcal{B}$ and that both $a$ and $\delta(a)$ have the same image in $\mathcal{O}_{E}$. Thus, since $\mathcal{O}_{\mathcal{E}} \to \mathcal{O}_{E}$ is surjective, so is $\mathcal{B} \to \mathcal{O}_{E}$. 

That $\mathcal{O}_{E} \xrightarrow{\sim} \mathcal{B} \to \mathcal{O}_{\mathcal{E}}$ induces a map of formal schemes is clear from the definitions. 
\end{proof}

\begin{prop}
\label{normestprop}
Let $\{ \| \cdot \|_{(U, o),v} \}_{v \in K}$ be an adelic norm associated to the fixed $K$-point $o \in U(K)$. (From now on we write $\| \cdot \|_{o,v} = \| \cdot \|_{(U,o),v}$.) Then for each $i$, there exists a constant $c_{i,v}$ such that $\| \frac{1}{k!} (\partial_{i})^{k} u \|_{o,v} \leq c^{k}_{i,v} \| u \|_{o,v}$ for all $u \in \mathcal{O}(U)$ and $k \geq 1$. For all but finitely many $v$ we may take $c_{i,v} = 1$. If instead we choose some $o \in U(K_{v})$, the inequality $\| \frac{1}{k!} (\partial_{i})^{k} u \| \leq c^{k} \| u \|$ holds for any norm  $\| \cdot \|$ associated to any affinoid neighbourhood of $o$ in $U^{v}$ and some $c$ depending on $\| \cdot \|$. 

Furthermore, if we consider the natural differential-operator preserving map
\[ \gamma :  \mathcal{O}(U) \to K[[y_{1}, \hdots, y_{q}]] \]
induced by expanding elements of $\mathcal{O}(U)$ in terms of their formal power series at $o$, then for all but finitely many places $v$ one has $\| \gamma(u) \|_{v} = \| u \|_{o,v}$, where the function $\| \cdot \|_{v}$ on $K[[y_{1}, \hdots, y_{q}]]$ is the ``Gauss norm'' (the maximum $v$-adic size of a coefficient, or $\infty$ if it doesn't exist).
\end{prop}

\begin{proof}
We start with the claim at all but finitely many places --- we will handle the finitely many remaining places separately. We may start by passing to the situation where the ideal $(f_{1}, \hdots, f_{q})$ defines the point $o \in U(K)$; because $f_{1}, \hdots, f_{q}$ give \'etale coordinates, this is true in a sufficiently small Zariski open neighbourhood of $o$ in $U$, and we may apply \autoref{norminvanyrest} to pass to such a neighbourhood. We may moreover assume, again using \autoref{norminvanyrest} and \cite[\href{https://stacks.math.columbia.edu/tag/02GT}{Lemma 02GT}]{stacks-project}, that the map $U \to \mathbb{A}^{q}$ induced by $f_{1}, \hdots, f_{q}$ is standard \'etale, so there exists a factorization
\begin{equation}
\label{stdetfac}
\underbrace{K[x_{1}, \hdots, x_{q}]}_{=: R_{q}} \to \underbrace{K[x_{1}, \hdots, x_{q}, x_{q+1}]}_{=: R_{q+1}} \to (R_{q+1})_{h} \to (R_{q+1})_{h}/(g) = \mathcal{O}(U)
\end{equation}
for some polynomials $h, g \in R_{q+1}$ such that $g$ is monic with $g'$ invertible in $(R_{q+1})_{h}/(g)$. 

By construction the ideal $\mathfrak{m}$ in $\mathcal{O}(U)$ generated by $(f_{1}, \hdots, f_{q})$ is maximal corresponding to $o$, and its inverse image under the $K$-algebra map $R_{q+1} \to \mathcal{O}(U)$ is a maximal ideal $\mathfrak{m}_{q+1}$ corresponding to a $K$-point of $\mathbb{A}^{q+1}$ whose first $q$ coordinates are $0$. After replacing $x_{q+1}$ with $x_{q+1} - a$ for some $a \in K$, one therefore obtains the same sequence but where $\mathfrak{m}_{q+1} = (x_{1}, \hdots, x_{q+1})$. 

Choosing an integral model $\mathcal{U}$ for $U$ we can spread out the sequence (\ref{stdetfac}) to a sequence of the same type defined over $\mathcal{O}_{K,S} \subset K$, where $S$ is a finite set of finite primes. We write $\mathcal{R}_{q}$ (resp. $\mathcal{R}_{q+1}$) for the polynomial rings over $\mathcal{O}_{K,S}$ obtained after spreading out. We may then complete each term in the spread-out sequence with respect to the pullback of the ideal $\mathfrak{m} = (f_{1}, \hdots, f_{q})$ to the corresponding term. The result is a sequence
\begin{equation}
\label{stdetfac2}
\underbrace{\mathcal{O}_{K,S}[[x_{1}, \hdots, x_{q}]]}_{= \mathcal{R}^{\wedge}_{q}} \to \underbrace{\mathcal{O}_{K,S}[[x_{1}, \hdots, x_{q}, x_{q+1}]]}_{= \mathcal{R}^{\wedge}_{q+1}} \to [(\mathcal{R}_{q+1})_{h}]^{\wedge} \to [(\mathcal{R}_{q+1})_{h}/(g)]^{\wedge} = \mathcal{O}(\mathcal{U})^{\wedge} .
\end{equation}

We then make the following observations:
\begin{itemize}
\item[-] By construction the polynomial $h$ does not vanish at $0$ (the image of $o$), hence by \autoref{mulinverse} the polynomial $h$ has a multiplicative inverse in $\mathcal{R}^{\wedge}_{q+1}$ after possibly increasing $S$. Thus we obtain a canonical map $(\mathcal{R}_{q+1})_{h} \to \mathcal{R}^{\wedge}_{q+1}$ whose completion is an inverse for $\mathcal{R}^{\wedge}_{q+1} \to [(\mathcal{R}_{q+1})_{h}]^{\wedge}$.
\item[-] If we consider the exact sequence
\[ 0 \to (g) \to (\mathcal{R}_{q+1})_{h} \to (\mathcal{R}_{q+1})_{h}/(g) \to 0 \]
of finite $(\mathcal{R}_{q+1})_{h}$-modules, then via \cite[\href{https://stacks.math.columbia.edu/tag/00MA}{Lemma 00MA}]{stacks-project} its completion is also short exact. In particular we may make the identification 
\[ [(\mathcal{R}_{q+1})_{h}]^{\wedge}/(g) = [(\mathcal{R}_{q+1})_{h}/(g)]^{\wedge} = \mathcal{O}(\mathcal{U})^{\wedge} . \]
But via the just-constructed isomorphism this gives $\mathcal{R}^{\wedge}_{q+1} / (g) = \mathcal{O}(\mathcal{U})^{\wedge}$. Since $g$ is monic, $\mathcal{O}(\mathcal{U})^{\wedge}$ is an integral extension of $\mathcal{R}^{\wedge}_{q}$, but $\mathcal{R}^{\wedge}_{q}$ is integrally closed after possibly increasing $S$ (it is a ufd as soon as $\mathcal{O}_{K,S}$ is a pid \cite[Thm 3.2]{zbMATH03187700}), hence the sequence (\ref{stdetfac2}) collapses and gives an isomorphism $\mathcal{R}^{\wedge}_{q} = \mathcal{O}(\mathcal{U})^{\wedge}$ which sends $x_{i}$ to $f_{i}$.
\end{itemize}

Using the fact that $\mathcal{O}(\mathcal{U})$ is integral (in fact, regular) we may observe that $\mathcal{O}(\mathcal{U}) \to \mathcal{O}(\mathcal{U})^{\wedge}$ is injective (using the Krull intersection theorem for Notherian integral rings). Hence by composition we obtain an injective map
\[ \textrm{res}: \mathcal{O}(\mathcal{U}) \to \mathcal{O}_{K,S}[[x_{1}, \hdots, x_{q}]] \]
which sends a function to its power series expansion in terms of $f_{1}, \hdots, f_{q}$. 

\begin{lem}
\label{keyinjlem}
For every prime $\mathfrak{p}_{(v)}$ of $\mathcal{O}_{K,S}$ and every integer $d$, the map $\textrm{res}^{d}_{(v)}$ obtained from $\textrm{res}$ by base-change along $\mathcal{O}_{K,S} \xrightarrow{c^{d}_{(v)}} \mathcal{O}_{K,(v)} \xrightarrow{e^{d}_{(v)}} \mathcal{O}_{K,(v)}/\mathfrak{p}^{d}_{(v)}$ is injective.
\end{lem}

\begin{proof}
After the first base-change along $c^{d}_{(v)}$ the map $\textrm{res}$ is indentified with
\[ \mathcal{O}(\mathcal{U}_{(v)}) \to \mathcal{O}_{K,(v)}[[x_{1}, \hdots, x_{q}]] \]
where $\mathcal{U}_{(v)}$ is the obvious base-change. The flatness of $c^{d}_{(v)}$ ensures that this map continues to be injective.

Applying the second base-change, we obtain a map 
\begin{equation}
\label{sndbc}
\textrm{res}^{d}_{(v)} = \xi^{d}_{(v)} : \mathcal{O}(\mathcal{U}_{(v)})/\mathfrak{p}^{d}_{(v)} \to \mathcal{O}_{K,S}/\mathfrak{p}^{d}_{(v)}[[x_{1}, \hdots, x_{q}]] .
\end{equation}
By construction, this map factors through an isomorphism 
\[ \varprojlim_{k} \mathcal{O}(\mathcal{U})/(\mathfrak{m}^{k} + \mathfrak{p}^{d}_{(v)}) \simeq \mathcal{O}_{K,(v)}/\mathfrak{p}^{d}_{(v)}[[x_{1}, \hdots, x_{q}]] \]
hence the kernel of $\textrm{res}^{d}_{(v)}$ is the intersection $\bigcap_{k} (\mathfrak{m}^{k} + \mathfrak{p}^{d}_{(v)}) = \bigcap_{k} (\mathfrak{m} + \mathfrak{p}^{d}_{(v)})^{k}$ inside $\mathcal{O}(\mathcal{U}_{(v)})/\mathfrak{p}^{d}_{(v)}$. When $d = 1$ the ring $\mathcal{O}(\mathcal{U}_{(v)})/\mathfrak{p}^{d}_{(v)}$ is integral (taking $S$ large enough), so by the Krull intersection theorem for integral Noetherian rings it follows that $\textrm{res}^{d}_{(v)}$ is injective for $d = 1$. 

We now argue by induction on $d$. The ideal $\mathfrak{p}_{(v)}$ is principal generated by some uniformizer $y_{(v)} \in \mathcal{O}_{K,(v)}$, hence the powers $\mathfrak{p}^{d}_{(v)}$ are also principal generated by $y^{d}_{(v)}$. Suppose that $\xi^{d}_{(v)}$ is injective, and consider some element $a \in \mathcal{O}(\mathcal{U}_{(v)})$ mapping into $\mathfrak{p}^{d+1}_{(v)}[[x_{1}, \hdots, x_{q}]]$ under $\xi_{(v)} := \textrm{res}$. Then, by induction, necessarily $a \in \mathfrak{p}^{d}_{(v)} \mathcal{O}(\mathcal{U}_{(v)})$ by injectivity of $\xi^{d}_{(v)}$, so write $a = y^{d}_{(v)} a'$. Then $\xi_{(v)}(a) = y^{d}_{(v)} \xi_{(v)}(a') \in \mathfrak{p}^{d+1}_{(v)}$, so because $\mathcal{O}_{K,(v)}[[x_{1}, \hdots, x_{q}]]$ is a ufd this means that $\xi_{(v)}(a') \in \mathfrak{p}_{(v)}$ and hence $a' \in \mathfrak{p}_{(v)}$ by the injectivity of $\xi^{1}_{(v)}$.
\end{proof}

\paragraph{Relation to Derivatives:} Now consider the fixed embedding $U \cong \Spec A/I \hookrightarrow \mathbb{A}^{n}$ associated to the adelic norm, where $A = K[x_{1}, \hdots, x_{n}]$. The operator $\partial_{i}$ may be viewed as a section of the tangent sheaf $T U$, which is naturally a subsheaf of the sheaf $\mathcal{O}(U) \otimes_{\mathcal{O}(\mathbb{A}^{n})} T \mathbb{A}^{n}$. This allows us to express each $\partial_{i}$ in the form
\begin{equation}
\label{ambientpartialexp}
\partial_{i} = \sum_{j=1}^{n} a_{ij} \partial^{\textrm{std}}_{j} 
\end{equation}
where $\partial^{\textrm{std}}_{j}$ are the standard partial derivative operators associated to $\mathbb{A}^{n}$ and $a_{ij} \in A/I$. That $\partial_{i}$ is a partial derivative operator on $U$ implies in particular that $\partial_{i}(I) \subset I$. 

By spreading out, we may identify the model $\mathcal{U}$  discussed above with the Zariski closure of $U$ inside $\Spec \mathcal{O}_{K,S}[x_{1}, \hdots, x_{n}]$, increasing $S$ if necessary. The partial derivative operators extend to the coordinate rings of $\mathcal{U}$ and $\mathcal{U}_{(v)}$, again increasing $S$ if necessary, and the power series ring $\mathcal{O}_{K,S}[[y_{1}, \hdots, y_{q}]]$ admits the obvious partial derivative operators which are compatible with the maps $\xi^{d}_{(v)}$ constructed above. 

\paragraph{Deduction of Norm Estimates: }All of the above has resulted in the following: there is, for each $d$ and $v \not\in S$, a differential-operator preserving map 
\[ \gamma^{d}_{(v)} : \mathcal{O}(\mathcal{U}_{(v)}) \to (\mathcal{O}_{K,(v)}/\mathfrak{p}^{d}_{(v)})[[x_{1}, \hdots, x_{q}]] \]
whose kernel is exactly $\mathfrak{p}^{d}_{(v)} \mathcal{O}(\mathcal{U}_{(v)})$. The map $\gamma$ in the statement of the proposition is just the scalar extension to $K$ of $\textrm{res}$. The claim about the ``Gauss norm'' may be reduced (by scaling) to the case where $u \in \mathcal{O}(\mathcal{U}_{(v)})$ where it follows from the injectivity of $\gamma^{1}_{(v)}$. 

Again by scaling, the norm inequality reduces to considering $u \in \mathcal{O}(\mathcal{U}_{(v)})$ with $\| u \|_{o,v} = 1$. In  the case under consideration where we want to take $c_{i,v} = 1$, the claim is simply that the number of powers of $p_{v}$, where $p_{v} \in \mathbb{Z}$ is the prime below $v$, which divide $(\partial_{i})^{k} u$ is at least the number $e_{k,p}$ which divides $k!$. Via $\gamma$, this computation can be carried out in the formal power series ring where it follows from the fact that $k!$ always divides $n (n-1) \cdots (n-k+1)$ for any $n$.

\paragraph{Finitely Many Remaining Places:} For the finitely many remaining places, we note that from (\ref{ambientpartialexp}) the claim that $\| (\partial_{i})^{k} u \|_{o,v} \leq c^{k}_{i,v} \| u \|_{o,v}$ can be obtained with $c_{i,v} = \textrm{max}_{j} \| a_{ij} \|_{v}$ by taking $\widetilde{u}$ to be a minimal representative for $u$ and repeatedly differentiating. The number of powers of $p$ dividing $k!$ is at most $k/(p-1)$, so one obtains the result after increasing $c_{i,v}$. The argument in the affinoid case is analogous.
\end{proof}

\vspace{0.5em}

In what follows we write $\overline{f} = (f_{1}, \hdots, f_{p})$ (i.e., using only the first $p$ coordinates), in contrast to notation elsewhere.

\begin{cor}
\label{retextendscor}
The map $\delta$ from (\ref{flatteningofa}) extends to define a retraction $\beta : \mathcal{T}(\iota, \alpha, \rho, \overline{f}) \to E$ for some choice of $\alpha$ and $\rho$. More precisely, for each $v$ there is a map $\beta^{v} : \mathcal{T}(\iota, \alpha, \rho, \overline{f})^{v} \to E^{v}$ which when restricted to the formal neighbourhood of $E^{v}$ inside $U^{v}$ agrees with the $v$-adic analytification of the map $\mathcal{E} \to E$ from \autoref{formalretractexists}. 
\end{cor}


\begin{proof}
Recall that we have functions $z_{1}, \hdots, z_{m}$ on the affine space $\mathbb{A}^{m}$ which, using $\iota$, we may restrict to functions on $E$. We consider the formal functions $\delta(z_{i})$, which we view as power series in $f_{1}, \hdots, f_{p}$ with coefficients in the coordinate ring $\mathcal{O}(U)$. Then from the formula (\ref{flatteningofa}) as well as the norm estimates in \autoref{normestprop}, the $\mathbf{J}$-coefficient of the formal function $\delta(z_{i})$ has norm of size at most $c_{v}^{|\mathbf{J}|}$, where $c_{v}$ is a constant depending on $v$ with $c_{v} = 1$ for all but finitely many $v$. It follows that we can find $\rho \in K$ so that, as power series in $f_{1}/\rho, \hdots, f_{p}/\rho$, the formal functions $\delta(z_{i})$ have integral coefficients for all $v$.

The above has shown (in a strong form) that the coefficients of the power series $\delta(z_{i})$ are adelic. To show it is $\alpha$-adelic for some $\alpha$ (i.e., power-adelic) it suffices to check the condition in \autoref{adelicsufftocheck}, and this follows from writing out the partial derivative operators in the explicit form (\ref{ambientpartialexp}) and observing that applying the operator $\partial_{i}$ increases the degree of any representative by at most the degree of some $a_{ij}$. Taking a constant $\alpha$ as in \autoref{adandpowbound}, it follows that the power series $\delta(z_{i})$ converge on some adelic neighbourhood defined by the conditions
\begin{equation}
\label{tubedefeqs2}
|(f_{i} / \rho)^{\alpha_{1}} z_{r}^{\alpha_{2}}| < |1| \hspace{3em}    1 \leq r \leq \ell, \hspace{1em} 1 \leq i \leq p , \hspace{1em} \alpha_{1} + \alpha_{2} \leq \alpha , \hspace{1em} \alpha_{1} \geq 1, \hspace{1em} \alpha_{2} \geq 0 .
\end{equation}
After adjusting $\rho$ we see that this implies inequalities of the form (\ref{tubedefeqs}), so the $\delta(z_{i})$ converge on some adelic tube $\mathcal{T}(\iota, \alpha, \rho, \overline{f})$.

To obtain the rigid-analytic map we now take presentations $E^{v} = \varinjlim_{r} \mathbb{E}_{r}$ and $U^{v} = \varinjlim_{r} \mathbb{U}_{r}$ by affinoid neighbourhoods $\mathbb{E}_{r}$ and $\mathbb{U}_{r}$ obtained by intersecting with closed affinoid balls $\mathbb{B}_{r} \subset \mathbb{A}^{n,v}$ at zero, and then consider the maps
\[ \mathbb{U}_{r} \cap \mathcal{T}(\iota, \alpha, \rho, \overline{f})^{v} \to \mathbb{E}_{r} \]
which on the level of structure sheaves send $z_{i} \mapsto \delta(z_{i})$. The universal properties of the Tate algebras $\mathbb{E}_{r}$ inside the category of Banach $K_{v}$-algebras, together with our estimates for the $\delta(z_{i})$ above, ensure this is well-defined.
\end{proof}

\subsubsection{Generalities on Power Series Equations}
\label{powerseriesgensec}

Our next series of constructions will require solving explicit power series equations. The key input is the multivariate Fa\`a di Bruno formula \cite{zbMATH00870188}. The form in which we will need it is as follows. 

\begin{thm}(\cite{zbMATH00870188})
\label{faadibrunothm}
Let $R$ be a $\mathbb{Q}$-algebra, and let $A = (A_{1}, \hdots, A_{\sigma})$ (resp. $B$) be formal power series (resp. a polynomial) with coefficients in $R$, where $A_{i} \in R[[u_{1}, \hdots, u_{\nu}]]$ for $1 \leq i \leq \sigma$ and $B \in R[[z_{1}, \hdots, z_{\sigma}]]$. View partial derivatives of elements of $R[[u_{1}, \hdots, u_{\nu}]]$ and $R[z_{1}, \hdots, z_{\sigma}]$ in the obvious way, i.e., with $\partial_{i}(R) = 0$, $\partial_{i}(u_{i'}) = \delta_{ii'}$ and $\partial_{i}(z_{i'}) = \delta_{ii'}$. 

Suppose additionally one of the following three conditions holds:
\begin{itemize}
\item[(I)] both $B$ and the composition $B \circ A$ are polynomials (they lie in $R[z_{1}, \hdots, z_{\sigma}]$ and $R[u_{1}, \hdots u_{\nu}]$, respectively); 
\item[(II)] we have $A(0) = 0$ and $B \circ A$ is a polynomial, where by $?(0)$ with $?$ a formal power series we mean the constant term; or
\item[(III)] both $A(0) = 0$ and $B(0) = 0$.
\end{itemize}
Then with $C = B \circ A$ we have
\begin{equation}
\label{faadibruno}
 (\partial_{\boldsymbol{J}} C)(B(A(0))) = \sum_{\substack{1 \leq |\boldsymbol{\lambda}| \leq |\boldsymbol{J}|}} (\partial_{\boldsymbol{\lambda}} B)(A(0)) \sum_{s=1}^{|\boldsymbol{J}|} \sum_{\mathcal{C}_{s}(\boldsymbol{J}, \boldsymbol{\lambda})} \boldsymbol{J}! \prod_{a = 1}^{s} \frac{[\boldsymbol{A}_{\boldsymbol{\ell}_{a}}(0)]^{\boldsymbol{k}_{a}}}{(\boldsymbol{k}_{a}!)[\boldsymbol{\ell}_{a}!]^{|\boldsymbol{k}_{a}|}} .
 \end{equation}
In this expression, we have used the following notation:
\begin{itemize}
\item[-] the vectors $\boldsymbol{\lambda}$ and $\boldsymbol{k}_{a}$ come from $\mathbb{Z}^{\sigma}_{\geq 0}$ and the vectors $\boldsymbol{J}$ and $\boldsymbol{\ell}_{a}$ come from $\mathbb{Z}^{\nu}_{\geq 0}$;
\item[-] for any vector $\boldsymbol{t} = (t_{1}, \hdots, t_{r}) \in (\mathbb{Z}_{\geq 0})^{r}$ we have $|\boldsymbol{t}| = t_{1} + \cdots + t_{r}$;
\item[-] we have
\begin{equation}
\label{Csdef}
\mathcal{C}_{s}(\boldsymbol{J}, \boldsymbol{\lambda}) = \left\{ (\boldsymbol{k}_{1}, \hdots, \boldsymbol{k}_{s} ; \boldsymbol{\ell}_{1}, \hdots, \boldsymbol{\ell}_{s}) : \substack{|\boldsymbol{k}_{i}| > 0 , \hspace{1em} \boldsymbol{0}  \prec \boldsymbol{\ell}_{1} \boldsymbol \prec \cdots \prec \boldsymbol{\ell}_{s} \\ \sum_{i=1}^{s} \boldsymbol{k}_{i} = \boldsymbol{\lambda} \textrm{ and } \sum_{i=1}^{s} |\boldsymbol{k}_{i}| \boldsymbol{\ell}_{i} = \boldsymbol{J} } \right\} , 
\end{equation}
\item[-] for vectors $\boldsymbol{t} = (t_{1}, \hdots, t_{r})$ and $\boldsymbol{t}' = (t'_{1}, \hdots, t'_{r})$, the symbol $\boldsymbol{t} \prec \boldsymbol{t}'$ means that one of the following conditions holds:
\begin{itemize}
\item[(i)] $|\boldsymbol{t}| < |\boldsymbol{t}'|$;
\item[(ii)] $|\boldsymbol{t}| = |\boldsymbol{t}'|$ and $t_{1} < t'_{1}$; or
\item[(iii)] $|\boldsymbol{t}| = |\boldsymbol{t}'|$, $t_{1} = t'_{1}, \hdots, t_{k} = t'_{k}$ and $t_{k+1} < t'_{k+1}$ for some $1 \leq k < r$;
\end{itemize}
\item[-] the notation $\boldsymbol{A}_{\boldsymbol{\ell}}$ for $\boldsymbol{\ell} = (\ell_{1}, \hdots, \ell_{\nu})$ means $(\partial_{\boldsymbol{\ell}} A_{1}, \hdots, \partial_{\boldsymbol{\ell}} A_{\sigma})$; and
\item[-] for a vector $\boldsymbol{t} = (t_{1}, \hdots, t_{r})$, we have $\boldsymbol{t}! = t_{1}! \cdots t_{r}!$.
\end{itemize}
\end{thm}

The authors in \cite{zbMATH00870188} work in the context where $A$ and $B$ are tuples of smooth functions and do discuss convergence, but it ultimately makes no difference in applying this formula whether the objects $A, B$ and $C$ are genuine functions or merely formal, provided the composition makes sense. The conditions (I), (II) and (III) are simply three different scenarios where the compositions make sense: in each case there is no issue computing the power series $C = B \circ A$, and making sense of the expressions $?(A(0))$, $?(B(A(0))$, etc. Note also that if given fixed $B$ and $C$ as in either (I), (II) or (III) we construct power series $A = (A_{1}, \hdots, A_{\sigma}) \in R[[u_{1}, \hdots, u_{v}]]^{\sigma}$ for which (\ref{faadibruno}) holds for all $\mathbf{J}$, then necessarily $C = B \circ A$. 

\paragraph{Solving for $A$:} We now suppose that we have a collection of functions $B_{1}, \hdots, B_{c}, C_{1}, \hdots, C_{c}$ such that $C_{r} = B_{r} \circ A$ for all $1 \leq r \leq c$. If $B_{r}$ is a polynomial we set $\deg B_{r}$ to be its degree, and otherwise set $\deg B_{r} = \infty$. We set $A(0) = \boldsymbol{e}$. Then fixing an index $r$ and multi-index $\mathbf{J}$ we obtain from (\ref{faadibruno}) the equation

\makebox[\displaywidth]{\parbox{0.10\textwidth}{
\begin{tabular}{cc}
\parbox{13cm}{\begin{align*}
\sum_{i=1}^{\sigma} (\partial_{i} B_{r})(\boldsymbol{e}) \frac{A_{i,\boldsymbol{J}}(0)}{\boldsymbol{J}!} & =  \frac{(\partial_{\boldsymbol{J}} C_{r})(B(\boldsymbol{e}))}{\mathbf{J}!} \\ 
& - \underbrace{\sum_{2 \leq |\boldsymbol{\lambda}| \leq \min\{ \deg B_{r}, |\mathbf{J}| \}} (\partial_{\boldsymbol{\lambda}} B_{r})(\boldsymbol{e}) \sum_{s=1}^{|\boldsymbol{J}|} \sum_{\mathcal{C}_{s}(\boldsymbol{J},\boldsymbol{\lambda})} \left( \prod_{a=1}^{s} \frac{[\boldsymbol{A}_{\boldsymbol{\ell}_{a}}(0)]^{\boldsymbol{k}_{a}}}{(\boldsymbol{k}_{a}!) [\boldsymbol{\ell}_{a}!]^{|\boldsymbol{k}_{a}|}} \right)}_{P_{\boldsymbol{J},r}} .
\end{align*}} & 
\vbox{\begin{equation} \label{lineqforA} \end{equation}}
\end{tabular}
}}

Now let $\mathbf{B}$ be the $c \times \sigma$ matrix with entries $(\partial_{i} B_{r})(\boldsymbol{e}) \in R$; note this matrix is independent of $\mathbf{J}$. We make a few simple observations.

\begin{prop}
\label{linsolBprop}
If $\mathbf{B}$ admits an invertible $\sigma \times \sigma$ minor $\mathbf{M}$, then $A$ is uniquely determined by the equations $C_{r} = B_{r} \circ A$, where $r$ ranges over the subset $\mathcal{M} \subset \{ 1, \hdots, c \}$ which indexes the rows of $\mathbf{M}$. Moreover one has the equation
\begin{equation}
\label{linsyssol}
\frac{1}{\mathbf{J}!} \begin{bmatrix} A_{1,\mathbf{J}}(0) \\ \vdots \\ A_{\sigma,\mathbf{J}}(0) \end{bmatrix} = \mathbf{M}^{-1} \left( \frac{1}{\mathbf{J}!} \begin{bmatrix} (\partial_{\boldsymbol{J}} C_{i_{1}})(B(\boldsymbol{e})) \\ \vdots \\ (\partial_{\boldsymbol{J}} C_{i_{\sigma}})(B(\boldsymbol{e})) \end{bmatrix} - \begin{bmatrix} P_{\boldsymbol{J}, i_{1}} \\ \vdots \\ P_{\boldsymbol{J}, i_{\sigma}} \end{bmatrix} \right) , 
\end{equation}
where $i_{1} < \cdots < i_{\sigma}$ and $\mathcal{M} = \{ i_{1}, \hdots, i_{\sigma} \}$. 
\end{prop}

From now on we suppose that $R$ is a finite-type, geometrically integral and geometrically reduced $K$-algebra, with $K$ a number field, and we fix a presentation $R = K[w_{1}, \hdots, w_{q}]/I$ of $R$ such that $\Spec R$ contains $0 \in \mathbb{A}^{q}$. We will consider representatives of elements of $R$ with respect to this presentation, as well as non-archimedean norms on $R$ of the form (\ref{nonarchnorm}). 

\begin{lem}
\label{linsolBprop3}
Suppose in the situation of \autoref{linsolBprop} all of the $B_{r}$ and $C_{r}$ are polynomials. Then there exist integers $\alpha$ and $\beta$ such that the $A_{i,\mathbf{J}}(0)$ admit representatives satisfying $\deg A_{i,\mathbf{J}}(0) \leq |\mathbf{J}| \alpha + \beta$ whenever $|\mathbf{J}| \geq 0$.
\end{lem}

\begin{proof}
We consider the following elements of $R$: 
\begin{itemize}
\item[(1)] $(\partial_{\boldsymbol{\lambda}} B_{r})(\boldsymbol{e})$ for $1 \leq r \leq c$ and $2 \leq |\boldsymbol{\lambda}| \leq \deg B_{r}$;
\item[(2)] $\partial_{\mathbf{J}} C_{r}(B(\mathbf{e}))$ for $1 \leq r \leq c$ and $|\mathbf{J}| \leq \deg C_{r}$; and
\item[(3)] the entries of $\mathbf{M}^{-1}$.
\end{itemize}
Let $\mathcal{S}$ be the set of these elements. We will construct polynomials $Q_{i,\mathbf{J}} \in \mathbb{Q}[\mathcal{S}]$ such that $Q_{i,\mathbf{J}}(\mathcal{S}) = A_{i,\mathbf{J}}(0)$ and give an upper bound for the degree of $Q_{i,\mathbf{J}}$. Since there are finitely many elements of $\mathcal{S}$, this will then lead to an upper bound on the degree of a representative of $A_{i,\mathbf{J}}(0)$ after choosing representatives for the elements of $\mathcal{S}$.

We define the $Q_{i,\mathbf{J}}$ inductively as follows. When $|\mathbf{J}| = 1$ we define $Q_{i,\mathbf{J}}$ by the expression (\ref{linsyssol}), noting that $P_{\mathbf{J},r} = 0$ for all $r$ in this case. We then regard each $P_{\mathbf{J},r}$ when $|\mathbf{J}| \geq 2$ as a polynomial in $\mathcal{S}$ and and the $Q_{i,\boldsymbol{\ell}}$ for $|\boldsymbol{\ell}| < |\mathbf{J}|$ where the polynomial is given by the expression in (\ref{lineqforA}), substituting $Q_{i,\boldsymbol{\ell}}$ for $A_{i,\boldsymbol{\ell}}(0)$. Then $Q_{i,\mathbf{J}}$ is defined in general using (\ref{linsyssol}) as a polynomial in $\mathcal{S}$ and the $P_{\mathbf{J},r}$, and hence $\mathcal{S}$ and the $Q_{i,\boldsymbol{\ell}}$ for $|\boldsymbol{\ell}| < |\mathbf{J}|$. Since inductively $Q_{i,\boldsymbol{\ell}}$ is a polynomial in $\mathcal{S}$, so is $Q_{i,\boldsymbol{J}}$. In the remainder of the argument we regard $Q_{i,\mathbf{J}}$ and $P_{\mathbf{J},r}$ as polynomials in $\mathcal{S}$, and use the notation $\deg_{\mathcal{S}}(-)$ to denote their degree.

We claim that $\deg_{\mathcal{S}} Q_{i,\mathbf{J}}$ is at most $5 |\boldsymbol{J}| - 3$ for all $|\boldsymbol{J}| \geq 1$. When $|\boldsymbol{J}| = 1$ this follows from the formula (\ref{linsyssol}). Otherwise, observe the following.
\begin{itemize}
\item[-] From the expression (\ref{linsyssol}), one obtains that $\deg_{\mathcal{S}} Q_{i,\mathbf{J}} \leq 2 + \max_{r} \deg_{\mathcal{S}} P_{\mathbf{J},r}$.
\item[-] From the expression defining $P_{\mathbf{J},r}$ in (\ref{lineqforA}), one obtains that 
\[ \deg_{\mathcal{S}} P_{\mathbf{J}, r} \leq 1 + \max_{\boldsymbol{\lambda}, s, \mathcal{C}_{s}(\mathbf{J}, \boldsymbol{\lambda})} \deg_{\mathcal{S}} \underbrace{\left( \prod_{a=1}^{s} \frac{(Q_{1, \boldsymbol{\ell}_{a}}, \hdots, Q_{\sigma, \boldsymbol{\ell}_{a}})^{\boldsymbol{k}_{a}}}{(\boldsymbol{k}_{a}!) [\boldsymbol{\ell}_{a}!]^{|\boldsymbol{k}_{a}|}} \right)}_{(\star)} . \]
Moreover by expanding the term labelled $(\star)$ one has that
\[ \deg_{\mathcal{S}} (\star) \leq \sum_{a=1}^{s} \sum_{j=1}^{\sigma} k_{a,j} \deg_{\mathcal{S}} Q_{j,\boldsymbol{\ell}_{a}} . \]
\item[-] Combining the two previous points our analysis shows that, for each $\mathbf{J}$, there exists $\boldsymbol{\lambda}$, $s$, and elements of $\mathcal{C}_{s}(\boldsymbol{J}, \boldsymbol{\lambda})$ such that
\[ \deg_{\mathcal{S}} Q_{i,\boldsymbol{J}} \leq 3 + \sum_{a=1}^{s} \sum_{j=1}^{\sigma} k_{a,j} \deg_{\mathcal{S}} Q_{j,\boldsymbol{\ell}_{a}} . \]
Applying the induction hypothesis, this gives us that 
\[ \deg_{\mathcal{S}} Q_{i,\boldsymbol{J}} \leq 3 + 5 \underbrace{\left( \sum_{a=1}^{s} \sum_{j=1}^{\sigma} k_{a,j} |\boldsymbol{\ell}_{a}| \right)}_{(\square)} - 3 \underbrace{\left( \sum_{a=1}^{s} \sum_{j=1}^{\sigma} k_{a,j} \right)}_{(\triangle)} . \]
\item[-] Recall from the definition of $\mathcal{C}_{s}(\mathbf{J}, \boldsymbol{\lambda})$ that $\sum_{a=1}^{s} |\boldsymbol{k}_{a}| \boldsymbol{\ell}_{a} = \boldsymbol{J}$. Thus, the summand labelled $(\square)$ is $|\mathbf{J}|$. On the other hand, we claim that $(\triangle) \geq 2$. Indeed, since $|\mathbf{k}_{a}| > 0$ always, the only potential issue would be if $s = 1$, but then $\boldsymbol{k}_{1} = \boldsymbol{\lambda}$ and we know that $|\boldsymbol{\lambda}| \geq 2$ from the expression defining $P_{\mathbf{J},r}$ in (\ref{lineqforA}). Substituting we find that
\[ \deg_{\mathcal{S}} Q_{i,\mathbf{J}} \leq 3 + 5 |\mathbf{J}| - 3(2) = 5 |\mathbf{J}| - 3 \]
as desired.
\end{itemize}
We learn from this that there exists constants $\alpha, \beta$ such that $\deg_{\mathcal{S}} Q_{i,\mathbf{J}} \leq \alpha |\mathbf{J}| + \beta$ for all $|\mathbf{J}| \geq 0$. Now by construction, $Q_{i,\mathbf{J}}(\mathcal{S}) = A_{i,\mathbf{J}}(0)$. Choosing representatives in $K[w_{1}, \hdots, w_{q}]$ for the elements of the finite set $\mathcal{S}$, it follows that $\deg A_{i,\mathbf{J}}(0) \leq \alpha |\mathbf{J}| + \beta$ after adjusting $\alpha$ and $\beta$.
\end{proof}

\begin{lem}
\label{radiusconvestcor}
Suppose in the situation of \autoref{linsolBprop} all of the $B_{r}$ and $C_{r}$ are polynomials, and write $\| \cdot \|_{v}$ for the non-archimedean norm on $R$ associated to the fixed presentation which extends a non-archimedean norm $| \cdot |_{v}$ on $K$. Then there exists $\tau$ such that $\| A_{i,\mathbf{J}}(0) \|_{v} / \tau^{|\mathbf{J}|+1} \leq 1$ for all $\mathbf{J},i$. For all but finitely many $v$ one may take $\tau = 1$. 
\end{lem}


\begin{proof}
Let $\mathcal{S}$ be defined as in the argument of \autoref{linsolBprop3}. In that proof we considered polynomials $Q_{i,\mathbf{J}}$ in the elements of $\mathcal{S}$ with coefficients in $\mathbb{Q}$, and which had the property that $Q_{i,\mathbf{J}}(\mathcal{S}) = A_{i,\mathbf{J}}(0)$. We then gave an upper bound on the degree of $Q_{i,\mathbf{J}}$ in terms of $|\mathbf{J}|$. Now we additionally argue that the coefficients of these polynomials actually lie in $\mathbb{Z}$. We argue by induction, with the case of $|\mathbf{J}| = 1$ immediate from the definitions.

First, observe that when solving for $A_{i,\mathbf{J}}$ using (\ref{linsyssol}) we may first clear the denominator $\mathbf{J}!$ and reduce to showing that the coefficients of $\mathbf{J}!\, P_{\mathbf{J},r}$ are integral for each $r$. In particular it suffices to consider the summands
\begin{equation}
\label{Jfacmulparenexp}
\boldsymbol{J}! \left( \prod_{a=1}^{s} \frac{(Q_{1, \boldsymbol{\ell}_{a}}, \hdots, Q_{\sigma, \boldsymbol{\ell}_{a}})^{\boldsymbol{k}_{a}}}{(\boldsymbol{k}_{a}!) [\boldsymbol{\ell}_{a}!]^{|\boldsymbol{k}_{a}|}} \right)
\end{equation}
as polynomials in $\mathcal{S}$ and show that their coefficients are integral. 

Now recall from (\ref{Csdef}) that $\sum_{a=1}^{s} |\boldsymbol{k}_{a}| \boldsymbol{\ell}_{a} = \boldsymbol{J}$. Entrywise, this expression implies that $\sum_{a=1}^{s} |\boldsymbol{k}_{a}| \ell_{a,i} = J_{i}$ for all $1 \leq i \leq \nu$. Set $R_{a,i} = |\boldsymbol{k}_{a}| \ell_{a,i}$. We claim that in fact $\prod_{a=1}^{s} (\mathbf{k}_{a} !) [\boldsymbol{\ell}_{a}]^{|\mathbf{k}_{a}|}$ divides $\mathbf{J}!$. As multinomial coefficients are integral, $\mathbf{k}_{a}!$ divides $|\mathbf{k}_{a}|!$, hence it suffices to show that for each $i$ the product $\prod_{a=1}^{s} |\mathbf{k}_{a}|! (\ell_{a,i} !)^{|\boldsymbol{k}_{a}|}$ divides $J_{i}!$. Again since multinomial coefficients are integral, the product $\prod_{a=1}^{s} R_{a,i}!$ divides $J_{i}!$, so it suffices to show that $|\mathbf{k}_{a}|! (\ell_{a,i} !)^{|\boldsymbol{k}_{a}|}$ divides $R_{a,i}!$ for all $a$ and $i$. This one can see by partitioning $R_{a,i}$ objects into $|\mathbf{k}_{a}|$ subsets of size $\ell_{a,i}$ and observing that $|\mathbf{k}_{a}|! (\ell_{a,i} !)^{|\mathbf{k}_{a}|}$ is the size of the symmetry group which preserves this partition.

 From this it follows that (\ref{Jfacmulparenexp}) can be rewritten as 
\[ ? \left( \prod_{a=1}^{s} (Q_{1, \boldsymbol{\ell}_{a}}, \hdots, Q_{\sigma, \boldsymbol{\ell}_{a}})^{\boldsymbol{k}_{a}} \right) \]
where $?$ is an integer. Since the various $Q_{i, \boldsymbol{\ell}_{a}}$ have integral coefficients by induction, so do the $Q_{i,\mathbf{J}}$. 

Now $\log_{p} \| \mathbf{A}_{i, \mathbf{J}}(0) \|_{v} = \log_{p} \| Q_{i,\mathbf{J}}(\mathcal{S}) \|_{v}$. Since $\mathcal{S}$ is a finite set, the norm of each element of $\mathcal{S}$ is at most some constant $\kappa$. Since the degree of $Q_{i,\mathbf{J}}$ is at most $\alpha |\mathbf{J}| + \beta$ as we saw in the proof of \autoref{linsolBprop3}, it follows that this norm is at most $\kappa^{\alpha |\mathbf{J}| + \beta}$. The result follows by taking $\tau = \max \{ \kappa^{\beta}, \kappa^{\alpha} \}$. 
\end{proof}



\subsubsection{Local Construction of Tubes}

\begin{prop}
\label{tubeconstrprop}
~  
\begin{itemize}
\item[(A)] There exists a unique isomorphism $\eta$ of formal schemes
\[ \Spf \varprojlim_{r} \mathcal{O}_{\mathbb{A}^{p} \times E} / (u_{1}, \hdots, u_{p})^{r} \xrightarrow{\sim} \mathcal{E} \]
inducing the identity map on $E$, sending $f_{i}$ to $u_{i}$ for $1 \leq i \leq p$, and sending $f_{p+1}, \hdots, f_{q}$ to the images of the restrictions $\restr{f_{p+1}}{E}, \hdots, \restr{f_{q}}{E}$, viewed inside the coordinate ring of the left-hand side via the map $\mathcal{O}_{E} \to \varprojlim_{r} \mathcal{O}_{\mathbb{A}^{p} \times E} / (u_{1}, \hdots, u_{p})^{r}$.
\end{itemize}
We note that the statement of (A) is local, hence it holds, via gluing, even if one does not assume $U$ or $E$ is affine.
\begin{itemize}

\item[(B)] For each finite place $v$ of $K$, the map $\eta$ extends to an open analytic embedding $\eta^{v} : E\langle p \rangle^{v} \to U^{v}$ where $E\langle p \rangle$ is some standard adelic tube.

\item[(C)] Fix an adelic tube $\mathcal{T}$ associated to $(U, E)$. Then after possibly refining $E\langle p \rangle$, the maps $\eta^{v}$ in (B) factor through $\mathcal{T}^{v} \subset U^{v}$. Moreover there is a refinement $\mathcal{T}'$ of $\mathcal{T}$ such that the image of $E\langle p \rangle^{v}$ contains $\mathcal{T}'^{v}$ for all $v$.
\end{itemize}
\end{prop}

We suppose that $U \subset \mathbb{A}^{m}$ is defined as the vanishing locus of functions $g_{1}, \hdots, g_{k} \in K[z_{1}, \hdots, z_{m}]$. With respect to such a presentation, the map of \autoref{tubeconstrprop} uniquely determines a collection of $m$ power series $G_{1}, \hdots, G_{m} \in \varprojlim_{r} \mathcal{O}_{\mathbb{A}^{p} \times E} / (u_{1}, \hdots, u_{p})^{r}$. We will also prove:

\begin{prop}
\label{definingpowerseriesadelicprop}
The power series $G_{1}, \hdots, G_{m}$ are power-adelic.
\end{prop}

\vspace{0.5em}

 We begin by fixing functions $h_{1}, \hdots, h_{\mu} \in K[z_{1}, \hdots, z_{m}]$ with the following property: the standard open sets $D(h_{1}), \hdots, D(h_{\mu})$ give an open cover of $\mathbb{A}^{m}$, and on each $D(h_{i})$ there exists a subset $\{ g_{i_{1}}, \hdots, g_{i_{p}} \}$ of $\{ g_{1}, \hdots, g_{k} \}$ of size $\tau = \codim_{\mathbb{A}^{m}} U$ such that $D(h_{i}) \cap U$ is the vanishing locus of $\{ g_{i_{1}}, \hdots, g_{i_{\tau}} \}$. That this can be achieved is a consequence of the fact that $U$ is smooth, hence a local complete intersection. We will construct the maps $\eta$ and $\eta^{v}_{o}$ locally after restricting to $D(h_{i})$, and obtain the general result by gluing. 

We fix embeddings $\iota_{i} : D(h_{i}) \hookrightarrow \mathbb{A}^{m+1}$, constructed in the obvious way: we add a variable $z_{m+1}$ and the define the image by $z_{m+1} h_{i} = 1$. Write $\widetilde{E}$ (resp. $\widetilde{U}$) for the image of $D(h_{i}) \cap E$ (resp. $D(h_{i}) \cap U$) under this embedding. We also write $\widetilde{\mathcal{E}}$ for $\mathcal{E} \cap D(h_{i})$, again viewed via its image under $\iota_{i}$. We will fix such an $i$ for now, allowing us to suppress the index, which won't be relevant until later. For ease of notation, we will temporarily write $g_{1}, \hdots, g_{\tau}$ for $g_{i_{1}}, \hdots, g_{i_{\tau}}$ and temporarily set $g_{\tau+1} = z_{m+1} h_{i} - 1$.

\vspace{0.5em}

We will construct a sequence $A_{1}, \hdots, A_{m+1} \in \varprojlim_{r} \mathcal{O}_{\mathbb{A}^{p} \times \widetilde{E}} / (u_{1}, \hdots, u_{p})^{r}$ of elements which give a map to $\mathbb{A}^{m+1}$ which factors through $\widetilde{\mathcal{E}}$. This will define our map $\eta$ on the formal completion of $\widetilde{E}$ inside $\mathbb{A}^{p} \times \widetilde{E}$. We write
\[ A_{i}(u_{1}, \hdots, u_{r}) = \sum_{\mathbf{J}} A_{i,\mathbf{J}} \overline{u}^{\mathbf{J}} , \] 
where $A_{i,\mathbf{J}} \in \Gamma(\widetilde{E}, \mathcal{O}_{\widetilde{E}})$, the sum ranges over all appropriate compositions of integers, and we use multi-index notation to exponentiate the multi-vector $\overline{u} = (u_{1}, \hdots, u_{r})$. Let $\widetilde{e}_{1}, \hdots, \widetilde{e}_{m+1} \in \mathcal{O}_{\widetilde{E}}$ be the functions corresponding to the embedding $\widetilde{E} \hookrightarrow \mathbb{A}^{m+1}$. We will consider formal functions $A_{1}, \hdots, A_{m+1}$ which satisfy the following conditions:
\begin{align}
\label{cond1}
g_{r} (A_{1}, \hdots, A_{m+1}) &= 0 \hspace{4em} &1 \leq r \leq \tau+1 \\
\label{cond2}
f_{t} (A_{1}, \hdots, A_{m+1}) &= u_{t} \hspace{4em} &1 \leq t \leq p \\
\label{cond3}
f_{t} (A_{1}, \hdots, A_{m+1}) &= \restr{f_{t}}{\widetilde{E}} \in \varprojlim_{r} \mathcal{O}_{\mathbb{A}^{p} \times \widetilde{E}} / (u_{1}, \hdots, u_{p})^{r} \hspace{4em} &p+1 \leq t \leq q \\
\label{cond4}
A_{i,\boldsymbol{0}} &= \widetilde{e}_{i} \hspace{4em} &1 \leq i \leq m+1 .
\end{align}
The condition (\ref{cond1}) is just the claim that the map 
\[ \alpha : K[z_{1}, \hdots, z_{m+1}] \to \varprojlim_{r} \mathcal{O}_{\mathbb{A}^{p} \times \widetilde{E}} / (u_{1}, \hdots, u_{p})^{r} \]
induced by $A_{1}, \hdots, A_{m+1}$ factors through $B := K[z_{1}, \hdots, z_{m+1}]/(g_{1}, \hdots, g_{\tau+1})$, and hence defines a map $\Spf \varprojlim_{r} \mathcal{O}_{\mathbb{A}^{p} \times \widetilde{E}} / (u_{1}, \hdots, u_{p})^{r} \to \widetilde{U}$. Given this, condition (\ref{cond2}) makes sense, since different lifts of the $f_{t}$ to functions on $\mathbb{A}^{m}$ will differ by elements in the ideal $(g_{1}, \hdots, g_{k+1})$. One likewise makes sense of (\ref{cond3}) as the condition that viewing $f_{t}(A_{1}, \hdots, A_{m})$ as a power series with coefficients in $\mathcal{O}_{\widetilde{E}}$ yields only a constant term whose value is the image of $f_{t}$ in $\mathcal{O}_{\widetilde{E}}$. Finally (\ref{cond4}) ensures that $\{ 0 \} \times \widetilde{E}$ maps identically onto $\widetilde{E}$.

Let $\beta_{r}$ be the induced map $B \to \mathcal{O}_{\mathbb{A}^{p} \times E} / (u_{1}, \hdots, u_{p})^{r}$. Condition (\ref{cond2}) implies that $(f_{1}, \hdots, f_{p})^{r}$ is in the kernel of $\beta_{r}$ for each $r$, which then implies $\alpha$ must also factor through $\Gamma(\widetilde{\mathcal{E}}, \mathcal{O}_{\widetilde{\mathcal{E}}}) = \varprojlim_{r} B/(f_{1}, \hdots, f_{p})^{r}$.  So solving these equations gives a well-defined map 
\[ \eta : \Spf \varprojlim_{r} \mathcal{O}_{\mathbb{A}^{p} \times \widetilde{E}} / (u_{1}, \hdots, u_{p})^{r} \to \widetilde{\mathcal{E}} . \]
It is clear from our discussion and the definitions that such a map would satisfy the properties listed in \autoref{tubeconstrprop}(A) with $E = \widetilde{E}$, except possibly uniqueness.

\begin{prop}
\label{solconstr}
~
\begin{itemize}
\item[(1)] There exists a unique solution $(A_{1}, \hdots, A_{m+1})$ satisfying (\ref{cond1}), (\ref{cond2}), (\ref{cond3}) and (\ref{cond4}). 

\item[(2)] Write $\| \cdot \|_{o,v} = \| \cdot \|_{(\widetilde{E}, o),v}$ for an adelic norm (resp. a norm) as in \S\ref{polyalgnormsec} associated to the point $o \in \widetilde{E}(K)$ (resp. $o \in \widetilde{E}(K_{v})$). In the adelic case, for all but finitely many non-archimedean places $v$ of $K$ the coefficients of the $A_{i}$ have norm at most $1$ in the norm $\| \cdot \|_{o,v}$. In the case where $v$ is fixed, the norm is bounded by $\kappa^{|\mathbf{J}| + 1}$, where $\kappa > 0$ is some fixed constant depending on $v$ and $o$ and $|\mathbf{J}|$ is the order of the coefficient indexed by the multivector $\mathbf{J}$.

\item[(3)] There exists an integer $\alpha$ such that, for each $i$, the $\overline{u}^{\overline{a}}$ coefficient of $A_{i}$ admits a polynomial representative under the embedding $\iota$ with degree at most $|a| \alpha + \beta$ for some constant $\beta$.
\end{itemize}
\end{prop}

\begin{proof}[Proof of (\ref{solconstr}):] ~

\vspace{0.5em}

\noindent Set $\sigma = m+1$ and $\nu = p$, and work with power series and polynomials having coefficients in $R = \Gamma(\widetilde{E}, \mathcal{O}_{\widetilde{E}})$. Note that elements of $K[z_{1}, \hdots, z_{m+1}]$ may also be interpreted as having coefficients in $R$ by using the embedding $K \hookrightarrow R$. Viewing these polynomials this way is merely a formal device used to assist in the computations, and should not be taken to have a geometric meaning. We write $\widetilde{e}_{i}$ for image of $z_{i}$ in $R$.

Fix lifts $\widetilde{f}_{1}, \hdots, \widetilde{f}_{q}$ of the functions $f_{1}, \hdots, f_{q}$ to $K[z_{1}, \hdots, z_{m+1}]$, and let $\overline{f}_{1}, \hdots, \overline{f}_{q}$ be the images of $f_{1}, \hdots, f_{q}$ in $R = \Gamma(\widetilde{E}, \mathcal{O}_{\widetilde{E}})$. Set $\{ B_{1}, \hdots, B_{c} \} = \{ g_{1}, \hdots, g_{\tau+1}, \widetilde{f}_{1}, \hdots, \widetilde{f}_{p}, \widetilde{f}_{p+1}, \hdots, \widetilde{f}_{q} \}$ and $\{ C_{1}, \hdots, C_{c} \} = \{ 0, \hdots, 0, u_{1}, \hdots, u_{p}, \overline{f}_{p+1}, \hdots, \overline{f}_{q} \}$. We regard the coefficients of the $B_{i}$ inside $R$, and similarly regard $C_{1}, \hdots, C_{c}$ as polynomials in $u_{1}, \hdots, u_{p}$ with coefficients in $R$. Let $\widetilde{\boldsymbol{e}} = (\widetilde{e}_{1}, \hdots, \widetilde{e}_{m+1})$, and set $A(0) = \widetilde{\boldsymbol{e}}$. 

Now each of the conditions (\ref{cond1}), (\ref{cond2}), (\ref{cond3}) is a condition of the form $C_{r} = B_{r} \circ A$, so we are in the setup of \S\ref{powerseriesgensec}. Now recall that $\widetilde{U} = D(h_{i}) \cap U$ was constructed so that $g_{1}, \hdots, g_{\tau}$ with $\tau = \codim_{\mathbb{A}^{m}} U$ succeed to define $\widetilde{U}$ inside $D(h_{i})$. We then observe that $(\tau+1) + q = m+1$, and the differentials
\[ dg_{t}, \hspace{2em} 1 \leq t \leq \tau+1, \hspace{3em} d\widetilde{f}_{t} \hspace{2em} 1 \leq t \leq q \]
are all independent along $\widetilde{U}$, and therefore along $\widetilde{E}$: indeed, $g_{1}, \hdots, g_{\tau+1}$ defines the smooth variety $\widetilde{U}$ as a complete intersection, and $\widetilde{f}_{1}, \hdots, \widetilde{f}_{q}$ give \'etale coordinates on $\widetilde{U}$. It follows that the matrix $\mathbf{B}$ appearing in \autoref{linsolBprop} has an invertible $\sigma \times \sigma$ minor, so we may conclude by applying \autoref{linsolBprop}, \autoref{linsolBprop3}, and \autoref{radiusconvestcor} to our situation.
\end{proof}

Let us now write $E_{i} = D(h_{i}) \cap E$. By our previous reasoning, \autoref{solconstr} has constructed maps $\eta_{i} : \Spf \varprojlim_{r} \mathcal{O}_{\mathbb{A}^{p} \times E_{i}} / (u_{1}, \hdots, u_{p})^{r} \xrightarrow{\sim} \mathcal{E} \cap D(h_{i})$, and the uniqueness statement in \autoref{solconstr} verifies there is a unique such $\eta_{i}$. As we could have done the same argument with $h_{i}$ replaced by $h_{i} h_{i'}$, one also concludes there is a unique map $\eta_{ii'} : \Spf \varprojlim_{r} \mathcal{O}_{\mathbb{A}^{j} \times (E_{i} \cap E_{i'})} / (u_{1}, \hdots, u_{p})^{r} \xrightarrow{\sim} \mathcal{E} \cap D(h_{i}) \cap D(h_{i'})$ satisfying the properties in \autoref{tubeconstrprop}(A); in particular, the maps $\eta_{i}$ for varying $i$ glue to a unique map $\eta : \Spf \mathcal{O}_{\mathbb{A}^{p} \times E} / (u_{1}, \hdots, u_{p})^{r} \xrightarrow{\sim} \mathcal{E}$.

\paragraph{Proof of \autoref{definingpowerseriesadelicprop}:} \autoref{solconstr} has constructed, for each index $i$, power series $A^{i}_{1}, \hdots, A^{i}_{m+1}$ which are power-adelic with coefficients in the coordinate ring of $D(h_{i}) \cap E$. From the uniqueness in \autoref{tubeconstrprop}(A), it follows that in each case the power series $A^{i}_{1}, \hdots, A^{i}_{m}$ are just the restrictions of a common set of power series $G_{1}, \hdots, G_{m}$. Since the power series $A^{i}_{1}, \hdots, A^{i}_{m}$ are adelic by \autoref{solconstr}(2), it follows from \autoref{normsmallerthanreponsomeopen} that the power series $G_{1}, \hdots, G_{m}$ are adelic. 

To verify the additional power-boundedness claim we may now use \autoref{adelicsufftocheck}. In particular it suffices to show that the coefficients $g_{\mathbf{o}}$ with $\mathbf{o} = (o_{1}, \hdots, o_{p})$ of some fixed $G = G_{?}$ admit representatives whose $v$-adic size (resp. degree) grow at most exponentially (resp. linearly) in $o := |\mathbf{o}|$. Fix, for some $i$, representatives $\widetilde{g}_{i,\mathbf{o}}$ for the coefficients of $A^{i}_{?}$ with $v$-adic size (resp. degree) growing at most exponentially (resp. linearly) in $o$. Then using \autoref{getfracreplem} we obtain for each $\widetilde{g}_{i,\mathbf{o}}$ representatives $e_{i,\mathbf{o}} / h^{n_{i,\mathbf{o}}}_{i}$ with $e_{i,\mathbf{o}} \in K[z_{1}, \hdots, z_{m}]$, and such that the $v$-adic size of $e_{i,\mathbf{o}}$ (resp. the degree of $e_{i,\mathbf{o}}$ and the number $n_{i,\mathbf{o}}$) grows at most exponentially (resp. linearly) in $o$. Since $e_{i,\mathbf{o}} / h^{n_{i,\mathbf{o}}}_{i} \in K[z_{1}, \hdots, z_{m}]$ (because it is equivalent to $g_{\mathbf{o}}$) it follows that we can factor $e_{i,\mathbf{o}} = h^{n_{i,\mathbf{o}}}_{i} b_{i,\mathbf{o}}$ with $b_{i,\mathbf{o}} \in K[z_{1}, \hdots, z_{m}]$. Then the $b_{i,\mathbf{o}}$ have degree bounded linearly in $o$ and norm bounded exponentially in $o$, so give the desired representatives for $g_{\mathbf{o}}$.

\paragraph{Proof of (B):} Using \autoref{adelictubezarcover}, this statement can be checked locally in the Zariski topology on $U$ (i.e., after replacing $U$ with $D(h_{i}) \cap U$ and $E$ with $E \cap D(h_{i})$). We check that the functions $(A_{1}, \hdots, A_{m+1})$ converge on some standard adelic tube $\widetilde{E}\langle p \rangle$. This follows immediately by combining (2) and (3) in \autoref{solconstr}.

The open embedding claim is part of the proof of (C), which we turn to next. 

\paragraph{Proof of (C):} We consider an adelic tube $\mathcal{T} = \mathcal{T}(\iota, \alpha, \rho, \overline{f})$ associated to the pair $(E, U)$ in the explicit fixed coordinates; here we understand $\overline{f} = (f_{1}, \hdots, f_{p})$, and $\iota$ to be the fixed embedding $\iota : U \hookrightarrow \mathbb{A}^{m}$. Our goal is to compare $\mathcal{T}^{v}$ to the image of $\eta^{v}$. 

We recall that $\mathcal{T}^{v}$ is defined by inequalities 
\begin{equation}
\label{tubedefeqs3}
|f^{\alpha_{1}}_{i} z_{r}^{\alpha_{2}}| < |\rho| \hspace{3em}    1 \leq r \leq \ell, \hspace{1em} 1 \leq i \leq p , \hspace{1em} \alpha_{1} + \alpha_{2} \leq \alpha , \hspace{1em} \alpha_{1} \geq 1, \hspace{1em} \alpha_{2} \geq 0 .
\end{equation}
To check that $E \langle p \rangle^{v}$ maps into $v$, one has to check that these inequalities hold on $E \langle p \rangle$ after replacing $f_{i}$ with $u_{i}$ and $z_{r}$ with $A_{r}$, where $A_{r} \in \mathcal{O}(E \langle p \rangle^{v})$ is the image of the coordinate $z_{i}$ under the pullback map. After possibly refining the tube $E\langle p \rangle$, we can use \autoref{adelictubezarcover} to reduce this to the same statement after intersecting with $D(h_{i})$. In this case one uses the $\alpha'$-adelicity (for some $\alpha'$) from \autoref{solconstr}(2) and \autoref{solconstr}(3) to show that $|u^{\alpha_{1}}_{i} A^{\alpha_{2}}_{r}| < |\rho|$ for all $i,r, \alpha_{1}$ and $\alpha_{2}$ as in (\ref{tubedefeqs3}) on $E\langle p \rangle$ after possibly shrinking $E \langle p \rangle$. 

\vspace{0.5em}

To produce $\mathcal{T}'$, we now attempt to construct an inverse to $\eta^{v}$ on $\mathcal{T}^{v}$; this will not quite succeed, but it will work up to refinement, in a precise sense. After possibly refining $\mathcal{T}$ and $E \langle p \rangle$, we can assume that the retraction maps $\beta^{v} : \mathcal{T}^{v} \to E^{v}$ from \autoref{retextendscor} are defined on each $\mathcal{T}^{v}$. Together with the functions $f_{1}, \hdots, f_{p}$, we obtain a map $\gamma^{v} : \mathcal{T}^{v} \to \mathbb{A}^{p,v} \times E^{v}$. We obtain a composition
\[ \gamma^{v,-1}(E\langle p \rangle^{v}) \xrightarrow{\gamma^{v}} E\langle p \rangle^{v} \xrightarrow{\eta^{v}} \mathcal{T}^{v} \]
which is evidently identified with the natural inclusion (this can be verified for instance by formally completing along $E$ and using that $\gamma^{v}$ is inverse to $\eta^{v}$ at the formal level). Since the natural inclusion is an open embedding, so are the two intermediate arrows; in particular $\eta^{v}$ is an open embedding, as promised.

Finally we explain the existence of the refinement $\mathcal{T}'$. It suffices to construct an adelic tube $\mathcal{T}'$ such that the restriction of $\gamma$ to $\mathcal{T}'^{v}$ maps to $E\langle p \rangle^{v}$; in particular, it suffices to compute the adelic neighbourhood ``$\gamma^{-1}(E\langle p \rangle)$'' and show it contains an adelic tube. Taking the inverse image of $E \langle p \rangle^{v}$ under $\gamma^{v}$ enforces conditions of the form 
\begin{equation}
\label{tubedefeqslim2}
\left|f_{i}^{\alpha_{1}} B_{r}^{\alpha_{2}}\right| \leq |\rho'| \hspace{3em}  1 \leq r \leq m, \hspace{1em} 1 \leq i \leq p , \hspace{1em} \alpha_{1} + \alpha_{2} \leq \alpha , \hspace{1em} \alpha_{1} \geq 1, \hspace{1em} \alpha_{2} \geq 0 
\end{equation}
where the $B_{r}$ are the power series $\delta(z_{i})$ defining the retraction, as in the proof of \autoref{retextendscor}. Using the fact that the $B_{r}$ are $\alpha''$-adelic for some $\alpha''$ (as shown in the proof of \autoref{retextendscor}) one can see that such equations hold after passing to some refinement $\mathcal{T}'$ of $\mathcal{T}$.

\section{Application to Surface Degenerations}

\subsection{Geometric Interpretations of $G$-functions}

\subsubsection{$G$-function background}

We consider a fixed vector $\mathbf{G} := (G_{1}, \hdots, G_{m})^{t} \in K[[s]]^{m}$ of formal power series.

\begin{defn}
\label{sysofGfuncdef}
The vector $\mathbf{G} = (G_{1}, \hdots, G_{m})^{t}$ is a called a \emph{system} of $G$-functions if there is a differential equation
\begin{equation} 
\label{Gfuncdiffsys}
b_{k} \frac{\partial^k G}{\partial s} + \cdots + b_{1} \frac{\partial G}{\partial s} + b_{0} G = 0
\end{equation}
with $b_{i} \in K(s)$ such that each $G_{i}$ is a solution to (\ref{Gfuncdiffsys}), and if moreover $\mathbf{G}$ has finite height (see \cite[\S1.3]{zbMATH08109702}). 
\end{defn}

\begin{defn}
\label{relevantplacesdef}
Given a point $\xi \in \overline{K}$ and a tuple $\mathbf{G} \in K[[s]]^{m}$, we say that a place $v \in \Sigma_{K(\xi)}$ is relevant for $(\xi, \mathbf{G})$ if $|\xi|_{v} < 1$ and the radius of convergence of $\mathbf{G}_{K(\xi)_{v}}$ is larger than $|\xi|_{v}$. 
\end{defn}

\begin{defn}
\label{holdvadicdef}
Given a vector $\mathbf{G} = (G_{1}, \hdots, G_{m})^{t} \in K[[s]]^{m}$, a homogeneous polynomial $P \in K[x_{1}, \hdots, x_{m}]$, a point $\xi \in \overline{K}$, and a relevant place $v \in \Sigma_{K(\xi)}$ for $(\xi, \mathbf{G})$, we say that $P$ holds $v$-adically at $\xi$ if one has $P(G_{1}(\xi), \hdots, G_{m}(\xi)) = 0$ after embedding $\xi$ and all coefficients into $K(\xi)_{v}$. 
\end{defn}

\begin{defn}
A homogeneous polynomial $P \in K[s]$ as in \autoref{holdvadicdef} is said to be a global relation at $\xi$ if it holds $v$-adically for each relevant place in $\Sigma_{K(\xi)}$. It is said to be strongly non-trivial if it does not arise as a factor of the specialization at $\xi$ of a relation $\widetilde{P} \in K[s][x_{1}, \hdots, x_{n}]$, homogeneous in $x_{1}, \hdots, x_{n}$, such that $\widetilde{P}(s, G_{1}, \hdots, G_{m}) = 0$ in $K[[s]]$. 
\end{defn}

\begin{rem}
There is also a weaker notion of a relation being \emph{non-trivial} (not necessarily strongly). We will have no need for it. By ``non-trivial'' we will always mean strongly non-trivial.
\end{rem}

\begin{thm}
\label{Gheightthm}
Let $\mathbf{G}$ be a system of $G$-functions. Write $\mathcal{S}(\delta) \subset \overline{K}$ for the set of points admitting a strongly non-trivial global homogeneous relation of degree $\delta$. Then there exists an effective constant $\kappa = \kappa(\mathbf{G})$, independent of $\delta$, and an exponent $b$ such that
\begin{equation}
\label{bombineq}
\theta(\xi) \leq \kappa \, \delta^{b}
\end{equation}
for all $\xi \in \mathcal{S}(\delta)$.
\end{thm}

\begin{proof}
This is \cite[Intro., Thm. E]{zbMATH00041964}. See also \cite[Lem. 2.6]{zbMATH08109702} for an explanation of how one can deduce \autoref{Gheightthm} from \cite[Ch. VII, \S5, Thm. 5.2]{zbMATH00041964}. 
\end{proof}

\subsubsection{$G$-functions from Geometry}
\label{Gfuncfromgeosec}

The G-functions we are interested in arise from semistable families $\overline{g} : \overline{X} \to \overline{S}$ of algebraic varieties over a number field $K$, with $\overline{S}$ a curve. (In fact for the applications in this paper, $\overline{g}$ is a relative surface, but for the moment we work more generally.) Let $S \subset \overline{S}$ be the locus where $\overline{g}$ is smooth, and write $g : X \to S$ for the restriction of $\overline{g}$ to $S$. Fix a point $s_{0} \in \overline{S} \setminus S$, and let $s$ be a uniformizing parameter at $s_{0}$. Write $\dR(\overline{X}/\overline{S})$ for the relative log de Rham cohomology complex of $\overline{g}$, and consider, for each $i$, the vector bundles $\mathcal{H}^{i} = R^{i} \overline{g}_{*} \dR(\overline{X}/\overline{S})$ with their Gauss-Manin connection $\nabla$. The residue of $\nabla$ at $s_{0}$ induces a $K$-linear operator $N : \mathcal{H}^{i}_{s_{0}} \to \mathcal{H}^{i}_{s_{0}}$. The following is a consequence of \cite[Ch. V, Appendix]{zbMATH00041964}:

\begin{thm}
\label{kerNareGfuncs}
Let $v_{0} \in \ker N$ be a vector. The vector $v_{0}$ extends uniquely to a section $\widetilde{v}$ of $\mathcal{H}^{i}$ in the formal neighbourhood of $s_{0}$ which is flat for $\nabla$. If one fixes algebraic trivializing sections $\omega_{1}, \hdots, \omega_{m}$ for $\mathcal{H}^{i}$ in some neighbourhood of $s_{0}$, the coordinates of $\widetilde{v}$ in this basis, viewed as power series in $s$, are $G$-functions.
\end{thm}

\begin{proof}
First we explain the existence of $\widetilde{v}$. Write $\mathcal{S}^{\wedge}$ for the formal neighbourhood of $s_{0}$ in $\overline{S}$. If one takes the base-change $\mathcal{H}^{i}_{\mathcal{S}^{\wedge}}$ one obtains a vector bundle with log connection on the formal disk $\mathcal{S}^{\wedge} \cong \Spf K[[s]]$. One can extend $v_{0}$ to an arbitrary section $v$, not necessarily flat, of $\mathcal{H}^{i}_{\mathcal{S}^{\wedge}}$. Then a candidate for $\widetilde{v}$ is given by the formula (cf. \cite[Proof of Thm. 5.1]{katznil})
\[ \widetilde{v} = \sum_{j \in \mathbb{Z}_{\geq 0}} (-s)^{j} \frac{\nabla^{j}}{j!} v . \]
For the output of the formula to be well-defined one needs to know that the poles of the partial sums of the expression are uniformly bounded; in such a situation, one can prove by direct calculation that the result is independent of the extension $v$ of $\widetilde{v}$ chosen, and is flat for the connection. Therefore to show $\widetilde{v}$ exists we may replace $K$ with $\mathbb{C}$, and then analytify to reduce to the same situation for the associated complex analytic family $g^{\textrm{an}}$ in some complex analytic disk $\mathcal{B}$ in $\overline{S}(\mathbb{C})$ centred at $s_{0}$. In this case $N$ induces an endomorphism of the Betti local system $\mathbb{V} = R^{i} g^{\textrm{an}}_{*} \mathbb{Z}$ over $\mathcal{B}$ which agrees with the logarithm of the monodromy around $s_{0}$. Since $\overline{g}$ is semistable, the monodromy around $s_{0}$ is unipotent, and vectors in the kernel of $N$ correspond exactly to those sections of $\mathbb{V}$ over $\mathcal{B}$ which extend to flat sections of $\mathcal{H}^{i,\textrm{an}}_{\mathbb{C}}$. Existence of such a $\widetilde{v}$ follows, as does uniqueness by taking $v$ to be another flat extension of $v_{0}$ and observing that the formula implies $\widetilde{v} = v$.

Next we show the $G$-function claim. Without loss of generality we may assume that $s$ and the sections $\omega_{1}, \hdots, \omega_{m}$ are defined on $\overline{S}$. We may then consider the family $\overline{h} = s \circ \overline{g}$. The spectral sequence of log de Rham cohomology (cf. \cite{katz1968}, \cite[Rem. 3.3]{katznil}, \cite[Ch. IX, 4.3, pg. 188]{zbMATH00041964}) gives a degenerate spectral sequence with term $R^{0} s_{*} (\mathcal{H}^{k}, \nabla) \implies R^{k} \overline{h}_{*} \Omega^{\bullet}_{\overline{X}/\mathbb{A}^{1}}$. If we take the fibre of this spectral sequence above the generic point of $\mathbb{A}^1$, we obtain a $K(\overline{S})$-module with connection $(\mathcal{H}^{k}_{K(\overline{S})}, \nabla)$, which can be regarded as a $K(s)$-module with connection using the map $K(s) \to K(\overline{S})$ induced by $s$.

Set $\mathcal{H} = \mathcal{H}^{i}_{K(\overline{S})}$. If we consider the decomposition $\mathcal{H} \cong \mathcal{O}_{\overline{S}}^{\oplus m}$ corresponding to $\omega_{1}, \hdots, \omega_{m}$, then this decomposition can be further refined by a basis $\omega_{ij}$ for $\mathcal{H}$ over $K(s)$ such that $\omega_{i} = \omega_{i1}$ and $1 \leq j \leq [K(\overline{S}) : K(s)]$. Let $\Lambda$ be the connection matrix with respect to this basis, so that we obtain a differential equation $\frac{d}{ds} v = \Lambda v$ which comes from geometry in the sense of \cite{zbMATH00041964} (using that our $K(s)$-module with connection is associated to $\overline{h}$). Then $\widetilde{v}$ is a solution to this differential equation, and \cite[Ch. V, Appendix]{zbMATH00041964} completes the proof (cf. \cite[Ch. IX, \S4, 4.1]{zbMATH00041964}).
\end{proof}

\begin{cor}
\label{kerNareGfuncsdualversion} 
Write $\mathcal{H}^{i,\wedge}$ for the formal completion of $\mathcal{H}$ at $s_{0}$, and suppose we have a map of vector bundles $\alpha : \mathcal{H}^{i,\wedge} \to \mathcal{O}_{\mathcal{S}^{\wedge}}$ with log connection over $\mathcal{S}^{\wedge}$, where $\mathcal{S}^{\wedge}$ is regarded as a log formal scheme and where $\mathcal{O}_{\mathcal{S}^{\wedge}}$ is regarded as a vector bundle with trivial connection. Fix a basis $\omega_{1}, \hdots, \omega_{m}$ for $\mathcal{H}^{i}$ in some neighbourhood of $s_{0}$, and consider the tuple of functions $\alpha(\omega_{1}), \hdots, \alpha(\omega_{m})$. Then this tuple, viewed as power series in $s$, is a system of G-functions.
\end{cor}

\begin{proof}
Observe via Poincar\'e duality that $\alpha$ corresponds to a flat section of $\mathcal{H}^{2n-i,\wedge}$, where $n$ is the relative dimension of $g$, and likewise $\omega_{1}, \hdots, \omega_{m}$ correspond to a dual coordinate frame. Now apply \autoref{kerNareGfuncs}. 
\end{proof}



In practice one is given a parameter value $s_{1} = s(\xi)$ in some neighbourhood of $0 \in \mathbb{A}^1$, with $\xi \in S(\overline{\mathbb{Q}})$, and one would like to produce algebraic relations on the values of the $G$-functions from \autoref{kerNareGfuncs} after substituting $s = s_{1}$. To do this one should try to understand to what extent the geometry of the degeneration influences the periods of the fibre $X_{\xi}$. However one could run into the issue that $\xi$ is not near the degeneration point $s_{0}$, even if $s_{1}$ is near $0$; the next lemma of Daw and Orr can be used to resolve this issue.

\begin{lem}
\label{daworrcoverlem}
Suppose that $C$ is a smooth projective algebraic curve over a characteristic zero field $K$ and $s_{0} \in C(K)$. Fix a rational function $u$ on $C$ with a zero at $s_0$ and nowhere else. Then after replacing $K$ by a finite extension, there exists an irreducible smooth projective algebraic curve $C_{4}$ over $K$, a finite covering $\nu : C_{4} \to C$, and a rational function $s : C_{4} \to \mathbb{P}^1$ with the following properties:
\begin{itemize}
\item[(i)] if $d$ is the order of vanishing of $u$ at $s_0$, then $\nu^{\sharp}(u) = s^d$; and
\item[(ii)] every zero of $s$ is simple, and $s : C_{4} \to \mathbb{P}^1$ is a Galois covering.
\end{itemize}
\end{lem}

\begin{proof}
This is \cite[Lem. 5.1]{zbMATH08109702}. The lemma is in fact stated somewhat differently, but the proof there shows the statement we give; in particular, under our smoothness assumption $C = C_{1}$ in the notation there, and our $u$ is the function $x_1$ chosen at the beginning of the proof.  
\end{proof}

\subsubsection{Relationship with Special Fibre Components} 
\label{relspecfibcompsec}

We continue with the setup of \S\ref{Gfuncfromgeosec}, except in this section we base-change to $\mathbb{C}$ and analytify so that all objects are considered in the category of complex analytic spaces. This means we can replace $\overline{S}$ with an open ball $\mathcal{B}$ around $s_{0}$ which the parameter $s$ identifies with an open disk in $\mathbb{C}$. Then $\overline{X}$ deformation retracts onto $Y = \overline{X}_{s_{0}}$. We start by reviewing a construction of Steenbrink \cite{zbMATH03515611}. We define a filtration $W_{\bullet}$ on the log de Rham complex $\Omega^{\bullet}_{\overline{X}}(\log Y)$ of the total space via $W_{k} \Omega^{p}_{\overline{X}}(\log Y) = \Omega^{k}_{\overline{X}}(\log Y) \wedge \Omega^{p-k}_{\overline{X}}$. We then set $A^{pq}_{\overline{X}} := \Omega^{p+q+1}_{\overline{X}}(\log Y) / W_{q} \Omega^{p+q+1}_{\overline{X}}(\log Y)$ and define $d_{1} : A^{pq}_{\overline{X}} \to A^{p+1,q}_{\overline{X}}$ by restricting $d$ and $d_{2} : A^{pq}_{\overline{X}} \to A^{p,q+1}_{\overline{X}}$ by $d_{2}(\omega) = \theta \wedge \omega$ where $\theta = \frac{ds}{s}$. This gives a double complex $A^{\bullet\bullet}_{\overline{X}}$ with associated total complex $A^{\bullet}_{\overline{X}}$.

We also get an increasing filtration $W_{\bullet}$ on $\Omega^{\bullet}_{\overline{X}/\mathcal{B}}(\log Y)$ by taking the image of the filtration on $\Omega^{\bullet}_{\overline{X}}(\log Y)$. The complex $A^{\bullet}_{\overline{X}}$ also has a weight filtration defined by (\ref{Acompweightfil}) below.

\begin{lem}
\label{steenbrinklogdRreslemma}
For each $p$, the natural sequence
\[ 0 \to \Omega^{p}_{\overline{X}/\mathcal{B}}(\log Y) \otimes \mathcal{O}_{Y} \xrightarrow{d_{2}} A^{p0}_{\overline{X}} \xrightarrow{d_2} A^{p1}_{\overline{X}} \xrightarrow{d_2} \cdots  \]
is exact. Moreover the map $\widetilde{\theta} : \Omega^{\bullet}_{\overline{X}/\mathcal{B}}(\log Y) \otimes \mathcal{O}_{Y} \to A^{\bullet}_{\overline{X}}$ given by $\omega \mapsto (-1)^{p} \theta \wedge \omega$ on $\Omega^{p}_{\overline{X}/\mathcal{B}}(\log Y) \otimes \mathcal{O}_{Y}$ is a filtered (with respect to the weight filtrations) quasi-isomorphism.
\end{lem}

\begin{proof}
See \cite[Lem 4.15]{zbMATH03515611} and \cite[Cor 4.16]{zbMATH03515611}, as well as \cite[pg. 269]{zbMATH05233837} for the filtered version of the claim. 
\end{proof}

Now suppose that $Y$ is a simple normal crossing divisor and fix an integer $k$ and a locally closed subvariety $D$ of $Y^{[k]}$ such that $D \cap Y^{[k+1]}$ is empty. Then there exists an analytic neighbourhood $\mathcal{D}$ of $D$ in $\overline{X}$ with the property that $Y \cap \mathcal{D}$ is a union of irreducible components $Y_{1}, \hdots, Y_{k}$ such that $D = Y_{1} \cap \cdots \cap Y_{k}$, and such that we obtain an induced map $\mathcal{D} \to \mathcal{B}$. In particular, setting $Y_{D} = Y_{1} \cup \cdots \cup Y_{k}$, we may assume that the pair $(\mathcal{D}, Y_{D})$ is homotopic to $(\mathbb{A}^{k} \times D, V(x_{1}) \times D \cup \cdots \cup V(x_{k}) \times D)$, where $x_{1}, \hdots, x_{k}$ are the coordinates on $\mathbb{A}^{k}$. We consider the log de Rham complex $\Omega^{\bullet}_{\mathcal{D}}(\log Y_{D})$, as well as its relative variant $\Omega^{\bullet}_{\mathcal{D}/\mathcal{B}}(\log Y_{D})$. We also analogously define $A^{pq}_{\mathcal{D}}$, $A^{\bullet\bullet}_{\mathcal{D}}$ and $A^{\bullet}_{\mathcal{D}}$. We then get the following analogue of \autoref{steenbrinklogdRreslemma}:

\begin{lem}
\label{steenbrinkresolutiononE}
The claims in \autoref{steenbrinklogdRreslemma} continue to hold if one replaces $\overline{X}$ with $\mathcal{D}$ and $Y$ with $Y_{D}$. 
\end{lem}

\begin{proof}
Repeat the proofs of \cite[Lem 4.15]{zbMATH03515611}, \cite[Cor 4.16]{zbMATH03515611} and the arguments in \cite[\S11.2.5]{zbMATH05233837}, or base-change along $\mathcal{D} \hookrightarrow \overline{X}$ the sheaves involved in \autoref{steenbrinklogdRreslemma} and use that restriction is exact. 
\end{proof}
In order to understand the cohomological behaviour of the map $\Omega^{\bullet}_{\overline{X}/\mathcal{B}}(\log Y) \to \Omega^{\bullet}_{\mathcal{D}/\mathcal{B}}(\log Y_{D})$, we are therefore reduced to understanding $A^{\bullet}_{\overline{X}} \to A^{\bullet}_{\mathcal{D}}$. We note that both $A^{\bullet}_{\overline{X}}$ and $A^{\bullet}_{\mathcal{D}}$ come with weight filtrations, defined via 
\begin{align}
\label{Acompweightfil}
W_{r} A^{\bullet}_{\overline{X}} &= \bigoplus_{p,q \geq 0} W_{r} A^{pq} \hspace{2em} W_{r} A^{pq} = W_{2q+r+1} \Omega^{p+q+1}_{\overline{X}} (\log Y) / W_{q} \Omega^{p+q+1}_{\overline{X}} (\log Y) , \\
\label{Acompweightfil2}
W_{r} A^{\bullet}_{\mathcal{D}} &= \bigoplus_{p,q \geq 0} W_{r} A^{pq} \hspace{2em} W_{r} A^{pq} = W_{2q+r+1} \Omega^{p+q+1}_{\mathcal{D}} (\log Y_{D}) / W_{q} \Omega^{p+q+1}_{\mathcal{D}} (\log Y_{D}) .
\end{align}
The map $A^{\bullet}_{X} \to A^{\bullet}_{\mathcal{D}}$ evidently respects the weight filtration. For each $j$, write $\widetilde{Y}^{[j]}$ for the disjoint union of the components of $Y^{[j]}$, and $a_{j} : \widetilde{Y}^{[j]} \to \overline{X}$ for the natural map. Likewise for $\widetilde{Y}^{[j]}_{D}$ and $a_{j} : \widetilde{Y}^{[j]}_{D} \to \mathcal{D}$. 
\begin{lem}
\label{gradedquotientsofsteenbrinkseq}
We have
\begin{align*}
\textrm{Gr}^{W}_{r} A^{\bullet}_{\overline{X}} &= \bigoplus_{\substack{k \geq 0 \\ k \geq -r}} (a_{2k+r+1})_{*} \Omega^{\bullet}_{\widetilde{Y}^{[2k+r+1]}}[-r-2k] \\
\textrm{Gr}^{W}_{r} A^{\bullet}_{\mathcal{D}} &= \bigoplus_{\substack{k \geq 0 \\ k \geq -r}} (a_{2k+r+1})_{*} \Omega^{\bullet}_{\widetilde{Y}^{[2k+r+1]}_{D}}[-r-2k]
\end{align*}
\end{lem}

\begin{proof}
The first statement is \cite[Lem 4.18]{zbMATH03515611}; for the second statement either mimic the proof of \cite[Lem 4.18]{zbMATH03515611} or argue by restriction along $\mathcal{D} \hookrightarrow \overline{X}$. 
\end{proof}
There is also a Betti version of the above story for $\overline{X}$. In particular if we write $A^{\bullet}_{\mathbb{C}} = A^{\bullet}_{\overline{X}}$, then by \cite[(4.19)]{zbMATH03515611} there exists a cohomological mixed Hodge complex $(A^{\bullet}_{\mathbb{Z}}, (A^{\bullet}_{\mathbb{Q}}, W), (A^{\bullet}_{\mathbb{C}}, F, W))$ on $Y$ (cf. \cite[(3.6)]{zbMATH03515611}) which extends $A^{\bullet}_{\mathbb{C}}$ with its weight filtration. In particular considering the spectral sequence induced by the weight filtration one obtains 
\begin{equation}
\label{steenspecsec}
E^{-r,q+r}_{1} = H^{q}(Y, \textrm{gr}^{W}_{r} A^{\bullet}_{?} ) \implies H^{q}(Y, A^{\bullet}_{?})
\end{equation}
where $?$ is either $\mathbb{Q}$ or $\mathbb{C}$. By \cite[Lem. 3.11, Thm. 3.18]{zbMATH05233837}] the edge maps and differentials of this spectral sequence are morphisms of (mixed) Hodge structures.

\subsubsection{Special Fibre Constraints From Weight-Monodromy}
\label{specialfibreconstrsec}
We continue with the setup of \S\ref{Gfuncfromgeosec}. As is well-known, the graded pieces $\textrm{gr}^{W}_{r}\, \mathcal{H}^{i}_{s_{0}}$ for the weight filtration can be non-trivial only if $r \in [0,2i]$, and one has $N W_{r} \subset W_{r-2}$. This in particular means that $N^{i+1} = 0$. 

The following is proven by Andr\'e in the course of proving \cite[IX \S4, Thm. 2]{zbMATH00041964}. 

\begin{prop}
\label{weightfildescviarestriction}
Suppose that $\dim \overline{X}_{s_{0}} = k$. Then $W_{2k-1} \mathcal{H}^{k}_{s_{0}}$ is exactly the kernel of the map
\[ \mathcal{H}^{k}_{s_{0}} = R^{k} \Gamma(\overline{X}_{s_{0}}, \dR(\overline{X}/\overline{S}) \otimes \mathcal{O}_{\overline{X}_{s_{0}}}) \to R^{k} \Gamma(\overline{X}^{[k+1]}_{s_{0}}, \dR(\overline{X}/\overline{S}) \otimes \mathcal{O}_{\overline{X}^{[k+1]}_{s_{0}}}) . \]
\end{prop}

\begin{proof}
The first equality comes from proper base-change. Given this equality, the vector space $\mathcal{H}^{k}_{s_{0}}$ is just what is called ``$H^{n}_{\textrm{DR lim}}$'' in the proof of \cite[IX \S4, Thm. 2]{zbMATH00041964}, and the proof there \cite[pg. 190, 4.4.10]{zbMATH00041964} gives an equality ``$\mathbb{C} \langle T_{p}(0) \rangle^{\perp}_{p} = M_{2n-1} H^{n}_{\textrm{DR lim}}$''. Here $M_{2n-1}$ is what we are calling $W_{2k-1}$, and ``$\mathbb{C} \langle T_{p}(0) \rangle^{\perp}_{p}$'' is collection of elements of $\mathcal{H}^{k}_{s_{0}}$ which lie in the kernel of all maps $T_{p}(0) : H^{n}_{\textrm{DR lim}} \to \mathbb{C}$ constructed in the course of the proof. Here $p$ ranges over all points $p \in Y^{[k+1]}$ ($= Y^{[n]}$ for Andr\'e), and the maps $T_{p}(0)$ are constructed by taking sections $\omega$ of $\mathcal{H}^{k}$ defined in a neighbourhood of $0$, restricting them to a small ball in $\overline{X}$ around $p$, expressing the result in cohomology as a power series in the uniformizing parameter $s$, and evaluating at $0$. Applying base-change functors, and using that
\[ R^{k} \Gamma(\overline{X}^{[k+1]}_{s_{0}}, \dR(\overline{X}/\overline{S}) \otimes \mathcal{O}_{\overline{X}^{[k+1]}_{s_{0}}}) = \bigoplus_{p \in \overline{X}^{[k+1]}_{s_{0}}} R^{k} \Gamma(\{ p \}, \dR(\overline{X}/\overline{S}) \otimes \mathcal{O}_{\{ p \}}) \]
and one checks formally that the map $\mathcal{H}^{k}_{s_{0}} \to R^{k} \Gamma(\{ p \}, \dR(\overline{X}/\overline{S}) \otimes \mathcal{O}_{\{ p \}})$ is just $T_{p}(0)$.
\end{proof}

\subsection{Height Bounds}

In this section we prove the various results announced in \S\ref{mainappsec}. 

\subsubsection{Proof of \autoref{mainsurfaceapp1}}

\paragraph{Preliminary Calculations:} In this case the monodromy-weight filtration on $\mathcal{H}^2_{s_{0}}$ is given by
\begin{align*}
W_{4} = \mathcal{H}^2_{s_{0}}, \hspace{2em} W_{3} &= \ker N^2, \hspace{2em} W_{2} = \ker N^2 \cap N^{-1}(\im N^2) + \im N^2 \\
W_{1} &= N(\ker N^2) + \im N^2, \hspace{2em} W_{0} = \im N^2 .
\end{align*}
In particular the graded quotient $W_{4}/W_{3}$ is non-trivial, which by \autoref{weightfildescviarestriction} means that, after possibly replacing $K$ with a finite extension, there is a point $p \in \overline{X}^{[3]}_{s_{0}}$ such that the natural map 
\[ \mathcal{H}^{2}_{s_{0}} \to R^{2} \Gamma(\{ p \}, \dR(\overline{X}/\overline{S}) \otimes \mathcal{O}_{\{ p \}}) \]
is non-zero. 

\paragraph{G-function setup:} Take $D = \{ p \}$, let $\mathcal{g}^{\wedge} : \mathcal{D}^{\wedge} \to \mathcal{S}^{\wedge}$ be the obvious map, and fix a neighbourhood $U \subset \overline{X}$ of $p$ with \'etale coordinates $f_{1}, f_{2}, f_{3}$ such that:
\begin{itemize}
\item[-] the vanishing locus $V(f_{1} f_{2} f_{3})$ agrees with $Y \cap U$; 
\item[-] the product $s = f_{1} f_{2} f_{3}$ gives (the pullback to $U$ of) a uniformizing parameter $s$ at $s_{0}$ defined on a neighbourhood $T \subset \overline{S}$ of $s_{0}$; and
\item[-] we have $\{ p \} = V(f_{1}, f_{2}, f_{3})$.
\end{itemize}
We now apply \autoref{mainadelictubethm} in the following two situations:
\begin{itemize}
\item[(1)] to the data $(U = U, E = \{ p \}, f_{1}, f_{2}, f_{3})$; and
\item[(2)] to the data $(U = T, E = \{ s_{0} \}, s)$.
\end{itemize}
In the first case we obtain an isomorphism $\chi : \Spf K[[x_{1}, x_{2}, x_{3}]] \xrightarrow{\sim} \mathcal{D}^{\wedge}$ of formal schemes with $v$-adic realizations $\chi^{v} : \Delta^{v,3}_{\chi} \hookrightarrow \overline{X}^{v}$, with $\Delta_{\chi}$ some adelic disc, and in the second case we obtain a formal isomorphism $\eta : \Spf K[[y]] \xrightarrow{\sim} \mathcal{S}^{\wedge}$, with $v$-adic realizations $\eta^{v} : \Delta^{v}_{\eta} \hookrightarrow \overline{S}^{v}$, and with $\Delta_{\eta}$ another adelic disk. After scaling the disks and the parameter $s$ we may assume that $\overline{g}$ induces surjective maps $\Delta^{v,3}_{\chi} \to \Delta^{v}_{\eta}$ and that $\Delta^{v}_{\eta}$ contains a component of the locus $|s| < 1$ at each finite place $v$ of $K$. Write $\mathcal{D}^{v}$ for the image of $\chi^{v}$ and $\mathcal{S}^{v}$ for the image of $\eta^{v}$. We write $\mathcal{g}^{v} : \mathcal{D}^{v} \to \mathcal{S}^{v}$ for the induced map. 

To construct our system of $G$-functions we fix a basis $\omega_{1}, \hdots, \omega_{m}$ for $\mathcal{H}$ over $T$, shrinking $T$ if necessary, and apply \autoref{kerNareGfuncsdualversion} to the morphism 
\[ \mathcal{H}^{2,\wedge} \to \underbrace{R^{2} \left(\restr{\overline{g}}{\mathcal{D}^{\wedge}}\right)_{*} \dR(\mathcal{D}^{\wedge}/\mathcal{S}^{\wedge})}_{=: \mathfrak{H}^{\wedge}} \cong (\mathcal{O}_{\mathcal{S}^{\wedge}}, d) , \]
where the notation $\dR(\mathcal{D}^{\wedge}/\mathcal{S}^{\wedge})$ denotes the corresponding log de Rham complex, and the isomorphism is at the level of vector bundles with connection on the formal log scheme $\mathcal{S}^{\wedge}$. (To see that $\mathfrak{H}^{\wedge}$ is a rank $1$ vector bundle apply \cite[Lem 2.2]{zbMATH08109694} (cf. \cite[1.13]{zbMATH03515611}); the Gauss-Manin connection of $\mathcal{D}^{\wedge}/\mathcal{S}^{\wedge}$ is also easily seen to be trivial.) Call the resulting system of $G$-functions $\mathbf{G}$. 

Now for each $v \in \Sigma_{K}$ we obtain $v$-adic realizations of $\mathbf{G}$ as follows. When $v$ is finite, we consider the relative de Rham cohomology $\mathfrak{H}^{v} := R^{2} \mathcal{g}^{v}_{*} \dR(\mathcal{D}^{v}/\mathcal{S}^{v})$, where $\dR(\mathcal{D}^{v}/\mathcal{S}^{v})$ is the obvious relative log de Rham complex. Applying \cite[Lem 2.2]{zbMATH08109694} (cf. \cite[1.13]{zbMATH03515611}) we find that $\mathfrak{H}^{v}$ is a free rank $1$-vector bundle whose formal completion at $s_{0}$ agrees with $\mathfrak{H}^{\wedge} := \mathcal{g}^{\wedge}_{*} \dR(\mathcal{D}^{\wedge}/\mathcal{S}^{\wedge})$. The Gauss-Manin connection on $\mathfrak{H}^{v}$ (resp. $\mathfrak{H}^{\wedge}$) is easily verified to be the trivial one, with the class of $\frac{dx_{2}}{x_{2}} \wedge \frac{dx_{3}}{x_{3}}$ a global flat section. In the fixed bases the restrictions $\mathcal{H}^{v} \to \mathfrak{H}^{v}$ give the desired $v$-adic realizations of $\mathfrak{G}^{v}$ for $v$ finite. 

In the case of $v$ infinite, we may fix analytic neighbourhoods $\mathcal{D}^{v}$ and $\mathcal{S}^{v}$ of $p$ and $s_{0}$ respectively, which in the fixed coordinates $(x_{1}, x_{2}, x_{3})$ and $s$ are open balls of radius $\sqrt{3} r^{1/3}_{v}$ and $r_{v}$ respectively for some small $r_{v}$ depending on $v$, and apply \cite[Lem 2.2]{zbMATH08109694} (cf. \cite[1.13]{zbMATH03515611}) to learn that $\mathfrak{H}^{v} := R^2 \mathcal{g}^{v}_{*} \dR(\mathcal{D}^{v}/\mathcal{S}^{v})$ is trivial, where $\mathcal{g}^{v}$ and $\dR(\mathcal{D}^{v}/\mathcal{S}^{v})$ are constructed in the analogous fashion. Again we get a $v$-adic realization of $\mathbf{G}$ by considering the restrictions $\mathcal{H}^{v}_{\mathcal{S}^{v}} \to \mathfrak{H}^{v}$ in the fixed bases. Note that in the complex analytic case we also have, for each $s_{1} \in \mathcal{S}^{v} \setminus \{ 0 \}$, a map $\gamma^{*}_{s_{1}} : H^{2}(X_{s_{1}}, \mathbb{Z}(2)) \to H^{2}(\mathcal{D}_{s_{1}}, \mathbb{Z}(2)) \simeq (2\pi i)^2 \mathbb{Z}$ that agrees with the map $\mathcal{H}^{v}_{s_{1}} \to \mathfrak{H}^{v}_{s_{1}}$ after tensoring with $\mathbb{C}$. 

When $v$ is finite, for each $s_{1} \in \mathcal{S}^{v}(\overline{K_{v}})$, we find by following the proof of \cite[Thm 4.4(ii)]{zbMATH08109694} that there is a rigid-subvariety $\mathfrak{D}^{v}_{s_{1}} \subset \mathcal{D}^{v}_{s_{1},\mathbb{C}_{p}}$ which is a product of rigid-analytic annuli and such that the maps $H^{i}_{\textrm{dR}}(\mathcal{D}^{v}_{s_{1},\mathbb{C}_{p}}) \to H^{i}_{\textrm{dR}}(\mathfrak{D}^{v}_{s_{1}})$ of overconvergent de Rham cohomology groups are isomorphisms for each $i$. 

\paragraph{Relations and non-Triviality:} Given the above descriptions, and a point $s_{1} \in \overline{S}(\overline{K})$ we produce relations as follows.
\begin{itemize}
\item[(1)] In the case where $v$ is finite and $v$ is relevant for $(s(s_{1}), \mathbf{G})$ (recall \autoref{relevantplacesdef}) and $\mathcal{L}$ is any line bundle on $X_{s_{1}}$, one knows from \cite[Thm 3.7.2, Prop. 4.7.3]{zbMATH02043955} that the restriction $\restr{\mathcal{L}}{\mathfrak{D}_{s_{1}}}$ is trivial. In particular, if $c^{1}(\mathcal{L}) \in H^{2}_{dR}(X^{v}_{s_{1}})$ is the Chern class of $\mathcal{L}$, the image of $c^{1}(\mathcal{L})$ in $H^{2}_{\textrm{dR}}(\mathfrak{D}^{v}_{s})$ vanishes. Now write $c^{1}(\mathcal{L}) = \sum_{i} a_{i} \omega_{i,s}$. We learn that 
\begin{equation}
\label{linebundlerel}
\sum_{i} a_{i} \mathbf{G}_{i}(s_{1}) = 0 .
\end{equation}
The coefficients $a_{i}$ are defined over a finite field extension $K(\mathcal{L}, s_{1})$ of $K$, and we choose $\mathcal{L}$ such that the degree of this field is minimized.
\item[(2)] In the case where $v$ is infinite and $v$ is relevant for $(s(s_{1}), \mathbf{G})$, we have that $\gamma^{*}_{s_{1}}(\mathcal{L}) = M (2 \pi i)^2$, where $M$ is an integer. Given two such line bundles $\mathcal{L}_{1}$ and $\mathcal{L}_{2}$ we can find a linear combination $\mathcal{L} = M_{2} \mathcal{L}_{1} + M_{1} \mathcal{L}_{2}$ such that $\gamma^{*}_{s_{1}}(\mathcal{L}) = 0$. Once again we can write $c^{1}(\mathcal{L})  = \sum_{i} a_{i} \omega_{i,s}$ and obtain a relation defined over a finite field extension $K(\mathcal{L}, s_{1})$ of $K$. We choose $\mathcal{L}$ once again so that the degree of this extension is minimized.
\end{itemize}

\begin{lem}
The relations in (1) and (2) are non-trivial.
\end{lem}

\begin{proof}
Note that it suffices to check non-triviality at some fixed complex place $v$ of $K$ (cf. the discussion in \cite[\S6.2.2]{zbMATH08109694}). The result then follows from \cite[Prop. 6.8]{zbMATH08109694} and the (assumed) $\mathbb{Q}$-simplicity of the local system underlying the transcendental part of the cohomology in degree two.
\end{proof}

Now to each point $s_{1} \in S(\overline{K})$ we associate a relation $R_{s_{1}} \in K[g_{1}, \hdots, g_{m}]$ which is the product over all of all Galois conjugates of all linear relations constructed in (1) and (2) above. Since each factor of this relation is non-trivial, so is the product (cf. \cite[pg. 140, Rem. 5.3]{zbMATH00041964} and \cite[pg. 138, Rem. 2]{zbMATH00041964}). 

\begin{lem}
The degree of $R_{s_{1}}$ is bounded by a uniform polynomial in $[K(s_{1}) : K]$.
\end{lem}

\begin{proof}
Since the relations at the finite places are independent of $v$ and the number of infinite places to consider is bounded by $[K(s) : \mathbb{Q}]$, it suffices to show that the degree of the minimal field $K'/K(s_{1})$ over which the de Rham realization of the Picard lattice is defined is bounded by a uniform polynomial in $[K(s_{1}) : K]$. 

To show this we may follow the argument of \cite[Prop. 5.2]{papas2023unlikelyintersectionstorellilocusv1}, which demonstrates the same fact except for lattices of endomorphisms (the absolute Hodge assumptions in loc. cit. are satisfied here because all classes involved are algebraic). In particular one may assume that $K'$ is Galois, and then one argues as in loc. cit. that one has an injective map $\textrm{Gal}(K'/K(s_{1})) \to \GL(\Lambda)$, where $\Lambda$ is the Picard lattice. Because the elements in the image all have finite order, this map stays injective after reduction modulo $3$, and then the result follows from the bound \cite[Thm 3.2]{zbMATH00033093} of Silverberg.
\end{proof}

Note that the proof of the previous lemma gives in particular that:

\begin{lem}
\label{picardlatticeboundeddeg}
The de Rham field of definition of the Picard lattice of $X_{s_{1}}$ has degree over $K$ bounded by a polynomial in $[K(s_{1}) : K]$. 
\end{lem}
Given all of the above, the desired result would now follow from \autoref{Gheightthm} if we knew that, whenever we had a place $v$ relevant for $(s(s_{1}), \mathbf{G})$, the point $s_{1}$ was in the component of the $v$-adic neighbourhood $|s| < 1$ containing $s_{0}$. This need not be true, but using \autoref{daworrcoverlem} one can replace $\overline{S}$ with a finite covering for which the zeros of the parameter $s$ correspond to degeneration points $s_{01}, \hdots, s_{0j}$, with the degenerations on each component of $|s| < 1$ for each $v$ all in the orbit of an automorphism of $\overline{S}$ commuting with $s : \overline{S} \to \mathbb{P}^1$. Carrying out the above construction with this new parameter $s$, one verifies that whenever we have a place $v$ relevant for $(s(s_{1}), \mathbf{G})$, there is a corresponding $s_{1}$ in $|s| < 1$ for which (1) or (2) as above apply; this reduction is standard (cf. \cite[\S5.5]{zbMATH08109702} \cite[\S6.3]{zbMATH08109694}). 

\subsubsection{Proof of \autoref{mainsurfaceapp2}}

\paragraph{Preliminary Calculations:} In this case the monodromy-weight filtration on $\mathcal{H}^2_{s_{0}}$ is given by
\begin{align*}
W_{4} = W_{3} = \mathcal{H}^2_{s_{0}}, \hspace{2em} W_{2} &= \ker N, \hspace{2em} W_{1} = \im N \hspace{2em} W_{0} = 0 .
\end{align*}
We can alternatively describe this weight filtration using the spectral sequence of Steenbrink, which is the spectral sequence for the complex $A^{\bullet}_{\overline{X}}$ from \S\ref{relspecfibcompsec} with its weight filtration. Using \autoref{gradedquotientsofsteenbrinkseq} this spectral sequence has the following terms contributing to $H^{2}(Y, A^{\bullet}_{\overline{X}})$. 
\begin{align}
\label{steenspecseqtermsH2}
E^{2,0}_{1} = H^{0}(\widetilde{Y}^{[3]}),& \hspace{2em} E^{1,1}_{1} = H^{1}(\widetilde{Y}^{[2]}) \hspace{2em} E^{0,2}_{1} = H^{2}(\widetilde{Y}^{[1]}) \oplus H^{0}(\widetilde{Y}^{[3]})(-1) \\
& E^{-1,3}_{1} = H^{1}(\widetilde{Y}^{[2]})(-1), \hspace{2em} E^{-2,4}_{1} = H^{0}(\widetilde{Y}^{[3]})(-2) .
\end{align}
Now by \cite[Cor. 4.20]{zbMATH03515611} this spectral sequence degenerates at the second page. Since $W_{4} = W_{3}$ the edge map $H^{2}(Y, A^{\bullet}_{\overline{X}}) \to E^{-2,4}_{2}$ is zero, and so we obtain a map $H^{2}(Y, A^{\bullet}_{\overline{X}}) \to E^{-1,3}_{2}$. The differential $d_{1}$ of the first page acts as $E^{-r,q+r}_{1} \to E^{-r+1,q+r}_{1}$, so in particular $E^{-1,3}_{2}$ is a subspace of $E^{-1,3}_{1}$ modulo the image of $E^{-2,3}_{1} \to E^{-1,3}_{1}$. On the other hand one computes, again from \autoref{gradedquotientsofsteenbrinkseq}, that $E^{-2,3}_{1} = H^{1}(\widetilde{Y}^{[3]})(2) = 0$, and so we obtain a non-zero map
\begin{equation}
\label{trueedgemap}
H^{2}(Y, A^{\bullet}_{\overline{X}}) \to E^{-1,3}_{1} = H^{1}(\widetilde{Y}^{[2]})(-1) .
\end{equation}
Now recall that $\widetilde{Y}^{[2]}$ is the disjoint union of the components $\overline{D}_{1}, \hdots, \overline{D}_{\ell}$ of $Y^{[2]}$, hence we have $H^{1}(\widetilde{Y}^{[2]})(-1) = \bigoplus_{i} H^{1}(\overline{D}_{i})(-1)$ and there exists $j$ such that the induced map $\rho : H^2(Y, A^{\bullet}_{\overline{X}}) \to H^1(\overline{D}_{j})(-1)$ is non-zero. 

For such a $j$ let $D \subset \overline{D}_{j} \setminus Y^{[3]}$ be an affine open curve, and define $\mathcal{D}, \widetilde{Y}_{D}$, etc., as in \S\ref{relspecfibcompsec}. We may choose $\mathcal{D}$ so that there exists \'etale coordinates $\overline{f} = (f_{1}, f_{2}, f_{3}) : \mathcal{D} \to \mathbb{C}^3$ such that $s = f_{1} f_{2}$ and $D_{i} = \mathcal{D} \cap Y_{i}$ is the vanishing locus of $f_{i}$. 

We claim that $\rho$ factors through a natural sequence
\begin{equation}
\label{factorrhoseq}
H^2(Y, A^{\bullet}_{\overline{X}}) \to H^2(\mathcal{D}, A^{\bullet}_{\mathcal{D}}) \to E^{-1,3}_{1}(A^{\bullet}_{\mathcal{D}}, W) = H^{1}(D)(-1) .
\end{equation}
Indeed, we may first observe that $E^{-2,4}_{1}(A^{\bullet}_{\mathcal{D}}, W) = H^{0}(\widetilde{Y}^{[3]}_{D})(-2) = 0$ and so we obtain an edge map $H^2(\mathcal{D}, A^{\bullet}_{\mathcal{D}}) \to E^{-1,3}_{2}(A^{\bullet}_{\mathcal{D}}, W)$. As before $E^{-1,3}_{2}(A^{\bullet}_{\mathcal{D}}, W)$ is a subquotient of $E^{-1,3}_{1}(A^{\bullet}_{\mathcal{D}}, W)$, where the quotient is by the image of $E^{-2,3}_{1}(A^{\bullet}_{\mathcal{D}}, W) = H^{1}(\widetilde{Y}^{[3]}_{D})(2) = 0$, and so we obtain an arrow $H^2(\mathcal{D}, A^{\bullet}_{\mathcal{D}}) \to E^{-1,3}_{1}(A^{\bullet}_{\mathcal{D}}, W)$. The passage from $\overline{X}$ to $\mathcal{D}$ is functorial, hence (\ref{factorrhoseq}) factors the composition of $\rho$ with the map $H^{1}(\overline{D}_{j})(-1) \to H^{1}(D)(-1)$. It is elementary that the latter map is injective; in particular, both maps in (\ref{factorrhoseq}) are non-zero.

\begin{lem}
\label{H1Dsummand}
The image of (\ref{factorrhoseq}) in $H^{1}(D)(-1)$ agrees with the image, under the natural map $H^{1}(\overline{D}_{j})(-1) \to H^{1}(D)(-1)$, of a subspace $W \subset H^{1}(\overline{D}_{j})(-1)$ corresponding to an isogeny factor the Jacobian of $\overline{D}_{j}$.
\end{lem}

\begin{proof}
Recalling that the morphism $\rho$ extends to a map of (mixed) Hodge structures (recall the discussion around (\ref{steenspecsec})), it follows that $\textrm{im}(\rho)$ underlies the de Rham realization of a $\mathbb{Q}$-summand of the Hodge structure on $H^{1}(\overline{D}_{j})(-1)$. Since $\overline{D}_{j}$ is a smooth projective curve, the result follows from the correspondence Hodge conjecture for the Jacobian $J(\overline{D}_{j})$ of $\overline{D}_{j}$. 
\end{proof}

Now recalling the quasi-isomorphisms of \autoref{steenbrinklogdRreslemma} and \autoref{steenbrinkresolutiononE}, the sequence (\ref{factorrhoseq}) can be rewritten as a sequence 
\begin{equation}
\label{logversionofres}
\mathcal{H}^{2}_{s_{0}} \to R^{2} \Gamma \Omega^{\bullet}_{\mathcal{D}/\mathcal{B}}(\log Y_{D}) \otimes \mathcal{O}_{Y_{D}} \to H^{1}(D)(-1) .
\end{equation}
Because the quasi-isomorphism of \autoref{steenbrinkresolutiononE} is filtered, the second arrow in (\ref{logversionofres}) can also be understood as arising from the weight spectral sequence for the natural weight filtration on $\Omega^{\bullet}_{\mathcal{D}/\mathcal{B}}(\log Y_{D}) \otimes \mathcal{O}_{Y_{D}}$ obtained as the image of the one on $\Omega^{\bullet}_{\mathcal{D}}(\log Y_{D})$. To further analyze the cohomology one may shrink $\mathcal{D}$ so that the \'etale coordinates $f_{1}, f_{2}, f_{3}$ identify $\mathcal{D}$ with a tube $\mathcal{T}$ around the union $\Delta \times D \cup \Delta \times D$ along $D$, where the coordinates on the first $\Delta \times D$ are $(f_{1}, f_{3})$ and the coordinates on the second are $(f_{2}, f_{3})$. The fibre above $s_{0}$ of this tube is then identified with the union $\Delta \times D \cup \Delta \times D$ itself. One then has 

\begin{align}
\label{relweightfildecomp}
W_{0} \Omega^{1}_{\mathcal{D}/\mathcal{B}}(\log Y_{D}) \otimes \mathcal{O}_{Y_{D}} &= (f_{1} \mathcal{O}_{Y_{D}} + f_{2} \mathcal{O}_{Y_{D}}) \frac{df_{2}}{f_{2}} \oplus \mathcal{O}_{Y_{D}} df_{3} \\
\Omega^{1}_{\mathcal{D}/\mathcal{B}}(\log Y_{D}) \otimes \mathcal{O}_{Y_{D}} &= \mathcal{O}_{Y_{D}} \frac{df_{2}}{f_{2}} \oplus \mathcal{O}_{Y_{D}} df_{3} \\
W_{0} \Omega^{2}_{\mathcal{D}/\mathcal{B}}(\log Y_{D}) \otimes \mathcal{O}_{Y_{D}} &= (f_{1} \mathcal{O}_{Y_{D}} + f_{2} \mathcal{O}_{Y_{D}}) \frac{df_{2}}{f_{2}} \wedge df_{3} \\
\Omega^{2}_{\mathcal{D}/\mathcal{B}}(\log Y_{D}) \otimes \mathcal{O}_{Y_{D}} &= \mathcal{O}_{Y_{D}} \frac{df_{2}}{f_{2}} \wedge df_{3} 
\end{align}
where we use the relation 
\[ 0 = \frac{ds}{s} = \frac{df_{1}}{f_{1}} + \frac{df_{2}}{f_{2}} \]
to simplify the presentation of the modules, and where the other steps of the weight filtration are trivial. In particular this complex has exactly one non-trivial graded quotient, namely $W_{1}/W_{0}$, which has two non-zero terms concentrated in degree $1$ and $2$, given by $\mathcal{O}_{D} \frac{df_{2}}{f_{2}}$ and $\mathcal{O}_{D} \frac{df_{2}}{f_{2}} \wedge df_{3}$, respectively. 

By explicit integration-by-parts one observes that $H^{2}(W_{0} \Omega^{\bullet}_{\mathcal{D}/B}(\log Y_{D}) \otimes \mathcal{O}_{Y_{D}}) = 0$ and hence 
\[ H^{2}(\Omega^{\bullet}_{\mathcal{D}/B}(\log Y_{D}) \otimes \mathcal{O}_{Y_{D}}) = H^{2}(W_{1}/W_{0}) = H^{1}_{\textrm{dR}}(D)(-1); \]
in particular the second arrow in the sequence (\ref{logversionofres}) is an isomorphism. 

\paragraph{G-function Setup:} We now assume we have chosen $D$ as above defined over $K$, closed in some Zariski open subset $U \subset \overline{X}$, and such that we have $K$-algebraic \'etale coordinates $\overline{f} = (f_{1}, f_{2}, f_{3}) : U \to \mathbb{A}^3$ with the property that $s = f_{1} f_{2}$ and $D$ agrees with the vanishing locus $V(f_{1},f_{2})$ in $U$. Write $\mathcal{D}^{\wedge}$ for the formal neighbourhood of $D$ in $\overline{X}$, let $\mathcal{g}^{\wedge} : \mathcal{D}^{\wedge} \to \mathcal{S}^{\wedge}$ be the induced map. Let $Y_{D}$ be the formal scheme $\mathcal{D}^{\wedge} \cap Y := \mathcal{D}^{\wedge} \times_{\overline{X}} Y$.

Now let $\mathcal{H} = R^{2} \overline{g}_{*} \Omega^{\bullet}_{\overline{X}/\overline{S}}(\log Y)$, let $\mathcal{H}^{\wedge}$ be the restriction to $\mathcal{S}^{\wedge}$, and set $\mathfrak{H}^{\wedge}_{D} = R^{2} \mathcal{g}^{\wedge}_{*} \Omega^{\bullet}_{\mathcal{D}^{\wedge}/\mathcal{S}^{\wedge}}(\log Y_{D})$. 

\begin{lem}
As a flat bundle with log connection, $\mathfrak{H}^{\wedge}_{D}$ is trivial.
\end{lem}

\begin{proof}
By using \autoref{mainadelictubethm} and the fixed \'etale coordinates to trivialize the map $\mathcal{g}^{\wedge} : \mathcal{D}^{\wedge} \to \mathcal{S}^{\wedge}$, we are reduced to considering the same problem for the map $D[2]^{\wedge} \to s_{0}[1]^{\wedge}$ which acts by $y \mapsto x_{1} x_{2}$. This map is a product of a trivial map $D \to s_{0}$ with the formal completion of the obvious map $p : \mathbb{A}^2 \to \mathbb{A}^1$ given by $y \mapsto x_{1} x_{2}$. As the Gauss-Manin connection of the trivial map is trivial, one reduces by a Kunneth calculation to computing the logarithmic Gauss-Manin connection of the map $p$, which is easily seen to be trivial as well.
\end{proof}
\noindent We consider the map $\rho : \mathcal{H}^{\wedge} \to \mathfrak{H}^{\wedge}_{D}$, which is a morphism of flat vector bundles with log connections. The image of $\rho$ is then a flat coherent subsheaf of a trivial bundle with flat log connection, hence necessarily a vector subbundle. The fibre at $s_{0}$ of $\rho$ is nothing more than the map of (\ref{logversionofres}). It then follows from \autoref{H1Dsummand} that the image of $\rho$ is a flat vector subbundle $\mathfrak{F}^{\wedge} \subset \mathfrak{H}^{\wedge}_{D}$ of rank equal to the $2r$, where $r$ is the rank of the corresponding isogeny factor of $J(\overline{D}_{j})$. 

Our tuple $\mathbf{G}$ of $G$-functions we define as follows. We fix a frame $\omega_{1}, \hdots \omega_{m}$ for $\mathcal{H}^{\wedge}$, and write $\{ \omega_{I} \}$ for the induced basis of $\bigwedge^{2r} \mathcal{H}^{\wedge}$. We obtain an induced map $\alpha : \bigwedge^{2r} \mathcal{H}^{\wedge} \to \bigwedge^{2r} \mathfrak{F}^{\wedge} \cong (\mathcal{O}_{\mathcal{S}^{\wedge}}, d)$. Applying \autoref{kerNareGfuncsdualversion}, and using in particular that the bundle $\bigwedge^{2r} \mathcal{H}$ appears as a direct summand inside the cohomology of the $2r$-fold self-product 
\[ \underbrace{\overline{X} \times_{\overline{S}} \cdots \times_{\overline{S}} \overline{X}}_{2r\textrm{ times}} \to \overline{S} \]
to show that the functions in question arise as solutions of a geometric system of differential equations, one learns that $\mathbf{G}$ is a tuple of G-functions. 

Now we apply \autoref{mainadelictubethm} with $E = D$ and $U$ and $\overline{f}$ as chosen above. We obtain an open embedding $\chi : D\langle 2 \rangle \hookrightarrow \mathcal{T}$ of adelic tubes with realizations $\chi^{v} : D\langle 2 \rangle^{v} \hookrightarrow \mathcal{T}^{v} \subset U^{v}$. We may also apply \autoref{mainadelictubethm} to the pair $(T, \{ s_{0} \})$ with coordinate $s$, with $T \subset \overline{S}$ an open subvariety for which $V(s) = \{ s_{0} \}$, and hence obtain $v$-adic disks $\mathcal{S}^{v}$ corresponding to a common trivialization $\eta : s_{0}[1]^{\wedge} \xrightarrow{\sim} \mathcal{S}^{\wedge}$. Shrinking the adelic tube $D\langle 2 \rangle$ if necessary we have induced maps $D\langle 2 \rangle^{v} \to \mathcal{S}^{v}$. 

Now fix an adelic ball $\mathfrak{D}$ of $D$ as in \autoref{chooseadelicballcontainingcohom}; i.e., so that the maps $H^{i}_{\textrm{dR}}(D^{v}) \to H^{i}_{\textrm{dR}}(\mathfrak{D}^{v})$ of overconvergent de Rham cohomology groups are all injective. Apply \autoref{adelictubecontainsprodlem} to obtain an adelic disk $\Delta$ so that $\Delta^{v,2} \times \mathfrak{D}^{v} \subset D \langle 2 \rangle^{v}$ for each $v$, compatibly with the product decomposition $\mathbb{A}^2 \times D$. Write $\mathcal{D}^{v}$ for the image under $\chi^{v}$ of $\Delta^{v,2} \times \mathfrak{D}^{v}$. We may consider the relative overconvergent de Rham cohomology of $\mathcal{D}^{v}$ over $\mathcal{S}^{v}$ (with the overconvergent structure coming from the inclusion $\mathcal{D}^{v} \hookrightarrow \overline{X}^{v}_{\mathcal{S}^{v}}$), set $Y^{v}_{D} := \mathcal{D}^{v} \cap Y^{v}$, let $\mathcal{g}^{v} : \mathcal{D}^{v} \to \mathcal{S}^{v}$ be the map induced by $\overline{g}$, and define 
\begin{equation}
\label{hdvdef}
\mathfrak{H}^{v}_{D} := R^{2} \mathcal{g}^{v}_{*} \Omega^{\bullet}_{\mathcal{D}^{v}/\mathcal{S}^{v}} (\log Y^{v}_{D}) .
\end{equation}
By construction, we have an inclusion $\Delta^{v,2} \times \mathfrak{D}^{v} \hookrightarrow \Delta^{v,2} \times \overline{D}^{v}_{j}$ for each $v$; because $\chi^{v}$ is an open embedding this induces an inclusion $c : \mathcal{D}^{v} \hookrightarrow \Delta^{v,2} \times \overline{D}^{v}_{j} =: \overline{\mathcal{D}}^{v}$. Let $\overline{\mathcal{g}}^{v} : \overline{\mathcal{D}}^{v} \to \mathcal{S}^{v}$ be the natural composition $\Delta^{v,2} \times \overline{D}^{v}_{j} \to \Delta^{v}_{\eta} \xrightarrow{\eta^{v}} \mathcal{S}^{v}$, where the first map factors through the map $\Delta^{v,2} \to \Delta^{v}_{\eta}$ given by $y \mapsto x_{1} x_{2}$. 

Set $\mathfrak{H}^{v}_{\overline{D}} := R^{2} \overline{\mathcal{g}}^{v}_{*} \Omega^{\bullet}_{\overline{\mathcal{D}}^{v}/\mathcal{S}^{v}}(\log V(x_{1}x_{2}))$. By construction, we have, for each $v$, a natural map $\mathfrak{H}^{v}_{\overline{D}} \to \mathfrak{H}^{v}_{D}$. Let $e : H^{1}(\overline{D}_{j}) \to H^{1}(\overline{D}_{j})$ be the idempotent corresponding to the subspace $W$ of \autoref{H1Dsummand}. Then $e$ extends to an idempotent on $\mathfrak{H}^{v}_{\overline{D}}$ with image $\mathfrak{H}^{v}_{W}$. Write $\mathfrak{F}^{v}$ for the image of $\mathfrak{H}^{v}_{W}$ in $\mathfrak{H}^{v}_{D}$, which is a flat vector subbundle. Let $\mathfrak{H}^{v}$ be the restriction of $\mathcal{H}^{v}$ to $\mathcal{S}^{v}$. 

\begin{lem}
The image of $\mathfrak{H}^{v}$ in $\mathfrak{H}^{v}_{D}$ lands inside $\mathfrak{F}^{v}$. 
\end{lem}

\begin{proof}
Because all the vector bundles involved are flat bundles (with logarithmic connections), the morphisms of bundles are compatible with the connections, and the connections on the bundles $\mathfrak{H}^{v}_{D}$ and $\mathfrak{F}^{v}$ are trivial, it suffices to show this at the fibre $s_{0}$. After making all the natural identifications, this is just \autoref{H1Dsummand}, as well as the comparison between (\ref{factorrhoseq}) and (\ref{logversionofres}). 
\end{proof}
When $v$ is a finite place, the $v$-adic analogue $\alpha^{v}$ of $\alpha$ is constructed by taking the $2r$'th exterior power of the the map $\mathfrak{H}^{v} \to \mathfrak{F}^{v}$. We may identify the formal completion of $\alpha^{v}$ with $\alpha$. One once again obtains (the $v$-adic realization of) the vector $\mathbf{G}$ by expressing $\alpha^{v}$ with respect to the fixed bases.

\vspace{0.5em}

It remains to construct an appropriate $\alpha^{v}$ when $v$ is an infinite place of $K$. In this case, the complex analytic space $D^{v}$ is a surface of genus $1$ with several punctures. By removing a closed punctured disk around each such puncture we obtain an open neighbourhood $\mathfrak{D}^{v} \subset D^{v}$ which is relatively compact; note that the maps $H^{i}_{\textrm{dR}}(D^{v}) \to H^{i}_{\textrm{dR}}(\mathfrak{D}^{v})$ of analytic de Rham cohomology groups are isomorphisms. Recalling that we have fixed \'etale coordinates $f_{1}, f_{2}, f_{3}$ on $U^{v} \supset D^{v}$, we obtain a map $\overline{f} : U^{v} \to \mathbb{A}^{3}$ which sends $Y^{v} \cap U^{v}$ to the union $V(x_{1}) \cup V(x_{2})$ of coordinate axes. Because the \'etale map $\overline{f}$ is a local isomorphism, one obtains, by taking an inverse image, a tubular neighbourhood $\mathcal{T}^{v}$ of $D^{v}$ in $U^{v}$. Using that $\mathfrak{D}^{v}$ is relatively compact we can construct an open embedding $\Delta^{v,2} \times \mathfrak{D}^{v} \hookrightarrow \mathcal{T}^{v} \subset U^{v}$ such that the two copies of $\Delta^{v}$ are complex analytic balls around the $f_{1} = 0$ and $f_{2} = 0$ axes. Analogously to before we define $\mathcal{D}^{v}$ to be the image of this embedding, and set $\mathcal{g}^{v} : \mathcal{D}^{v} \to \mathcal{S}^{v}$ to be the restriction of $\overline{g}$ to $\mathcal{D}^{v}$, where $\mathcal{S}^{v} \subset \overline{S}^{v}$ is a disk around $s_{0}$, shrinking $\mathcal{D}^{v}$ as necessary. We then set $\mathfrak{H}^{v}$ to be the restriction of $\mathcal{H}^{v}$ to $\mathcal{S}^{v}$, and define $\mathfrak{H}^{v}_{D} = R^{2} \mathcal{g}^{v}_{*} \Omega^{\bullet}_{\mathcal{D}^{v}/\mathcal{S}^{v}}(\log Y^{v}_{D})$, where we define $Y^{v}_{D} := Y^{v} \cap \mathcal{D}^{v}$. The $v$-adic analogue $\alpha^{v}$ of $\alpha$ is then the map $\mathfrak{H}^{v} \to \mathfrak{F}^{v} \subset \mathfrak{H}^{v}_{D}$, where $\mathfrak{F}^{v}$ denotes the image of $\mathfrak{H}^{v} \to \mathfrak{H}^{v}_{D}$. 

Finally we note that, taking some point $s_{1} \in \mathcal{S}^{v}$ with $s_{1} \neq s_{0}$, the map $\mathfrak{H}^{v}_{s_{1}} \to \mathfrak{H}^{v}_{D,s_{1}}$ admits a $\mathbb{Q}$-structure, as it is identified with $H^{2}(X^{v}_{s_{1}}, \mathbb{Q}) \to H^{2}(\mathcal{D}^{v}_{s_{1}}, \mathbb{Q})$ under the Betti-de Rham isomorphism.

\paragraph{Relations and non-Triviality:} Now we consider a point $s_{1} \in \overline{S}(\overline{K})$ for which (\ref{rankNineq2}) holds and produce relations as follows. We let $d_{1}$ denote the rank of the Picard lattice of $s_{1}$, which by assumption satisfies 
\begin{equation}
\label{picardlatticerankassumption}
d_{1} \geq \rk \Pic(X_{\overline{\eta}}) + \dim \im(N) + 1 \geq \rk \Pic(X_{\overline{\eta}}) + 2r + 1 .
\end{equation} 
We write $L_{s_{1}}$ for the portion of the Neron-Severi lattice which is orthogonal to (the image of) $\Pic(X_{\overline{\eta}})$. Note that $\dim L_{s_{1}} \geq 2r + 1$. 

\begin{itemize}
\item[(1)] In the case where $v$ is finite and $v$ is relevant for $(s(s_{1}), \mathbf{G})$, we have by \autoref{Picardrankimagedimlem} below that the image of $L_{s_{1}}$ in $\mathfrak{H}^{v}_{D, s_{1}}$ spans a vector space of dimension at most $r$, or what is equivalent, the kernel of $\alpha^{v}_{s_{1}}$ intersects $L_{s_{1},K_{v}(s_{1})}$ in a subspace of dimension at least $r+1$. Then if we consider $\bigwedge^{2r} L_{s_{1}}$, any vector $w$ in this space is sent to zero under the map $\alpha^{v}_{s_{1}}$. We may fix such a vector $w$ and consider its coefficients $a_{I}$ with respect to the basis $\{ \omega_{I} \}$; these coefficients are defined over a number field $K(s_{1}, w)$. By evaluating $\alpha^{v}_{s_{1}}$ at $w$ we then obtain a relation $\sum_{I} a_{I} \mathbf{G}_{I}(s_{1}) = 0$. 

\item[(2)] In the case where $v$ is infinite and $v$ is relevant for $(s(s_{1}), \mathbf{G})$, we fix two linearly independent vectors $w, w' \in \bigwedge^{2r} L_{s_{1}}$. Then both $w$ and $w'$ induce classes in $\bigwedge^{2r} H^{2}(X^{v}_{s_{1}}, \mathbb{Z}(1))$, so the values $\alpha^{v}_{s_{1}}(w)$ and $\alpha^{v}_{s_{1}}(w')$ are $(2 \pi i)^{2r}$-multiples of integers $M$ and $M'$. It follows that $\alpha^{v}_{s_{1}}(M' w + M w) = 0$. Letting $a_{I}$ denote the coordinates of $M' w + M w$ with respect to the basis $\{ \omega_{I} \}$, we obtain a relation $\sum_{I} a_{I} \mathbf{G}_{I}(s_{1}) = 0$ which is defined over a finite extension $K(s_{1}, w, w')$ of $K(s_{1})$. 
\end{itemize}

\begin{lem}
\label{Picardrankimagedimlem}
The image of $L_{s_{1}}$ in $\mathfrak{H}^{v}_{D, s_{1}}$ spans a vector space of dimension at most $r$.
\end{lem}

\begin{proof}
Using the comparison between overconvergent Hyodo-Kato and de Rham cohomology, we may replace $\mathfrak{H}^{v}_{s_{1}}$ with $R^2 \Gamma_{\textrm{HK}}(X^{v}_{s_{1}})$, $\mathfrak{H}^{v}_{D,s_{1}}$ with $R^2 \Gamma_{\textrm{HK}}(\mathcal{D}^{v}_{s_{1}})$, and regard the lattice $L_{s_{1}}$ inside $R^2 \Gamma_{\textrm{HK}}(X^{v}_{s_{1}})$. The lattice $L_{s_{1}}$ spans a subspace on which the Hyodo-Kato monodromy operator $N$ acts trivially and the Frobenius endomorphism $\varphi$ acts as $p$. Its image in $R^2 \Gamma_{\textrm{HK}}(\mathcal{D}^{v}_{s_{1}})$ is a subspace with the same properties. On the other hand, this image lies by construction inside the image of the Hyodo-Kato cohomology group corresponding to $\mathfrak{F}_{s_{1}}$, i.e., the image of the composition
\[ R^2 \Gamma_{\textrm{HK}}([\Delta^{v,2} \times \overline{D}_{j}^{v}]_{s_{1}}) \xrightarrow{e} R^2 \Gamma_{\textrm{HK}}([\Delta^{v,2} \times \overline{D}_{j}^{v}]_{s_{1}}) \to R^2 \Gamma_{\textrm{HK}}(\mathcal{D}^{v}_{s_{1}}) . \]
By construction such a subspace is, as an $(N, \varphi)$-module, isomorphic of a summand of $H^{1}_{\textrm{HK}}(\overline{D}^{v}_{j})(-1)$ of rank $2r$ corresponding to a factor of the Jacobian of $\overline{D}^{v}_{j}$. As a consequence of $p$-adic weight-monodromy, this weight $2$ subspace of $H^{1}_{\textrm{HK}}(\overline{D}^{v}_{j})(-1)$ has rank at most $r$.
\end{proof}
Now we explain the non-triviality of the relations. Since both cases (1) and (2) produce linear relations on $\alpha^{v}_{s_{1}}$, non-triviality follows from \cite[Prop. 6.8]{zbMATH08109694} assuming the representation of $\pi_{1}(S,s)$ on $\bigwedge^{2r}\, V$ is simple. From our assumption that the Zariski closure of the action of $\pi_{1}(S,s)$ is the orthogonal group $\textrm{Aut}(V,Q)$, this follows from \cite[Thm. 19.2(i)]{zbMATH00051906} assuming that $2r < \dim V/2$. Since in general we have $2r \leq \dim V/2$, it suffices to treat the case $2r = \dim V/2$. In this case, by \cite[Thm. 19.2(ii)]{zbMATH00051906}, the representation of $\pi_{1}(S,s)$ on $\bigwedge^{2r} V$ has exactly two summands $V_{1}$ and $V_{-1}$. By the remark following \cite[Thm. 19.2]{zbMATH00051906}, these summands admit the following description: there is a map 
\[ \tau : \bigwedge^{2r} V \to \bigwedge^{2r} V^{*} \to \bigwedge^{2r} V \]
induced by the two natural bilinear forms on $\bigwedge^{2r} V$ (the one induced by $Q$ and the one induced by exterior power). The map $\tau$ squares to the identity, and the decomposition into $1$ and $(-1)$ eigenspaces gives the desired subspaces $V_{1}$ and $V_{-1}$. Now in order to apply \cite[Prop. 6.8]{zbMATH08109694} to conclude non-triviality, it suffices to check that the vector $w$ (resp. $M' w + M w'$ considered in case (1) (resp. (2)) does not lie in either $V_{1}$ or $V_{-1}$. If it did, then the dual vector $Q(w, -)$ (resp. $Q(M' w + M w', -)$) would agree up to a sign with the dual vector $w \wedge (-)$ (resp. $(M' w + M w') \wedge (-)$). But the cup product-induced pairing is definite on the Hodge lattice, and $w \wedge w = 0$ (resp. $(M' w + M w') \wedge (M' w + M w') = 0$). 

Applying once again \autoref{picardlatticeboundeddeg}, the product of all Galois conjugates of all relations under consideration is once again a polynomial of degree bounded by a uniform polynomial in $[K(s_{1}) : K]$. The proof is then completed via an application of \autoref{daworrcoverlem}, as in the proof of \autoref{mainsurfaceapp1}. 

\subsubsection{Proof of \autoref{truefinitenessforK3s}}

Let $g : X \to S$ be a complex algebraic family of polarized K3 surfaces, with $S$ an irreducible smooth curve, and whose map to the associated moduli space $\mathcal{M}$ is quasi-finite; if this map is constant there is nothing to show. By carrying out reasoning analogous to that which appears in \cite[\S8.1]{zbMATH08109694}, we may reduce to the case where the field of definition $K \subset \mathbb{C}$ of $g$ is a number field. More precisely, we may let $Y$ be the $\overline{\mathbb{Q}}$-Zariski closure of the image $B$ of $S$ in $\mathcal{M}$, and let $\{ \mathcal{M}_{i} \}_{i \in I}$ be the countable collection of all special subvarieties of $\mathcal{M}$ that parameterize K3 surfaces with $5-k$ extra line bundles. Since $S$, and therefore both $B$ and $Y$, contain Hodge generic points, any component of an intersection $Y \cap \mathcal{M}_{i}$ of codimension one in $Y$ is atypical for the natural variation of Hodge structure on $Y$ in the sense of \cite[Def. 2.2]{zbMATH07790853}. If $K$ has positive transcendence degree, then $\dim Y \geq 2$, so all such components are positive dimensional. Then the geometric Zilber-Pink theorem in loc. cit. implies that the union $\bigcup_{i} [Y \cap \mathcal{M}_{i}]^{\operatorname{cd} = 1}$ of all codimension one components lies in a finite union of strict special subvarieties of $Y$; in particular each such component is a special subvariety of $Y$. None of these can contain the image of $B$, since $B$ is Hodge generic in $Y$, so this union intersects $B$ in at most finitely many points. 

For each $\mathcal{M}_{i}$, let $(\mathcal{M}_{i} \cap Y)^{\circ} \subset Y$ be the union of those components of $\mathcal{M}_{i} \cap Y$ of codimension at least two in $Y$. It then suffices to show that $B$ intersects at most finitely many loci $(\mathcal{M}_{i} \cap Y)^{\circ}$. Since the $(\mathcal{M}_{i} \cap Y)^{\circ}$ are all defined over $\overline{\mathbb{Q}}$ and $B$ is defined over a field of positive transcendence degree, this follows from \cite[Prop. 8.2]{zbMATH08109694}. 

We now assume $K$ is a number field. Using the Kuga-Satake correspondence (cf. \cite[\S4, \S5]{zbMATH00931955}) one obtains from $g : X \to S$ a curve $T$ and a family $h : A \to T$ of polarized abelian varieties together with an embedding $\iota : \mathbb{V}_{\mathbb{Q}} \hookrightarrow p^{*} \textrm{End}(R^{1} h_{*} \mathbb{Q})$ of variations of polarized $\mathbb{Q}$-Hodge structures, where $\mathbb{V} = R^{2} g_{*} \mathbb{Z}$ and $p : S \to T$ is a quasi-finite map. We recall the following result of Masser-W\"ustholtz: 
\begin{thm}
\label{MWthm}
There exists constants $a, b > 0$, depending on integers $n$ and $\delta$, such that for every polarized abelian variety $A$ of dimension $n$ and polarization degree $\delta$, we have
\begin{equation}
\label{MWineq}
|\operatorname{Disc}(\operatorname{End}(A))| \leq a \max\{ h(A), [\mathbb{Q}(A) : \mathbb{Q}] \}^{b} 
\end{equation}
where $h(A)$ is the stable (logarithmic) Faltings height of $A$ and $\mathbb{Q}(A)$ denotes the minimal field of definition of the associated polarized abelian variety. 
\end{thm}

\begin{proof}
This follows from the results of \cite{zbMATH06348598} in the manner explained following the statement of \cite[Thm. 6.6]{zbMATH07608391}. 
\end{proof}

In the above the discriminant may be computed with respect to the pairing on $D = \textrm{End}(A)_{\mathbb{Q}}$ given by $(x,y) \mapsto \textrm{Tr}_{D/\mathbb{Q}}(x y)$, with $\textrm{Tr}_{D/\mathbb{Q}}$ the (non-reduced) trace associated to the regular representation of $D$. We denote this pairing by $P : D \times D \to \mathbb{Q}$. We note that for the purposes of computing the magnitude of the discriminant, it does not matter if one includes a twist by the Rosati involution in the definition of $P$; cf. \cite[\S5.4]{zbMATH07608391}. The polarization of $A$ also induces a bilinear form $Q : H^{1}(A,\mathbb{Q}) \otimes H^{1}(A,\mathbb{Q}) \to \mathbb{Q}$, as well as a form on $\textrm{End}(H^{1}(A,\mathbb{Q}))$ via the isomorphism $\textrm{End}(H^{1}(A,\mathbb{Q})) \simeq H^{1}(A, \mathbb{Q}) \otimes H^{1}(A, \mathbb{Q})^{\vee}$ induced by $Q$. Note that $D$ is naturally a subspace of $\textrm{End}(H^{1}(A,\mathbb{Q}))$. Write $B$ for the restriction of $Q$ to $D$.

\begin{lem}
\label{discwithformsthesame}
Up to a uniform constant, the discriminant of $\textrm{End}(A)$, computed using $B$, is bounded by the discriminant of $\textrm{End}(A)$ computed using $P$. 
\end{lem}

\begin{proof}
We first recall from general linear algebra that if we have a non-degenerate bilinear form $Q$ on a vector space $V$, and we write $(-)^{t} : V \xrightarrow{\sim} V^{*}$ for the induced identification between $V$ and its dual, then $Q(v, w) = \textrm{Tr}(v w^{t})$ for all $v, w \in V$. Setting $V = H^{1}(A, \mathbb{Q})$, the form on $V \otimes V^{*} = \textrm{End}(H^{1}(A, \mathbb{Q}))$ is obtained by using $(-)^{t}$ to identify $V \otimes V^{*}$ with $V \otimes V$ and then applying the usual tensor product form. Then given two pure tensors $v w^{t}, u z^{t} \in \textrm{End}(V)$, we have that
\[ Q(v w^{t}, u z^{t}) = Q(v, u) Q(w, z) = Q(w,z) \textrm{Tr}(v u^{t}) = \textrm{Tr}(v w^{t} \circ (u z^{t})^{t}) . \]
It follows by linearity that $Q(a, b) = \textrm{Tr}(a b^{t})$, for all $a, b \in \textrm{End}(V)$. Then $Q(ab, c) = \textrm{Tr}(abc^{t}) = \textrm{Tr}(a(c b^{t})^{t}) = Q(a, cb^{t})$ where $a, b, c \in \textrm{End}(V)$. 

From this it follows that the form $Q(-, (-)^{t})$ on $\textrm{End}(H^{1}(A,\mathbb{Q}))$ is symmetric, hence so is the form $B'(-, -) := B(-, (-)^{t})$. Moreover $F := B'$ satisfies the additional identity $F(xy, z) = F(x, yz)$; a bilinear form $F$ on $D$ with this property is called associative, and $P$ is also associative. Then we claim that $F$ is bounded by $P$ up to a uniform constant.

Given a symmetric associative form $F$ on $D$ we have $F(x,y) = F(1, xy)$ identically, and from this one sees that the space of symmetric associative forms on $D$ is equivalent to the space of trace functionals: linear maps $T : D \to \mathbb{Q}$ such that $T(xy) = T(yx)$ identically. Since $D$ is a semisimple $\mathbb{Q}$-algebra with simple decomposition $D = D_{1} \times \cdots \times D_{k}$, the space of such functionals is the direct sum of the corresponding spaces of such functionals on each $D_{r}$. Using the Albert classification, each $D_{r}$ is a central simple algebra over a field $L$, and hence is isomorphic to a matrix algebra $\mathbb{M}_{n_{r}}(L')$ over some extension $L'$ of $L$. It is well-known that the space of symmetric associative $\mathbb{Q}$-linear functionals on such an algebra is one-dimensional, with each being a scalar multiple of the usual trace $\mathbb{M}_{n_{r}}(L') \to L'$ composed with the field trace. 

It follows that both $B'$ and $P$ are determined, as bilinear forms on $D$, by their values $B(e_{r}, e_{r}) = Q(e_{r}, e_{r})$ and $P(e_{r}, e_{r})$ on the minimal central idempotents $e_{1}, \hdots, e_{k}$ of the algebra $D = D_{1} \times \cdots \times D_{k}$. Since $P$ is defined using the trace associated to the regular representation, we have $P(e_{i}, e_{i}) = \dim_{\mathbb{Q}} D_{i}$. Now note that $e_{i}$ is the unique self-adjoint projector from $H^1(A, \mathbb{Q})$ to its image $W_{i}$. Choose a basis $b_{1}, \hdots, b_{2g}$ for $H^{1}(A, \mathbb{Q})$, compatible with the decomposition $H^{1}(A, \mathbb{Q}) = W \oplus W^{\perp}$, such that $e_{i} = b_{1} b^{*}_{1} + \cdots + b_{2j} b^{*}_{2j}$ with $\dim_{\mathbb{Q}} W = 2j$. Now $B'(e_{i}, e_{i}) = Q(\operatorname{id}_{W}, \operatorname{id}_{W}) = \dim_{\mathbb{Q}} W$. It follows that $B'$ is bounded by $P$ up to a uniform constant. 

Now by \cite[Lem. 5.6]{zbMATH07608391}, the magnitude of the discriminant of $\textrm{End}(A, \mathbb{Q})$ computed with respect to $B$ agrees with the magnitude computed with respect to $B'$, so the result follows. 
\end{proof}
\noindent As a consequence of \autoref{discwithformsthesame}, \autoref{MWthm} also bounds, up to adjusting the constant $a$, the discriminant of $\textrm{End}(A)$ viewed inside $\textrm{End}(H^{1}(A,\mathbb{Z}))$ and regarded as a lattice under $Q$.

Let $\mathbb{L}_{\mathbb{Q}} \subset p^{*} \textrm{End}(R^{1} h_{*} \mathbb{Q})$ be the image of $\iota$. Using the semisimplicity of the category of polarizable pure $\mathbb{Q}$VHS, we can find a complementary summand $\mathbb{W}_{\mathbb{Q}}$ such that $p^{*} \textrm{End}(R^{1} h_{*} \mathbb{Q}) = \mathbb{L}_{\mathbb{Q}} \oplus \mathbb{W}_{\mathbb{Q}}$. Let $e_{\mathbb{L}}$ and $e_{\mathbb{W}}$ be the two corresponding idempotent projectors. Fix integral subsystems $\mathbb{L} \subset \mathbb{L}_{\mathbb{Q}}$ and $\mathbb{W} \subset \mathbb{W}_{\mathbb{Q}}$ such that $\mathbb{L}_{\mathbb{Q}}$ is the $\mathbb{Q}$-scalar extension of $\mathbb{L}$ (resp. $\mathbb{W}_{\mathbb{Q}}$ is the $\mathbb{Q}$-scalar extension of $\mathbb{W}$). Fix integers $M$ and $N$ such that $M e_{\mathbb{L}} (\textrm{End}(R^{1} h_{*} \mathbb{Z})) \subset \mathbb{L}$ (resp. $N e_{\mathbb{W}} (\textrm{End}(R^{1} h_{*} \mathbb{Z})) \subset \mathbb{W}$. For each $s \in S(\mathbb{C})$ we can consider the lattice $M e_{\mathbb{L}}(\textrm{End}(A_{p(s)}))$, whose discriminant under $Q$ is bounded by a polynomial in $h(A_{p(s)})$ and $[\mathbb{Q}(A_{p(s)}) : \mathbb{Q}]$. It follows that the discriminant of the Hodge sublattice of $\mathbb{L}$ is likewise bounded. Fix an integer $R$ such that $\iota(R \mathbb{V}) \subset \mathbb{L} \subset p^{*} \textrm{End}(R^{1} h_{*} \mathbb{Z})$.  Pulling back under $R \iota$, we obtain that:

\begin{cor}
\label{MWcor}
There exists constants $a, b > 0$, depending on $\delta$, such that for every polarized K3 surface $X$ with polarization of degree $\delta$, we have
\begin{equation}
\label{Picineq}
|\operatorname{Disc}(\Pic(X_{s}))| \leq a \max\{ \theta(s), [\mathbb{Q}(X_{s}) : \mathbb{Q}] \}^{b} 
\end{equation}
where $\theta : S(\overline{K}) \to \mathbb{R}_{\geq 0}$ is a logarithmic Weil height and $\mathbb{Q}(X_{s})$ denotes the minimal field of definition of the polarized K3 surface.
\end{cor}

\begin{proof}
Our reasoning above shows that the above corollary holds except with $\theta(s)$ replaced by $h(A_{p(s)})$, and $\mathbb{Q}(X_{s})$ replaced by $\mathbb{Q}(A_{p(s)})$. Here $h(-)$ is, as before, the logarithmic stable Faltings height. Using that $p$ descends to a number field (because the Kuga-Satake correspondence does), we may replace $\mathbb{Q}(A_{p(s)})$ with $\mathbb{Q}(X_{s})$. The height function $h(-)$ is equivalent, up to a multiplicative constant, to any logarithmic Weil height on $T$ \cite[p. 356]{zbMATH03944027} so we obtain the result after pulling such a height back along the quasi-finite map $p$. 
\end{proof}

Combining \autoref{K3cor} with \autoref{MWcor} it follows that the points of the set (\ref{truefinitenessset}), which in this case lie in $S(\overline{K})$, satisfy $\textrm{Disc}(\Pic(X_{s})) = O([K(s) : K]^{b})$ for some $b > 0$. For each $s$ in the set (\ref{truefinitenessset}), choose a maximal linearly independent set $\{ b_{1,s}, \hdots, b_{m_{s},s} \}$ of elements of $\Pic(X_{s})$ such that $Q(b_{i,s}, b_{i,s})$ is polynomially bounded by $\textrm{Disc}(\Pic(X_{s}))$ for each $i$; this can be done using Minkowski's bounds for the successive minima of a lattice. Take two vectors $v_{s}$ and $w_{s}$ from such a basis that are not in the $\mathbb{Q}$-span of $\Pic(X_{\overline{\eta}})$. Then the Hodge tensor $u_{s} := v_{s} \otimes w_{s} - w_{s} \otimes v_{s} \in \bigwedge^2 H^{2}(X_{s}, \mathbb{Z}) \subset \bigotimes^2 H^{2}(X_{s}, \mathbb{Z})$, and we have $Q(u_{s}, u_{s}) = O([K(s) : K]^{b})$ for some $b > 0$.

Observe that $u_{s}$ atypically defines the point $s \in S(\mathbb{C})$ in the sense of \cite[Def. 1.8, Def. 1.9]{urbanik2025complexityatypicalspecialpoints}. Then by \cite[Thm. 1.10]{urbanik2025complexityatypicalspecialpoints}, the number $N_{q}$ of points of (\ref{truefinitenessset}) defined by some such $u$ with $Q(u,u) \leq q$ is $O(q^{\ep})$ for any $\ep > 0$. In particular, it follows that $N_{q} = O(d^{\ep})$, where $d$ is the largest degree over $K$ of a point $s$ in the set (\ref{truefinitenessset}) defined atypically by some $u_{s}$ with $Q(u_{s},u_{s}) \leq q$. 

Suppose now that (\ref{truefinitenessset}) is infinite. Take $\ep = 1/2$ and $q$ (and hence $d$) large. Then $N_{q} \leq c d^{1/2}$ on one hand for  some $c < d^{1/2}$, but on the other hand by taking Galois conjugates we see that $N_{q} \geq d$, giving a contradiction.

\appendix

\section{Hyodo-Kato Cohomology}
\label{hyodokatosec}

In their paper \cite{zbMATH07243787}, Colmez and Nizio\l{} develop a Hyodo-Kato theory for smooth rigid-analytic varieties $X$ over a complete algebraically closed non-archimedean field $C$ of mixed characteristic $(0,p)$. Here $C$ is assumed to arise as the $p$-adic completion of the algebraic closure of a finite extension $L/\mathbb{Q}_{p}$. This theory assigns to each such rigid-analytic $X$ an object $R \Gamma_{\textrm{HK}}(X)$ in the derived category of the category of $(\varphi, N)$ modules over $\breve{F}$, where $F$ is the fraction field of $W(k)$, with $k$ the residue field of $L$. Here $\varphi$ and $N$ are two endomorphisms satisfying $N \varphi = p \varphi N$. In \cite{bosco2023rationalpadichodgetheory}, the above Hyodo-Kato theory is generalized to possibly singular rigid varieties, and a verification that the rigid-analytic theory agrees with the more classical algebraic Hyodo-Kato theory is given in \cite{arXiv:2502.13315}. 

The theory has both a regular and an overconvergent variant; we will typically use the overconvergent one. Aside from standard functorality properties, what we will need is the following.
\begin{itemize}
\item[(1)] The theory admits a comparison with the (overconvergent) de Rham cohomology of the rigid variety $X$ over $C$. In particular, one has a natural functorial isomorphism
\[ R \Gamma_{\textrm{HK}}(X) \otimes C \xrightarrow{\sim} R \Gamma_{\textrm{dR}}(X) \]
where on the right one has the (overconvergent) de Rham cohomology of $X$. Here the tensor product should either be completed and derived in the case when one works ordinary $(\varphi, N)$-modules (as in the work of Colmez-Nizio\l{}), or it should be the solid tensor product if one works with condensed $(\varphi, N)$-modules (as in the work of Bosco). In either case, when the underlying cohomology groups involved are finite-dimensional, one recovers a comparison $H^{i}_{\textrm{HK}}(X) \otimes C \xrightarrow{\sim} H^{i}_{\textrm{dR}}(X)$ where the tensor product is interpreted in the ordinary sense, which is what we will use. See \cite[(5.17)]{zbMATH07243787} and \cite[Thm. 3.14(iii)]{bosco2023rationalpadichodgetheory}. 
\item[(2)] In the non-overconvergent case, the theory admits a comparison with the log crystalline cohomology, i.e., whenever $X$ admits a semistable model $\mathcal{X}$, the Hyodo-Kato cohomology of $X$ agrees with the log-crystalline cohomology of $\mathcal{X}$. More precisely, the formal scheme $\mathcal{X}$ has a canonical log structure (cf. \cite[\S1.6]{zbMATH07109551}), giving rise to a log scheme denoted by $\mathcal{X}^{0}$. Using the site of log-divided power thickenings as in \cite{beilinson2013crystalline} (cf. \cite[\S3.1]{bosco2023rationalpadichodgetheory}), we can define the log crystalline cohomology of $\mathcal{X}^{0}_{\mathcal{O}_{C}/p}$ over $\mathcal{O}^{0}_{\breve{F}}$ by taking derived global sections on the log-crystalline topos. There is then a natural comparison isomorphism (see \cite[Thm. 3.14]{bosco2023rationalpadichodgetheory})
\[ R\Gamma_{\textrm{cris}}(\mathcal{X}^{0}_{\mathcal{O}_{C}/p} / \mathcal{O}^{0}_{\breve{F}})_{\mathbb{Q}_{p}} \xrightarrow{\sim} R \Gamma_{\textrm{HK}}(X) . \]
In the overconvergent case a similar statement holds \cite[Thm. 3.28]{bosco2023rationalpadichodgetheory} but where the log crystalline cohomology is replaced by the log-rigid cohomology, as developed in \cite[\S1]{zbMATH02204422} (cf. \cite[\S3.1.2]{zbMATH07160161}).
\end{itemize}
In practice one computes the Hyodo-Kato cohomology using either the log-crystalline or log rigid cohomology when one has a semi-stable model, and then uses functoriality properties to deduce the correct output for the theory in situations where no such model is available.

\bibliography{hodge_theory}

@article{katz1968,
author = "Katz, Nicholas M. and Oda, Tadao",
doi = "10.1215/kjm/1250524135",
fjournal = "Journal of Mathematics of Kyoto University",
journal = "J. Math. Kyoto Univ.",
number = "2",
pages = "199--213",
publisher = "Duke University Press",
title = "On the differentiation of {D}e {R}ham cohomology classes with respect to parameters",
url = "https://doi.org/10.1215/kjm/1250524135",
volume = "8",
year = "1968"
}

@article{zbMATH03882517,
 author = {Laurent, Michel},
 title = {Exponential diophantine equations},
 fjournal = {Inventiones Mathematicae},
 journal = {Invent. Math.},
 issn = {0020-9910},
 volume = {78},
 pages = {299--327},
 year = {1984},
 language = {French},
 doi = {10.1007/BF01388597},
 url = {https://eudml.org/doc/143175},
 zbMATH = {3882517},
 Zbl = {0554.10009}
}

@Article{zbMATH03899065,
 Author = {Raynaud, M.},
 Title = {Courbes sur une vari{\'e}t{\'e} ab{\'e}lienne et points de torsion},
 FJournal = {Inventiones Mathematicae},
 Journal = {Invent. Math.},
 ISSN = {0020-9910},
 Volume = {71},
 Pages = {207--233},
 Year = {1983},
 Language = {French},
 DOI = {10.1007/BF01393342},
 zbMATH = {3899065},
 Zbl = {0564.14020}
}

@incollection{zbMATH07272616,
 author = {Klingler, B. and Ullmo, E. and Yafaev, A.},
 title = {Bi-algebraic geometry and the {Andr{\'e}}-{Oort} conjecture},
 booktitle = {Algebraic geometry: Salt Lake City 2015. 2015 summer research institute in algebraic geometry, University of Utah, Salt Lake City, UT, USA, July 13--31, 2015. Proceedings. Part 2},
 isbn = {978-1-4704-3578-3; 978-1-4704-4667-3; 978-1-4704-4680-2},
 pages = {319--359},
 year = {2018},
 publisher = {Providence, RI: American Mathematical Society (AMS); Cambridge, MA: Clay Mathematics Institute},
 language = {English},
 doi = {10.1090/pspum/097.2/01709},
 url = {webusers.imj-prg.fr/~bruno.klingler/papiers/Survey2.pdf},
 zbMATH = {7272616},
 Zbl = {1457.11087}
}

@misc{pila2024canonicalheightsshimuravarieties,
      title={Canonical {H}eights on {S}himura {V}arieties and the {A}ndr\'e-{O}ort {C}onjecture}, 
      author={Jonathan Pila and Ananth N. Shankar and Jacob Tsimerman and Hélène Esnault and Michael Groechenig},
      year={2024},
      eprint={2109.08788},
      archivePrefix={arXiv},
      primaryClass={math.NT},
      url={https://arxiv.org/abs/2109.08788}, 
}

@InCollection{zbMATH05013101,
 Author = {Pink, Richard},
 Title = {A combination of the conjectures of {Mordell}-{Lang} and {Andr{\'e}}-{Oort}},
 BookTitle = {Geometric methods in algebra and number theory},
 ISBN = {0-8176-4349-4},
 Pages = {251--282},
 Year = {2005},
 Publisher = {Basel: Birkh{\"a}user},
 Language = {English},
 DOI = {10.1007/0-8176-4417-2_11},
 zbMATH = {5013101},
 Zbl = {1200.11041}
}

@Article{zbMATH01818752,
 Author = {Zilber, Boris},
 Title = {Exponential sums equations and the {Schanuel} conjecture},
 FJournal = {Journal of the London Mathematical Society. Second Series},
 Journal = {J. Lond. Math. Soc., II. Ser.},
 ISSN = {0024-6107},
 Volume = {65},
 Number = {1},
 Pages = {27--44},
 Year = {2002},
 Language = {English},
 DOI = {10.1112/S0024610701002861},
 zbMATH = {1818752},
 Zbl = {1030.11073}
}

@Misc{zbMATH03929174,
 Author = {Raynaud, M.},
 Title = {Sous-vari{\'e}t{\'e}s d'une vari{\'e}t{\'e} ab{\'e}lienne et points de torsion},
 Year = {1983},
 Language = {French},
 HowPublished = {Arithmetic and geometry, {Pap}. dedic. {I}. {R}. {Shafarevich}, {Vol}. {I}: {Arithmetic}, {Prog}. {Math}. 35, 327-352 (1983).},
 zbMATH = {3929174},
 Zbl = {0581.14031}
}

@article{zbMATH03944027,
 author = {Faltings, G.},
 title = {Finiteness theorems for abelian varieties over number fields.},
 fjournal = {Inventiones Mathematicae},
 journal = {Invent. Math.},
 issn = {0020-9910},
 volume = {73},
 pages = {349--366},
 year = {1983},
 language = {German},
 doi = {10.1007/BF01388432},
 url = {https://eudml.org/doc/143051},
 zbMATH = {3944027},
 Zbl = {0588.14026}
}

@article{zbMATH06348598,
 author = {Masser, David and W{\"u}stholz, Gisbert},
 title = {Polarization estimates for abelian varieties},
 fjournal = {Algebra \& Number Theory},
 journal = {Algebra Number Theory},
 issn = {1937-0652},
 volume = {8},
 number = {5},
 pages = {1045--1070},
 year = {2014},
 language = {English},
 doi = {10.2140/ant.2014.8.1045},
 zbMATH = {6348598},
 Zbl = {1311.11051}
}

@book{zbMATH00051906,
 author = {Fulton, William and Harris, Joe},
 title = {Representation theory. {A} first course},
 fseries = {Graduate Texts in Mathematics},
 series = {Grad. Texts Math.},
 issn = {0072-5285},
 volume = {129},
 isbn = {0-387-97495-4; 0-387-97527-6},
 year = {1991},
 publisher = {New York etc.: Springer-Verlag},
 language = {English},
 zbMATH = {51906},
 Zbl = {0744.22001}
}

@article{zbMATH00931955,
 author = {Andr{\'e}, Yves},
 title = {On the {Shafarevich} and {Tate} conjectures for hyperk{\"a}hler varieties},
 fjournal = {Mathematische Annalen},
 journal = {Math. Ann.},
 issn = {0025-5831},
 volume = {305},
 number = {2},
 pages = {205--248},
 year = {1996},
 language = {English},
 doi = {10.1007/BF01444219},
 url = {https://eudml.org/doc/165441},
 zbMATH = {931955},
 Zbl = {0942.14018}
}

@book{zbMATH02043955,
 author = {Fresnel, Jean and van der Put, Marius},
 title = {Rigid analytic geometry and its applications},
 fseries = {Progress in Mathematics},
 series = {Prog. Math.},
 issn = {0743-1643},
 volume = {218},
 isbn = {0-8176-4206-4},
 year = {2004},
 publisher = {Boston, MA: Birkh{\"a}user},
 language = {English},
 zbMATH = {2043955},
 Zbl = {1096.14014}
}

@article{zbMATH02133479,
 author = {Gro{\ss}e-Kl{\"o}nne, Elmar},
 title = {de {Rham} cohomology of rigid spaces},
 fjournal = {Mathematische Zeitschrift},
 journal = {Math. Z.},
 issn = {0025-5874},
 volume = {247},
 number = {2},
 pages = {223--240},
 year = {2004},
 language = {English},
 doi = {10.1007/s00209-003-0544-9},
 zbMATH = {2133479},
 Zbl = {1078.14026}
}

@article{zbMATH02204422,
 author = {Grosse-Kl{\"o}nne, Elmar},
 title = {Frobenius and monodromy operators in rigid analysis, and {Drinfel}'d's symmetric space},
 fjournal = {Journal of Algebraic Geometry},
 journal = {J. Algebr. Geom.},
 issn = {1056-3911},
 volume = {14},
 number = {3},
 pages = {391--437},
 year = {2005},
 language = {English},
 doi = {10.1090/S1056-3911-05-00402-9},
 zbMATH = {2204422},
 Zbl = {1084.14021}
}

@misc{daw2025newcaseszilberpinky13,
      title={Some new cases of {Z}ilber-{P}ink in ${Y}(1)^3$}, 
      author={Christopher Daw and Martin Orr and Georgios Papas},
      year={2025},
      eprint={2510.09603},
      archivePrefix={arXiv},
      primaryClass={math.NT},
      url={https://arxiv.org/abs/2510.09603}, 
}

@article{zbMATH07931113,
 author = {Papas, Georgios},
 title = {Some cases of the {Zilber}-{Pink} conjecture for curves in {{\(\mathcal{A}_g\)}}},
 fjournal = {IMRN. International Mathematics Research Notices},
 journal = {Int. Math. Res. Not.},
 issn = {1073-7928},
 volume = {2024},
 number = {5},
 pages = {4160--4206},
 year = {2024},
 language = {English},
 doi = {10.1093/imrn/rnad201},
 zbMATH = {7931113},
 Zbl = {1562.11108}
}

@article{zbMATH07608391,
 author = {Daw, Christopher and Orr, Martin},
 title = {Quantitative reduction theory and unlikely intersections},
 fjournal = {IMRN. International Mathematics Research Notices},
 journal = {Int. Math. Res. Not.},
 issn = {1073-7928},
 volume = {2022},
 number = {20},
 pages = {16138--16195},
 year = {2022},
 language = {English},
 doi = {10.1093/imrn/rnab173},
 zbMATH = {7608391},
 Zbl = {1516.11065}
}

@article{zbMATH07481643,
 author = {Daw, Christopher and Orr, Martin},
 title = {Unlikely intersections with ${E} \times${CM} curves in {{\(\mathcal{A}_2\)}}},
 fjournal = {Annali della Scuola Normale Superiore di Pisa. Classe di Scienze. Serie V},
 journal = {Ann. Sc. Norm. Super. Pisa, Cl. Sci. (5)},
 issn = {0391-173X},
 volume = {22},
 number = {4},
 pages = {1705--1745},
 year = {2021},
 language = {English},
 doi = {10.2422/2036-2145.202006_014},
 zbMATH = {7481643},
 Zbl = {1483.11122}
}

@misc{papas2023unlikelyintersectionstorellilocusv1,
      title={Unlikely intersections in the {T}orelli locus and the {G}-functions method, version 1}, 
      author={Georgios Papas},
      year={2022},
      eprint={2201.11240},
      archivePrefix={arXiv},
      primaryClass={math.AG},
      url={https://arxiv.org/abs/2201.11240}, 
}

@article{zbMATH00033093,
 author = {Silverberg, A.},
 title = {Fields of definition for homomorphisms of abelian varieties},
 fjournal = {Journal of Pure and Applied Algebra},
 journal = {J. Pure Appl. Algebra},
 issn = {0022-4049},
 volume = {77},
 number = {3},
 pages = {253--262},
 year = {1992},
 language = {English},
 doi = {10.1016/0022-4049(92)90141-2},
 zbMATH = {33093},
 Zbl = {0808.14037}
}

@misc{papas2023unlikelyintersectionstorellilocus,
      title={Unlikely intersections in the {T}orelli locus and the {G}-functions method}, 
      author={Georgios Papas},
      year={2023},
      eprint={2201.11240},
      archivePrefix={arXiv},
      primaryClass={math.AG},
      url={https://arxiv.org/abs/2201.11240}, 
}

@article{beilinson2013crystalline,
  title={On the crystalline period map},
  author={Beilinson, Alexander},
  journal={Cambridge Journal of Mathematics},
  volume={1},
  number={1},
  pages={1--51},
  year={2013},
  publisher={International Press of Boston, Inc. Somerville, MA 02143, USA}
}

@book{kedlayabook,
 author = {Kedlaya, Kiran},
 title = {Weil cohomology in practice},
 year = {2025},
 publisher = {Self-published},
 language = {English}
}

@article{zbMATH08109702,
 author = {Daw, Christopher and Orr, Martin},
 title = {Zilber-{P}ink in a product of modular curves assuming multiplicative degeneration},
 fjournal = {Duke Mathematical Journal},
 journal = {Duke Math. J.},
 issn = {0012-7094},
 volume = {174},
 number = {13},
 pages = {2877--2926},
 year = {2025},
 language = {English},
 doi = {10.1215/00127094-2025-0011},
 zbMATH = {8109702}
}

@article{zbMATH07286305,
 author = {Schreieder, Stefan and Soldatenkov, Andrey},
 title = {The {Kuga}-{Satake} construction under degeneration},
 fjournal = {Journal of the Institute of Mathematics of Jussieu},
 journal = {J. Inst. Math. Jussieu},
 issn = {1474-7480},
 volume = {19},
 number = {6},
 pages = {2165--2182},
 year = {2020},
 language = {English},
 doi = {10.1017/S1474748019000239},
 zbMATH = {7286305},
 Zbl = {1455.14018}
}

@book{zbMATH03425769,
 author = {Kempf, G. and Knudsen, F. and Mumford, D. and Saint-Donat, Bernard},
 title = {Toroidal embeddings. {I}},
 fseries = {Lecture Notes in Mathematics},
 series = {Lect. Notes Math.},
 issn = {0075-8434},
 volume = {339},
 year = {1973},
 publisher = {Springer, Cham},
 language = {English},
 doi = {10.1007/BFb0070318},
 zbMATH = {3425769},
 Zbl = {0271.14017}
}

@misc{arXiv:2502.13315,
 author = {Shao, Xinyu},
 title = {Hyodo-{K}ato cohomology in rigid geometry: some foundational results},
 year = {2025},
 howpublished = {Preprint, {arXiv}:2502.13315 [math.{AG}] (2025)},
 url = {https://arxiv.org/abs/2502.13315},
 arXiv = {arXiv:2502.13315}
}

@article{zbMATH07243787,
 author = {Colmez, Pierre and Nizio{\l}, Wies{\l}awa},
 title = {On {{\(p\)}}-adic comparison theorems for rigid analytic varieties: {I}},
 fjournal = {M{\"u}nster Journal of Mathematics},
 journal = {M{\"u}nster J. Math.},
 issn = {1867-5778},
 volume = {13},
 number = {2},
 pages = {445--507},
 year = {2020},
 language = {English},
 doi = {10.17879/90169648511},
 zbMATH = {7243787},
 Zbl = {1467.14064}
}

@article{zbMATH07160161,
 author = {Colmez, Pierre and Dospinescu, Gabriel and Nizio{\l}, Wies{\l}awa},
 title = {Cohomology of {{\(p\)}}-adic {Stein} spaces},
 fjournal = {Inventiones Mathematicae},
 journal = {Invent. Math.},
 issn = {0020-9910},
 volume = {219},
 number = {3},
 pages = {873--985},
 year = {2020},
 language = {English},
 doi = {10.1007/s00222-019-00923-z},
 zbMATH = {7160161},
 Zbl = {1475.11110}
}

@article{zbMATH03187700,
 author = {Buchsbaum, D. A.},
 title = {Some remarks on factorization in power series rings},
 fjournal = {Journal of Mathematics and Mechanics},
 journal = {J. Math. Mech.},
 issn = {0095-9057},
 volume = {10},
 pages = {749--753},
 year = {1961},
 language = {English},
 zbMATH = {3187700},
 Zbl = {0115.26203}
}

@Book{zbMATH00041964,
 Author = {Andr{\'e}, Yves},
 Title = {G-functions and geometry},
 FSeries = {Aspects of Mathematics},
 Series = {Aspects Math.},
 ISSN = {0179-2156},
 Volume = {E13},
 ISBN = {3-528-06317-3},
 Year = {1989},
 Publisher = {Wiesbaden etc.: Friedr. Vieweg \&| Sohn},
 Language = {English},
 zbMATH = {41964},
 Zbl = {0688.10032}
}

@misc{daw2025largegaloisorbitsconjecture,
      title={The large {G}alois orbits conjecture under multiplicative degeneration}, 
      author={Christopher Daw and Martin Orr},
      year={2025},
      eprint={2306.13463},
      archivePrefix={arXiv},
      primaryClass={math.NT},
      url={https://arxiv.org/abs/2306.13463}, 
}

@article{zbMATH00870188,
 author = {Constantine, G. M. and Savits, T. H.},
 title = {A multivariate {Faa} di {Bruno} formula with applications},
 fjournal = {Transactions of the American Mathematical Society},
 journal = {Trans. Am. Math. Soc.},
 issn = {0002-9947},
 volume = {348},
 number = {2},
 pages = {503--520},
 year = {1996},
 language = {English},
 doi = {10.1090/S0002-9947-96-01501-2},
 zbMATH = {870188},
 Zbl = {0846.05003}
}

@book{zbMATH03857279,
 author = {Bosch, Siegfried and G{\"u}ntzer, Ulrich and Remmert, Reinhold},
 title = {Non-{Archimedean} analysis. {A} systematic approach to rigid analytic geometry},
 fseries = {Grundlehren der Mathematischen Wissenschaften},
 series = {Grundlehren Math. Wiss.},
 issn = {0072-7830},
 volume = {261},
 year = {1984},
 publisher = {Springer, Cham},
 language = {English},
 zbMATH = {3857279},
 Zbl = {0539.14017}
}

@misc{bosco2023rationalpadichodgetheory,
      title={Rational $p$-adic Hodge theory for rigid-analytic varieties}, 
      author={Guido Bosco},
      year={2023},
      eprint={2306.06100},
      archivePrefix={arXiv},
      primaryClass={math.AG},
      url={https://arxiv.org/abs/2306.06100}, 
}

@Article{zbMATH07109551,
 Author = {{\v{C}}esnavi{\v{c}}ius, K{e}stutis and Koshikawa, Teruhisa},
 Title = {The {{\(A_{\text{inf}}\)}}-cohomology in the semistable case},
 FJournal = {Compositio Mathematica},
 Journal = {Compos. Math.},
 ISSN = {0010-437X},
 Volume = {155},
 Number = {11},
 Pages = {2039--2128},
 Year = {2019},
 Language = {English},
 DOI = {10.1112/S0010437X1800790X},
 zbMATH = {7109551},
 Zbl = {1451.14077}
}

@article{katznil,
     author = {Katz, Nicholas},
     title = {Nilpotent connections and the monodromy theorem : applications of a result of {Turrittin}},
     journal = {Publications Math\'ematiques de l'IH\'ES},
     pages = {175--232},
     publisher = {Institut des Hautes \'Etudes Scientifiques},
     volume = {39},
     year = {1970},
     zbl = {0221.14007},
     mrnumber = {45 #271},
     language = {en},
     url = {http://www.numdam.org/item/PMIHES_1970__39__175_0/}
}

@article{zbMATH07790853,
 author = {Baldi, Gregorio and Klingler, Bruno and Ullmo, Emmanuel},
 title = {On the distribution of the {Hodge} locus},
 fjournal = {Inventiones Mathematicae},
 journal = {Invent. Math.},
 issn = {0020-9910},
 volume = {235},
 number = {2},
 pages = {441--487},
 year = {2024},
 language = {English},
 doi = {10.1007/s00222-023-01226-0},
 zbMATH = {7790853},
 Zbl = {1547.14008}
}

@misc{arXiv:2406.16628,
 author = {Gregorio Baldi and David Urbanik},
 title = {Effective atypical intersections and applications to orbit closures},
 year = {2025},
 howpublished = {Preprint, {arXiv}:2406.16628 [math.{AG}] (2025)},
 url = {https://arxiv.org/abs/2406.16628},
 arXiv = {arXiv:2406.16628}
}

@misc{urbanik2025complexityatypicalspecialpoints,
      title={On the {C}omplexity of {A}typical {S}pecial {P}oints}, 
      author={David Urbanik},
      year={2025},
      eprint={2512.04491},
      archivePrefix={arXiv},
      primaryClass={math.AG},
      url={https://arxiv.org/abs/2512.04491}, 
}

@article{zbMATH08109694,
 author = {Urbanik, David},
 title = {Geometric {{\(G\)}}-functions and atypicality},
 fjournal = {Duke Mathematical Journal},
 journal = {Duke Math. J.},
 issn = {0012-7094},
 volume = {174},
 number = {12},
 pages = {2425--2512},
 year = {2025},
 language = {English},
 doi = {10.1215/00127094-2024-0079},
 zbMATH = {8109694}
}

@Book{zbMATH06255263,
 Author = {Bosch, Siegfried},
 Title = {Lectures on formal and rigid geometry},
 FSeries = {Lecture Notes in Mathematics},
 Series = {Lect. Notes Math.},
 ISSN = {0075-8434},
 Volume = {2105},
 ISBN = {978-3-319-04416-3; 978-3-319-04417-0},
 Year = {2014},
 Publisher = {Cham: Springer},
 Language = {English},
 DOI = {10.1007/978-3-319-04417-0},
 zbMATH = {6255263},
 Zbl = {1314.14002}
}

@misc{stacks-project,
  author       = {The {Stacks project authors}},
  title        = {The Stacks project},
  howpublished = {\url{https://stacks.math.columbia.edu}},
  year         = {2020},
}

@ARTICLE{atyp,
       author = {{Klingler}, Bruno},
        title = "{Hodge loci and atypical intersections: conjectures}",
      journal = {arXiv e-prints},
         year = 2017,
        month = nov,
          eid = {arXiv:1711.09387},
        pages = {arXiv:1711.09387},
archivePrefix = {arXiv},
       eprint = {1711.09387},
 primaryClass = {math.AG},
       adsurl = {https://ui.adsabs.harvard.edu/abs/2017arXiv171109387K}
}

@book{zbMATH05233837,
 author = {Peters, Chris A. M. and Steenbrink, Joseph H. M.},
 title = {Mixed {Hodge} structures},
 fseries = {Ergebnisse der Mathematik und ihrer Grenzgebiete. 3. Folge},
 series = {Ergeb. Math. Grenzgeb., 3. Folge},
 issn = {0071-1136},
 volume = {52},
 isbn = {978-3-540-77015-2},
 year = {2008},
 publisher = {Berlin: Springer},
 language = {English},
 doi = {10.1007/978-3-540-77017-6},
 zbMATH = {5233837},
 Zbl = {1138.14002}
}

@misc{zbMATH03515611,
 author = {Steenbrink, J. H. M.},
 title = {Limits of {Hodge} structures},
 year = {1973},
 language = {English},
 howpublished = {Report 73-04. {Amsterdam}, {The} {Netherlands}: {Department} of {Mathematics}, {University} of {Amsterdam}. 12 p. (1973).},
 zbMATH = {3515611},
 Zbl = {0329.14006}
}

@book{zbMATH00217454,
 author = {Becker, Thomas and Weispfenning, Volker},
 title = {Gr{\"o}bner bases: a computational approach to commutative algebra. {In} cooperation with {Heinz} {Kredel}},
 fseries = {Graduate Texts in Mathematics},
 series = {Grad. Texts Math.},
 issn = {0072-5285},
 volume = {141},
 isbn = {0-387-97971-9},
 year = {1993},
 publisher = {New York: Springer-Verlag},
 language = {English},
 zbMATH = {217454},
 Zbl = {0772.13010}
}

@article{Andre1992,
author = {Andr\'e, Yves},
journal = {Compositio Mathematica},
language = {eng},
number = {1},
pages = {1-24},
publisher = {Kluwer Academic Publishers},
title = {Mumford-{T}ate groups of mixed {H}odge structures and the theorem of the fixed part},
url = {http://eudml.org/doc/90145},
volume = {82},
year = {1992},
}
\bibliographystyle{alpha}

\end{document}